\documentclass[a4paper,10pt,sort&compress]{elsarticle}

\newif\ifusetikzexternalize
\usetikzexternalizetrue

\usepackage{amsmath,amsthm,amssymb,scalerel}
\usepackage{bm}
\usepackage{mathtools}

\usepackage{thmtools, thm-restate}

\newcommand\mtext[1]{\text{#1}}
\newcommand{\+}[1]{\boldsymbol{\ensuremath{\mathbf{#1}}}}
\newcommand\timpact{\varsigma}

\newcommand\symmdiff{\triangle}

\newcommand\roundpar[1]{\left( #1 \right)}
\newcommand\squarepar[1]{\left[ #1 \right]}
\newcommand\curlypar[1]{\left\lbrace #1 \right\rbrace}

\newcommand\oneDIntegral[4]{\int_{#1}^{#2} #3 \hspace{1ex} d#4}
\newcommand\integral[3]{\int_{#1} #2 \hspace{1ex} d#3}

\newcommand\defeq{\mathrel{\mathop:}=}

\newcommand\half{\frac{1}{2}}

\usepackage{stackengine}
\DeclareFontFamily{U}{mathx}{\hyphenchar\font45} 
\DeclareFontShape{U}{mathx}{m}{n}{
      <5> <6> <7> <8> <9> <10>
      <10.95> <12> <14.4> <17.28> <20.74> <24.88>
      mathx10
      }{}
\DeclareSymbolFont{mathx}{U}{mathx}{m}{n}
\DeclareMathAccent{\widecheck}    {0}{mathx}{"71}
\newcommand\approximate[1]{\widecheck{#1}}

\def\onefluid{one-velocity }
\def\twofluid{two-velocity }
\def\Onefluid{One-velocity }
\def\Twofluid{Two-velocity } 

\def\stagger{\widetilde}
\def\volfrac{\alpha}
\def\volfracstag{\stagger\alpha}
\def\orientation{o}

\def\normalstag{\stagger{\+\eta}}
\def\tangentstag{\stagger{\+\tau}}

\def\tvar{T}

\def\level{L}
\def\basemesh{h_0}

\def\onevelo{\bar{u}}

\def\volumeflux{\upsilon}
\def\volumefluxOne{\bar\volumeflux}

\def\volumefluxTwo{\widehat\volumeflux}
\def\massflux{m}
\def\massfluxOne{\bar\massflux}

\def\massfluxTwo{\widehat\massflux}
\def\dt{\delta t}
\def\dx{h}

\newcommand\amr[1]{(#1)}

\newcommand\jumpOper{\mathfrak{J}}

\def\gfmthing{{\Xi}}

\def\gravpot{G}
\def\gravpotc{G_c}

\newcommand\advectionstag[2]{\stagger{\mathfrak{A}}[#1]#2}
\newcommand\divh{\mathfrak{D}}
\newcommand\gradh{\mathfrak{S}}

\newcommand\gfmgrad{\mathfrak{g}}
\def\symmgrad{\stagger\gradh^S}

\newcommand\gradient{\nabla}
\newcommand\divergence{\gradient \cdot}
\newcommand\curl{\gradient \times}

\newcommand\abs[1]{\left\lvert #1 \right\rvert}
\newcommand\norm[1]{\lVert #1 \rVert}

\newcommand\waveamplitude{\delta}

\newcommand\laplacenr{\text{La}}
\newcommand\galilei{\text{Ga}}

\newcommand\reynolds{\text{Re}}

\newcommand\weber{\text{We}}
\newcommand\bond{\text{Bo}}

\newcommand\ratio[1]{\mathcal{R}_{#1}}
\newcommand\ratiofull[1]{{#1^g}/{#1^l}}

\newcommand\mathcfl{\nu}

\newcommand\posflux[1]{\max\roundpar{0, #1}}

\newcommand\wycflval{\nu^\dagger}

\usepackage{stmaryrd}

\newcommand\jump[1]{\left\llbracket#1\right\rrbracket}

\newcommand\sgrad[1]{\gfmgrad^{#1}(p,\xi)}
\newcommand\sgradf[1]{\gfmgrad^{#1}_f(p,\xi)}

\newcommand\sgradone{\bar{\gfmgrad}(p)}

\newcommand\stackrelwidth[2]{\stackrel{\mathmakebox[\widthof{#2}]{#1}}{#2}}

\newcommand\setextrusion[1]{{\omega(\mathcal{#1})}}

\def\apertnormal{\+\eta^a}

\newcommand\mean[1]{\dgal*{#1}}
\usepackage{xparse}
\NewDocumentCommand{\dgal}{sO{}m}{%
  \IfBooleanTF{#1}
    {\dgalext{#3}}
    {\dgalx[#2]{#3}}%
}

\NewDocumentCommand{\dgalext}{m}{%
  \sbox0{%
    \mathsurround=0pt 
    $\left\{\vphantom{#1}\right.\kern-\nulldelimiterspace$%
  }%
  \sbox2{\{}%
  \ifdim\ht0=\ht2
    \{\kern-.625\wd2 \{#1\}\kern-.625\wd2 \}%
  \else
    \left\{\kern-.7\wd0\left\{#1\right\}\kern-.7\wd0\right\}%
  \fi
}

\NewDocumentCommand{\dgalx}{om}{%
  \sbox0{\mathsurround=0pt$#1\{$}%
  \sbox2{\{}%
  \ifdim\ht0=\ht2
    \{\kern-.625\wd2 \{#2\}\kern-.625\wd2 \}%
  \else
    \mathopen{#1\{\kern-.7\wd0 #1\{}
    #2
    \mathclose{#1\}\kern-.7\wd0 #1\}}
  \fi
}

\DeclarePairedDelimiter{\ceil}{\lceil}{\rceil}

\providecommand{\symmdiff}{\mathbin{\mathpalette\xdotminus\relax}}

\usepackage{geometry}
\usepackage{graphicx}

\usepackage{tikz}
\usetikzlibrary{
  arrows.meta,
  positioning}

\ifusetikzexternalize
  \usetikzlibrary{
    arrows.meta,
    external}
  \tikzset{external/system call={pdflatex \tikzexternalcheckshellescape -halt-on-error
    -interaction=batchmode -jobname "\image" "\texsource" && 
    pdfseparate -f 1 -l 1 "\image".pdf "\image"_tmp.pdf &&
    mv "\image"_tmp.pdf "\image".pdf}}
  \newcommand\inputtikzorpdf[1]{
    \tikzsetnextfilename{#1}
    \input{./tikz/#1}
  }
\else

  \newcommand\inputtikzorpdf[1]{\includegraphics{./pdf/#1}}
  \newcommand\tikzsetnextfilename[1]{}

\fi

\usepackage[utf8]{inputenc}
\usepackage{pgfplots}
\usepackage{pgfgantt}
\usepackage{pdflscape}
\pgfplotsset{compat=newest}
\pgfplotsset{plot coordinates/math parser=false}
\pgfplotsset{
    legend image with text/.style={
        legend image code/.code={%
            \node[anchor=center] at (0.3cm,0cm) {#1};
        }
    },
}
\usetikzlibrary{
    matrix,
}
\pgfplotsset{
    compat=1.3,
}

\usepgfplotslibrary{fillbetween}

\usepackage{subcaption}

\def\twofigwidth{0.475\textwidth}
\def\threefigwidth{0.3\textwidth}

\usepackage[numbers]{natbib}
\newcommand{\ignore}[1]{}
\newcommand{\nobibentry}[1]{{\let\nocite\ignore\bibentry{#1}}}

\usepackage[hidelinks]{hyperref}
\usepackage{cleveref}
\usepackage{autonum} 
\newcommand\caref[1]{eq.~\eqref{#1}} 
\usepackage{import}
\usepackage{stackengine}

\usepackage{booktabs}
\usepackage{array}
\usepackage{varwidth} 
\usepackage{multirow}

\newtheorem{timestep}{Time integration step}
\newenvironment{manualtimestep}[1]{%
  \renewcommand\thetimestep{#1}%
  \timestep
}{\endtimestep}

\newcounter{timeintegration}
\newcommand\compareFormulations[4]{
  \stepcounter{timeintegration}
  \begin{samepage}
  \begin{manualtimestep}{T\arabic{timeintegration}}[#1]\label{step:#2}\ \\
    \nopagebreak
    \begin{minipage}{\twofigwidth}
      \begin{equation}
        #4\tag{T\arabic{timeintegration}.I}\label{eqn:#2:one}
      \end{equation}
    \end{minipage}
    \hfill
    \begin{minipage}{\twofigwidth}
      \begin{equation}
        #3\tag{T\arabic{timeintegration}.II}\label{eqn:#2:two}
      \end{equation}
    \end{minipage}
  \end{manualtimestep}
  \end{samepage}
  \vspace*{\floatsep}
}

\usepackage{enumitem}
\newtheorem{challenge}{Challenges}
\crefname{challenge}{challenge}{challenges}
\newlist{chalenum}{enumerate}{1} 
\setlist[chalenum]{label=\arabic*), ref=\thechallenge.\arabic*}
\crefalias{chalenumi}{challenge}
\crefformat{appendix}{#2#1#3}

\begin{document}
  \begin{frontmatter}

    \author{
      Ronald A. Remmerswaal
    }
    \author{
      Arthur E.P. Veldman
    }

    \address{
      Bernoulli Institute, University of Groningen\\
      PO Box 407, 9700 AK Groningen, The Netherlands
    }

    \begin{abstract}
      The numerical modelling of convection dominated high density ratio two-phase flow poses several challenges, amongst which is resolving the relatively thin shear layer at the interface.
To this end we propose a sharp discretisation of the \twofluid model of the two-phase Navier--Stokes equations. 
This results in the ability to model the shear layer, rather than resolving it, by allowing for a velocity discontinuity in the direction(s) tangential to the interface.

In a previous paper (Remmerswaal and Veldman (2022), \href{https://arxiv.org/abs/2209.14934}{arXiv:2209.14934}) we have discussed the transport of mass and momentum, where the two fluids were not yet coupled.
In this paper an implicit coupling of the two fluids is proposed, which imposes continuity of the velocity field in the interface normal direction.
The coupling is included in the pressure Poisson problem, and is discretised using a multidimensional generalisation of the ghost fluid method.
Furthermore, a discretisation of the diffusive forces is proposed, which leads to recovering the continuous \onefluid solution as the interface shear layer is resolved.

The proposed \twofluid formulation is validated and compared to our \onefluid formulation, where we consider a multitude of two-phase flow problems.
It is demonstrated that the proposed \twofluid model is able to consistently, and sharply, approximate solutions to the inviscid Euler equations, where velocity discontinuities appear analytically as well.
Furthermore, the proposed \twofluid model is shown to accurately model the interface shear layer in viscous problems, and it is successfully applied to the simulation of breaking waves where the model was used to sharply capture free surface instabilities. 
    \end{abstract}

    \title{A sharp, structure preserving \twofluid model for two-phase flow}

  \end{frontmatter}

  \section{Introduction}\label{sec:intro}
The numerical simulation of two-phase flow is of great interest to engineering problems (e.g. liquid sloshing), as well as to fundamental research into fluid flow (e.g. to study the physics of liquid impacts).
One of the main challenges in performing such simulations is the computational effort required to resolve the shear layer (SL) present at the phase interface.
The shear layer is a thin (relative to other length scales of interest, we will soon justify this by comparing such length scales) region at the interface where viscous effects play a role, and is for example present on the wave crest of a breaking wave, as illustrated by the green region in~\cref{fig:intro:length_scales}.

The accurate prediction of impact pressures resulting from breaking wave impacts is challenging, in particular because it was found~\citep{Lafeber2012a} that the fragmentation of the interface prior to impact results in a high variability of the local interface shape, and thus a high variability of the resulting impact pressures.
This fragmentation is due to the development of free surface instabilities, such as the Kelvin--Helmholtz (KH) instability, resulting from the shearing gas flow $U_\tau$ past the wave crest, as illustrated by the large red arrow in~\cref{fig:intro:length_scales}.
In order to numerically capture such free surface instabilities, a sharp and accurate model for two-phase flow is required, and we must moreover have a sufficiently fine mesh to resolve the most unstable KH wavelength, which will be denoted by $\lambda_\mtext{KH}$.
To obtain a relative impression of the length scales mentioned, we will now compare the most unstable KH wavelength with an estimate of the shear layer thickness $\delta_\mtext{SL}$. 
Doing so we are able to judge whether or not additional mesh refinement is required if we want to resolve the shear layer as well.

Using the single-phase boundary layer theory by~\citet{Blasius1907} we find that the $.99$ thickness of the shear layer in the gas phase can be estimated by
\begin{equation}
  \delta_\mtext{SL}(\tau) = 4.91 \sqrt{\frac{\mu^g \tau}{\rho^g U_\tau}},
\end{equation}
where $\tau$ is the coordinate along the length of the shear layer which starts at $\tau = 0$, $\rho^\pi$ denotes the density of the $\pi$-phase (for $\pi \in \{l, g\}$) and $\mu^\pi$ denotes the corresponding dynamic viscosity.
Here $U_\tau$ denotes the magnitude of the free stream velocity, see also~\cref{fig:intro:length_scales}.
\begin{figure}
  \centering
  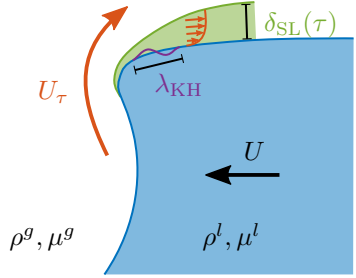
  \caption{Illustration of a breaking wave moving towards the left.
  The shear layer thickness $\delta_\mtext{SL}(\tau)$ and most unstable wavelength $\lambda_\mtext{KH}$ are indicated.
  Note that for illustrational purposes the shear layer is shown much thicker, relative to the most unstable wavelength, than it would usually be.}
  \label{fig:intro:length_scales}
\end{figure}

According to linearised potential flow theory (see e.g.~\citet[Art. 267 \& 268]{Lamb1932}) the most unstable KH wavelength is given by
\begin{equation}\label{eqn:intro:wavelength}
  \lambda_\mtext{KH} = \frac{3\pi\sigma (\rho^l + \rho^g)}{\rho^g\rho^l U_\tau^2},
\end{equation}
where $\sigma$ denotes the surface tension coefficient.
It follows that the ratio of viscous and inertial length scales is given by
\begin{equation}
  \frac{\delta_\mtext{SL}(\lambda_\mtext{KH})}{\lambda_\mtext{KH}} = 4.91 \sqrt{\frac{\mu^g}{\rho^g U_\tau\lambda_\mtext{KH}}} = 4.91 \sqrt{\frac{\rho^l \mu^g}{3 \pi \sigma (\rho^g + \rho^l)}} \sqrt{U_\tau} \approx \frac{\sqrt{U_\tau}}{40},
\end{equation}
where we have substituted the fluid properties of water and air.
In particular, if we consider velocities up to $30\%$ of the speed of sound in air, in order to justify our assumption of incompressibility, we find that
\begin{equation}\label{eqn:intro:lengthratio}
  \frac{\delta_\mtext{SL}(\lambda_\mtext{KH})}{\lambda_\mtext{KH}} \approx \frac{\sqrt{0.3 \times 330}}{40} \approx \frac{1}{4}.
\end{equation}
Hence we find, according to this admittedly simplified analysis, that for incompressible flow the shear layer is at least a factor four thinner than the most unstable wavelength, and therefore resolving the shear layer always yields a much more stringent resolution restriction than one would obtain when wanting to resolve only the most unstable wavelength $\lambda_\mtext{KH}$.

The aforementioned analysis motivates the use of a \twofluid model, wherein each phase is described by its own momentum equation, resulting in two distinct velocity fields which will be denoted by $\+u^l, \+u^g$.
If the shear layer is not resolved, then numerically the velocity component(s) tangential to the interface can be viewed as discontinuous. 
Therefore a \twofluid model can be utilised to permit such a velocity discontinuity, thereby modelling the shear layer rather than resolving it.
We impose that only the interface normal component of the velocity field is continuous (hence assuming that no phase change can occur)
\begin{equation}\label{eqn:intro:normal_smoothness}
  \jump{u_\eta} = 0,
\end{equation}
thereby allowing for a discontinuity in the tangential component(s) of the velocity field.
Here the jump in a quantity is defined as the difference between the gas and liquid value
\begin{equation}
  \jump{\varphi} \defeq \varphi^g - \varphi^l, \quad \mean{\varphi} \defeq \varphi^g + \varphi^l,
\end{equation}
where we have introduced a complimentary notation for the sum over the phases.
In \onefluid models continuity of the full velocity field is imposed instead
\begin{equation}\label{eqn:intro:full_smoothness}
  \jump{\+u} = \+0.
\end{equation}

Several multi-equation models for two-phase flow can be found in the literature, resulting in at most two mass conservation equations, two momentum conservation equations, two energy equations and an equation that governs the advection of the interface.
In~\citet{Saurel1999} a hyperbolic seven equation model (2-2-2-1) is considered for the simulation of two-phase compressible flow.
The energy equations are omitted in~\citet{Preisig2010}, who therefore consider a five equation model (2-2-0-1), and in~\citet{Favrie2014} a six equation model (2-1-2-1) is considered where a momentum equation for the mixture is considered.

For incompressible two-phase flow not much can be found in the literature regarding multi-equation models (in particular two-velocity models).
Note that for incompressible flow the mass conservation equations are implied by the equation that governs the advection of the interface, together with a suitable incompressibility constraint (which we will count as a mass conservation equation).
In~\citet{Desjardins2010} and \citet{Vukcevic2017} two separate equations are used for the transport of momentum, but after each time step the velocity jump is removed in the pressure Poisson problem by imposing~\cref{eqn:intro:full_smoothness}. 
In our terminology this constitutes as a \onefluid model (1-1-0-1).
Our proposed \twofluid model is a sharp four equation model (1-2-0-1) wherein the velocity is permitted to be discontinuous according to~\cref{eqn:intro:normal_smoothness}.

In a previous paper~\citep{pubs_Remmerswaal2022a} we have posed the following three challenges regarding the development of a sharp two-velocity model for two-phase flow.
\begin{challenge}[\Twofluid formulation]\leavevmode
  In the presence of a velocity discontinuity, how to\ldots
  \begin{chalenum}
    \item transport mass and momentum\label{chal:mass_mom}
    \item impose continuity of the interface normal component of velocity\label{chal:continuity}
    \item ensure that the viscous \onefluid solution is obtained under mesh refinement\label{chal:refinement}
  \end{chalenum}
\end{challenge}
\Cref{chal:mass_mom} was considered previously~\citep{pubs_Remmerswaal2022a}, and in the remainder of this paper we will discuss~\cref{chal:continuity,chal:refinement}.

\subsection*{Overview}
We start by giving an overview of our proposed numerical models for the one- and \twofluid formulation in~\cref{sec:overview}.
Then we turn to the Poisson problem in~\cref{sec:poisson} where we propose a generalisation of the ghost fluid method in order to consistently impose~\cref{eqn:intro:normal_smoothness}, and thereby addressing~\cref{chal:continuity}.
In~\cref{sec:diffusion} we propose a simple diffusion model for the \twofluid formulation which will ensure that the velocity discontinuity vanishes as the shear layer is resolved, hence addressing~\cref{chal:refinement}.
A new well-balanced and mimetic gravity model is introduced in~\cref{sec:gravity}.
Numerical results are shown in~\cref{sec:results} and concluding remarks are given in~\cref{sec:conclusion}.
  \section{Overview of the (one- and) \twofluid model}\label{sec:overview}
The notation is based on the notation presented in~\citet{Lipnikov2014}, and is identical to the notation used in~\citet{pubs_Remmerswaal2022a}.
We consider a staggered rectilinear mesh (with block-based adaptive refinement~\citep{VanderPlas2017}) for which we denote the set of centred control volumes by $\mathcal{C}$, where each such control volume $c\in\mathcal{C}$ has faces which are denoted by $f \in \mathcal{F}(c)$, see also~\cref{fig:notation:notation:mesh_c}.
For each face $f$ we denote the corresponding staggered control volume by $\omega_f$, which is illustrated in~\cref{fig:notation:notation:mesh_f}, and is defined as the extrusion of the face $f$ to its neighbouring control volume centroids.
The control volumes neighbouring a face $f$ are denoted by $\mathcal{C}(f) \subset \mathcal{C}$, see also~\cref{fig:notation:notation:mesh_f}.
Scalar functions (e.g. the volume fraction and pressure) are defined at the centroid of the centred control volume, and the set of such scalar functions is denoted by $\mathcal{C}^h$, whereas vector valued functions (e.g. the velocity) are defined on the centroid of a face, and the set of such functions is denoted by $\mathcal{F}^h$.
The function $u \in \mathcal{F}^h$ approximates the face normal (denoted by $\+n_f$) component of the vector field $\+u$ at the face centroid $\+x_f$: $u^{(n)}_f \approx \+n_f \cdot \+u(t^{(n)}, \+x_f)$.


We will now describe the time integration steps used for both our one- and \twofluid model, that is, how to advance the liquid volume fraction $\volfrac^{l,(n)} \in \mathcal{C}^h$, velocities $u^{l,(n)}, u^{g,(n)} \in \mathcal{F}^h$ and pressure $p^{(n)} \in \mathcal{C}^h$ to the next time level $t = t^{(n+1)}$?
Here the centred volume fraction of the $\pi$-phase (for $\pi \in \{l, g\}$) is defined as the fraction of $\pi$-phase volume contained inside the control volume $c$
\begin{equation}\label{eqn:overview:fractiondef}
  \volfrac_c^{\pi,(n)} \defeq \frac{|c \cap \Omega^{\pi,(n)}|}{|c|},
\end{equation}
where $\Omega^{\pi,(n)} \subset \Omega \subset \mathbb{R}^d$ is the part of the $d$-dimensional domain corresponding to the $\pi$-phase at $t = t^{(n)}$, and $|A|$ denotes the volume (area in 2D) of some set $A$.

We note that our proposed methods allow for the presence of arbitrary geometry, which is modelled implicitly using the cut-cell method~\citep{kleefsman2005volume,Droge2007}.
This means that in such cut-cells we find $\mean{\volfrac} < 1$, as the remainder of the control volume is filled by the geometry.
Throughout this paper it will however be assumed that each control volume consists entirely of liquid and/or gas, i.e. $\mean{\volfrac} = 1$, for ease of discussion.

\begin{figure}
  \captionbox{A control volume $c \in \mathcal{C}$ with boundary $\partial c = \bigcup_{f \in \mathcal{F}(c)} f$, where $\mathcal{F}(c) = \{f_1, f_2, f_3, f_4\}$.
  Here $\orientation_{c,f_1} = \orientation_{c,f_2} = +1$ and $\orientation_{c,f_3} = \orientation_{c,f_4} = -1$.\label{fig:notation:notation:mesh_c}}
  [\twofigwidth]{
    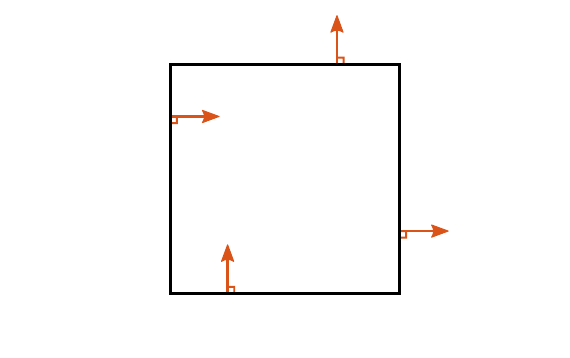
  }
  \hfill
  \captionbox{A staggered control volume $\omega_f$ based on the extrusion of the face $f \in \mathcal{F}$ (dashed line) to the centroids of the neighbouring control volumes $\mathcal{C}(f) = \{c_1, c_2\}$. 
  The faces of the staggered control volume are given by $\mathcal{G}(\omega_f) = \{g_1, g_2, g_3, g_4\}$.\label{fig:notation:notation:mesh_f}}
  [\twofigwidth]{
    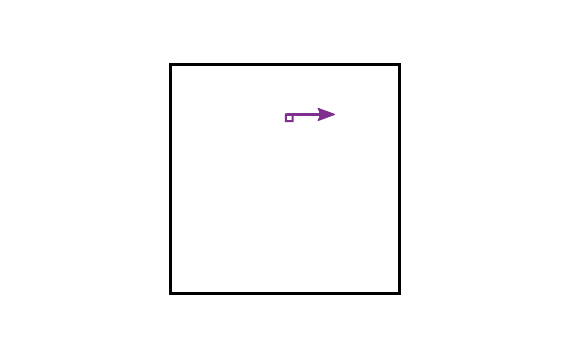 
  }
\end{figure}
The time integration starts with the advection of the interface, as well as the transport of mass and momentum, which are described in detail in~\citet{pubs_Remmerswaal2022a}.
The advection of the interface, for which we use the dimensionally unsplit geometric VOF method based on donating regions (DRs) from~\citep{Owkes2014}, is done according to~\cref{step:overview:fraction_advection}.
\compareFormulations{Interface advection}{overview:fraction_advection}{
  \volfrac^{l,(n+1)} = \volfrac^{l,(n)} - \dt \divh\volumefluxTwo^{l,(n)}
}{
  \volfrac^{l,(n+1)} = \volfrac^{l,(n)} - \dt \divh\volumefluxOne^{l,(n)} 
}
Note that on the left-hand side we show the \onefluid formulation whereas on the right-hand side we show the \twofluid formulation, which we will also do for the remaining time integration steps.
Here $\volumefluxOne^l$ and $\volumefluxTwo^l$ denote the liquid volume flux (i.e., the amount of liquid that flows through a face $f$ during a time step) resulting from the one- and \twofluid formulation respectively, and $\divh: \mathcal{F}^h \rightarrow \mathcal{C}^h$ denotes the finite volume approximation of the divergence operator, which is defined as
\begin{equation}\label{eqn:notation:divergence}
  |c|(\divh u)_c \defeq \sum_{f\in\mathcal{F}(c)} \orientation_{c,f} |f| u_f \approx \integral{\partial c}{u_n}{S},
\end{equation}
where $\orientation_{c,f} \in \{-1, +1\}$ encodes the orientation of the face normal $\+n_f$ such that $\orientation_{c,f} \+n_f$ points out of the control volume $c$, see also~\cref{fig:notation:notation:mesh_c}.
We let the gradient operator $\gradh: \mathcal{C}^h \rightarrow \mathcal{F}^h$ be defined as the negative adjoint of $\divh$, w.r.t. the $L^2$ inner product
\begin{equation}\label{eqn:notation:adjointness}
  \sum_{\mathcal{C}} p \divh u = -\sum_\setextrusion{F} u \gradh p, \quad \forall u \in \mathcal{F}^h, p \in \mathcal{C}^h,
\end{equation}
where we have introduced the following short-hand notation for numerical integration
\begin{equation}\label{eqn:notation:integration}
  \sum_{\mathcal{C}} p \defeq \sum_{c\in\mathcal{C}} |c|p_c, \quad \sum_\setextrusion{F} u \defeq \sum_{f\in\mathcal{F}} |\omega_f|u_f.
\end{equation}
Using~\cref{eqn:notation:divergence,eqn:notation:adjointness} we find that the gradient operator can explicitly be written as 
\begin{equation}\label{eqn:notation:gradient}
  \gradh p \defeq -\frac{1}{h_f}\sum_{c\in\mathcal{C}(f)} \orientation_{c,f} p_c \approx \+n_f \cdot (\gradient p)(\+x_f),
\end{equation}
where $h_f = {|\omega_f|}/{|f|}$ is the face projected distance between the two neighbouring nodes, see also~\cref{fig:notation:notation:mesh_f}.

As the fluids under consideration are assumed to be incompressible, the transport of mass is a direct consequence of the advection of the interface.
Moreover, since we consider a staggered variable arrangement, it follows that obtaining the velocity from momentum requires a staggered notion of mass, and therefore equivalently requires a staggered volume fraction field which we denote by $\volfracstag^\pi \in \mathcal{F}^h$.
The staggered volume fraction field is defined as the volume weighted average of the centred volume fraction
\begin{equation}\label{eqn:interface:twofluid:volfracstag}
  \volfracstag_f^\pi \defeq \frac{1}{2|\omega_f|}\sum_{c\in\mathcal{C}(f)} |c| \volfrac^\pi_c,
\end{equation}
see also~\cref{fig:notation:notation:mesh_f}.
The staggered mass at $t = t^{(n)}$ is then denoted by $(\volfracstag\rho)^{\pi,(n)}$.

Simultaneous with the advection of the interface, the staggered momentum field is advected using the same mass fluxes as used for the transport of mass
\begin{equation}
  \massfluxOne^\pi \defeq \rho^\pi \volumefluxOne^\pi, \quad \massfluxTwo^l \defeq \rho^l \volumefluxTwo^l, \quad \massflux^g \defeq \rho^g \volumeflux^g,
\end{equation}
where the volume flux $\volumeflux^g$ is obtained using the velocity field $u^g$.
Using the same mass fluxes for the transport of mass and momentum allows for conservation of linear momentum, as well as conservation of kinetic energy in the semi-discrete limit $\dt \rightarrow 0$.
For details we refer to~\citet{Veldman2019} as well as the discussion found in~\citep{pubs_Remmerswaal2022a}.
The momentum field for the $\pi$-phase at $t = t^{(n)}$ is denoted by $(\volfracstag \rho u)^{\pi,(n)}$.
The resulting momentum advection method, which is similar to the method first proposed by~\citet{Rudman1998}, is summarised by~\cref{step:overview:momentum_advection}.
\compareFormulations{Transport of momentum}{overview:momentum_advection}{
  \arraycolsep=1.4pt
  \def\arraystretch{1.4}
  \begin{array}{rl}
    u^{l,*} &= \frac{(\volfracstag\rho u)^{l,(n)} - \dt \roundpar{\advectionstag{\massfluxTwo}u}^{l,(n)}}{(\volfracstag\rho)^{l,(n+1)}}\\
    u^{g,*} &= \frac{(\volfracstag\rho)^{g,*}u^{g,(n)} - \dt \roundpar{\advectionstag{\massflux}u}^{g,(n)}}{(\volfracstag\rho)^{g,*}}
  \end{array}
}{
  \onevelo^{*} = \frac{(\mean{\volfracstag\rho} \onevelo)^{(n)} - \dt \mean{\advectionstag{\massfluxOne} \onevelo}^{(n)}}{\mean{\volfracstag\rho}^{(n+1)}}
}
The continuity of the staggered velocity field in the \onefluid model is emphasised by denoting it by $\onevelo \in \mathcal{F}^h$, and the advection operator is denoted by $\advectionstag{m}: \mathcal{F}^h \rightarrow \mathcal{F}^h$.
The staggered gas volume fraction used in~\cref{eqn:overview:momentum_advection:two}, which is denoted by $\volfracstag^{g,*}$, is used to ensure that $u^{g,*} \equiv 1$ if $u^{g,(n)} \equiv 1$, and does not coincide with $\volfracstag^{g,(n+1)}$.
This is contrary to the liquid phase, where $\volfracstag^{l,(n+1)}$ is used directly; this is possible because, as explained in detail in~\citep{pubs_Remmerswaal2022a}, we have chosen to transport the interface using the (extrapolated) liquid velocity field.
It follows that~\cref{step:overview:momentum_advection} does not conserve linear momentum in the gas phase when the \twofluid formulation is used, which we deem acceptable because the contribution of the gas phase to total linear momentum is small whenever $\rho^g \ll \rho^l$, as is the case in our application.

The transport of mass and momentum is carried out under the following CFL constraint
\begin{equation}\label{eqn:overview:onefluid:dr:cfl}
  \max_{c\in\mathcal{C}} \mathcfl_c < \wycflval,
\end{equation}
where $\mathcfl_c$ denotes the local control volume CFL number, and is defined as
\begin{equation}\label{eqn:overview:onefluid:dr:cfl_cell}
  \mathcfl_c \defeq \frac{\dt}{|c|}\sum_{f\in\mathcal{F}(c)} \posflux{-\orientation_{c,f}|f| \onevelo^{(n)}_f}.
\end{equation}
We require that the CFL limit is bounded $\wycflval < 1$, and we usually use a value of $\wycflval = \frac{3}{4}$.

The discretisation of diffusion, which we summarise in~\cref{step:overview:diff}, is done implicitly in time, and is described in more detail in~\cref{sec:diffusion}.
\compareFormulations{Diffusion}{overview:diff}{
  u^{\pi,**} = u^{\pi,*} + \dt {\dfrac{\stagger \divh  \tvar^\mu}{\mean{\volfracstag\rho}^{(n+1)}}}
}{
  \onevelo^{**} = \onevelo^{*} + \dt {\dfrac{\stagger \divh  \tvar^\mu}{\mean{\volfracstag\rho}^{(n+1)}}}
}
Here $\stagger \divh: \mathcal{G}^h \rightarrow \mathcal{F}^h$ denotes the staggered divergence operator which is defined in terms of the boundary integral over the boundary of the staggered control volume $\omega_f$
\begin{equation}\label{eqn:notation:notation:divergence_stag}
  |\omega_f| (\stagger \divh \tvar^\mu)_f \defeq \sum_{g \in \mathcal{G}(\omega_f)} \stagger\orientation_{f,g} |g| \tvar^\mu_g,
\end{equation}
where $\mathcal{G}(\omega_f)$ is the set of faces of the boundary of $\omega_f$, and $\stagger\orientation_{f,g}$ is such that $\stagger\orientation_{f,g} \stagger{\+n}_g$ points out of the staggered control volume $\omega_f$, see also~\cref{fig:notation:notation:mesh_f}.
The set of all faces of all staggered control volumes is denoted by $\mathcal{G}$, with corresponding function space $\mathcal{G}^h$.
The viscous stress tensor $ \tvar^\mu \in \mathcal{G}^h$ will be further specified in~\cref{sec:diffusion}.
Since diffusion will be treated implicitly in time, it does not yield any time step constraint.

The discretisation of the gravity force $F$, which will be discussed in~\cref{sec:gravity} and is given by~\cref{eqn:gravity:model}, results in~\cref{step:overview:grav}.
\compareFormulations{Gravity}{overview:grav}{
  u^{\pi,***} = u^{\pi,**} + \dt {\frac{F}{\mean{\volfracstag\rho}^{(n+1)}}}
}{
  \onevelo^{***} = \onevelo^{**} + \dt {\frac{{F}}{\mean{\volfracstag\rho}^{(n+1)}}}
}
For gravity waves we must ensure that the velocity~\citep[Art. 231]{Lamb1932} of the shortest wave, which we take equal to\footnote{Using the mesh width $h$ was not sufficient for obtaining stability.} $h / {2\pi}$, does not exceed their numerical speed limit, resulting in the following time step restriction
\begin{equation}\label{eqn:overview:grav_timestep}
  \dt \le \sqrt{\frac{\rho^l + \rho^g}{\abs{\rho^l - \rho^g}}\frac{h}{\abs{\+g}_2}},
\end{equation}
where $\+g$ denotes the acceleration due to gravity.

All that remains now is to project the `predicted' velocity field $u^{\pi,***}$ onto the divergence free subspace of $\mathcal{F}^h$.
This is done via the pressure gradient $\sgrad{\pi}: \mathcal{C}^h \times \mathcal{F}^h \rightarrow \mathcal{F}^h$, where $\xi \in \mathcal{F}^h$ denotes the gradient jump, which will be discussed in~\cref{sec:poisson}.
The gradient has been scaled by the density, includes the effect of surface tension and moreover ensures that the velocity field satisfies~\cref{eqn:intro:normal_smoothness} in some approximate sense for the \twofluid formulation.
The inclusion of surface tension, via the Young--Laplace equation, results in the following time step restriction
\begin{equation}\label{eqn:overview:st_timestep}
  \dt \le \sqrt{\frac{(\rho^l + \rho^g) h^3}{2 \pi \sigma}},
\end{equation}
which corresponds to taking a time step no larger than the oscillation period of a capillary wave of wavelength $h$~\citep[Art. 266]{Lamb1932}.
We summarise the resulting pressure Poisson problem in~\cref{step:overview:poisson}.
{  \arraycolsep=1.4pt
  \def\arraystretch{1.4}
  \compareFormulations{Pressure Poisson}{overview:poisson}{
  \begin{array}{rl}
    u^{\pi,(n+1)} &= u^{\pi,***} - \dt \sgrad{\pi}\\
    \divh\mean{au}^{(n+1)} &= 0\\
    \normalstag \cdot \jumpOper(u^{g}, u^{l})^{(n+1)} &= 0
  \end{array}
}{
  \begin{array}{rl}
    \onevelo^{(n+1)} &= \onevelo^{***} - \dt \sgradone\\
    \divh\onevelo^{(n+1)} &= 0
  \end{array}
}}

Here the fluid apertures $a^\pi \in \mathcal{F}^h$ denote the fraction of the face $f$ contained inside the $\pi$-phase, and are illustrated in~\cref{fig:interface:twofluid:cutcell}.
The interface normal, interpolated to the face $f$ and subsequently normalised, is denoted by $\normalstag \in [\mathcal{F}^h]^d$.
The additional constraint on the velocity fields, given by the third equation in~\cref{eqn:overview:poisson:two}, implicitly imposes~\cref{eqn:intro:normal_smoothness} and will be discussed in detail in~\cref{sec:poisson}.
There we will also define the scaled gradient operator $\gfmgrad^\pi$ as well as the jump operator $\jumpOper$.

  \section{Poisson problem}\label{sec:poisson}
We will now focus our attention on the most crucial part of the discretisation for our proposed \twofluid formulation: the coupling of the two fluids via the pressure Poisson equation.
Contrary to the approach followed in~\citep{Desjardins2010,Vukcevic2017}, wherein the velocity jump is simply removed after each time step, we impose that merely the normal component of the velocity is continuous, as stated in~\cref{eqn:intro:normal_smoothness}.

We start by briefly discussing the pressure Poisson problem discrete in time, but continuous in space.
Then we discuss discretisation approaches for the divergence as well as gradient operator in the presence of a velocity discontinuity, resulting in a generalisation of the ghost fluid method (GFM)~\citep{Liu2000}.
The resulting method is then validated using an academic test case, and slightly altered to ensure that conservation properties are adhered to.
We then discuss the relation between the resulting one- and \twofluid formulations.

\subsection{The pressure Poisson problem}\label{sec:poisson:analytical}
Before we discuss the full discretisation of the pressure Poisson problem shown in~\cref{eqn:overview:poisson:two}, we first present the continuous in space but discrete in time equivalent.
Continuous in space, the velocity update due to the pressure is given by
\begin{equation}\label{eqn:poisson:jumps:velo_update}
  \+u^{\pi,(n+1)} = \+u^{\pi,***} - \frac{\dt}{\rho^\pi} \gradient p^\pi.
\end{equation}
Imposing~\cref{eqn:intro:normal_smoothness} on $\+u^{\pi,(n+1)}$ in the left-hand side of~\cref{eqn:poisson:jumps:velo_update} results in the following jump condition on the pressure gradient 
\begin{equation}\label{eqn:poisson:gradjump}
  \jump{\frac{1}{\rho} \partial_\eta p} = \frac{\jump{u_\eta^{***}}}{\dt}.
\end{equation}
This novel jump condition is constructed exactly such that our velocity field remains continuous in the interface normal direction, and the discretisation of this jump condition is therefore essential in our proposed \twofluid formulation.
Surface tension is modelled using the Young--Laplace equation
\begin{equation}\label{eqn:poisson:valjump}
  \jump{p} = -\sigma \kappa,
\end{equation}
where $\kappa$ denotes the interface curvature (sum of the principal curvatures) and $\sigma$ denotes the surface energy coefficient.
We therefore have two jump conditions, as given by~\cref{eqn:poisson:gradjump,eqn:poisson:valjump}, on the pressure which need to be taken into account when discretising the gradient operator.

The conservative form of the divergence constraint can be derived from the mass conservation equations, which we repeat here for convenience
\begin{equation}\label{eqn:poisson:mass_cons}
  \integral{c^{\pi,(n+1)}}{\rho^\pi}{V} - \integral{c^{\pi,(n)}}{\rho^\pi}{V} + \oneDIntegral{t^{(n)}}{t^{(n+1)}}{\integral{\partial c^\pi(t) \setminus I(t)}{\rho^\pi u^\pi_n}{S}}{t} = 0.
\end{equation}
Here $c^\pi(t) = c \cap \Omega^\pi(t)$ and $I(t)$ denotes the interface that separates the two fluid phases.
Dividing this equation by $\rho^\pi$, which is assumed constant, adding the result for both phases, and subsequently letting $\dt = t^{(n+1)} - t^{(n)} \rightarrow 0$, results in the divergence constraint
\begin{equation}\label{eqn:poisson:cons_constraint}
  \integral{\partial c}{u_n}{S} = 0.
\end{equation}
In what follows we will discuss the discretisation of the Poisson problem, as defined by~\cref{eqn:poisson:cons_constraint,eqn:poisson:gradjump,eqn:poisson:valjump,eqn:poisson:jumps:velo_update}.

\subsection{Discrete pressure Poisson problem}
The discretisation of the Laplace operator in the presence of discontinuities in the coefficients and/or solution can be done in various ways~\citep{Leveque1994,Ewing1999,Li2003,Guittet2015,Egan2020}.
We are specifically interested in a discretisation for which the divergence and gradient operator are discretised separately since we want the divergence constraint to hold exactly.
Therefore a discretisation which directly discretises the Laplace operator as a whole is not of interest, and in what follows we will separately discuss the discretisation of the divergence as well as gradient operator.

\begin{figure}
  \subcaptionbox{Definition of the face aperture.\label{fig:interface:twofluid:cutcell_aperture}}
  [\twofigwidth]{
\begingroup%
  \makeatletter%
  \providecommand\color[2][]{%
    \errmessage{(Inkscape) Color is used for the text in Inkscape, but the package 'color.sty' is not loaded}%
    \renewcommand\color[2][]{}%
  }%
  \providecommand\transparent[1]{%
    \errmessage{(Inkscape) Transparency is used (non-zero) for the text in Inkscape, but the package 'transparent.sty' is not loaded}%
    \renewcommand\transparent[1]{}%
  }%
  \providecommand\rotatebox[2]{#2}%
  \newcommand*\fsize{\dimexpr\f@size pt\relax}%
  \newcommand*\lineheight[1]{\fontsize{\fsize}{#1\fsize}\selectfont}%
  \ifx\svgwidth\undefined%
    \setlength{\unitlength}{132.40209698bp}%
    \ifx\svgscale\undefined%
      \relax%
    \else%
      \setlength{\unitlength}{\unitlength * \real{\svgscale}}%
    \fi%
  \else%
    \setlength{\unitlength}{\svgwidth}%
  \fi%
  \global\let\svgwidth\undefined%
  \global\let\svgscale\undefined%
  \makeatother%
  \begin{picture}(1,0.50693777)%
    \lineheight{1}%
    \setlength\tabcolsep{0pt}%
    \put(0,0){\includegraphics[width=\unitlength,page=1]{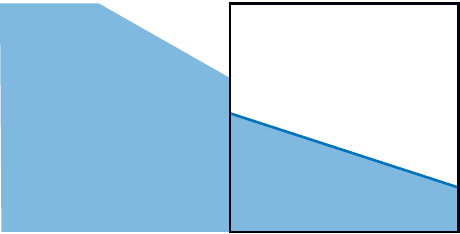}}%
    \put(0.55800033,0.38402477){\color[rgb]{0,0,0.10980392}\makebox(0,0)[lt]{\lineheight{1.25}\smash{\begin{tabular}[t]{l}$f^g$\end{tabular}}}}%
    \put(0,0){\includegraphics[width=\unitlength,page=2]{cutcell_aperture.pdf}}%
    \put(0.38562297,0.13654514){\color[rgb]{0,0,0.10980392}\makebox(0,0)[lt]{\lineheight{1.25}\smash{\begin{tabular}[t]{l}$f^l$\end{tabular}}}}%
    \put(0.40350435,0.42550979){\color[rgb]{0,0.44705882,0.74509804}\makebox(0,0)[lt]{\lineheight{1.25}\smash{\begin{tabular}[t]{l}$\+\eta_{c_1}$\end{tabular}}}}%
    \put(0.85697869,0.17946067){\color[rgb]{0,0.44705882,0.74509804}\makebox(0,0)[lt]{\lineheight{1.25}\smash{\begin{tabular}[t]{l}$\+\eta_{c_2}$\end{tabular}}}}%
  \end{picture}%
\endgroup%

  }
  \hfill
  \subcaptionbox{Boundary integral split into two parts.\label{fig:interface:twofluid:cutcell_divergence}}
  [\twofigwidth]{
\begingroup%
  \makeatletter%
  \providecommand\color[2][]{%
    \errmessage{(Inkscape) Color is used for the text in Inkscape, but the package 'color.sty' is not loaded}%
    \renewcommand\color[2][]{}%
  }%
  \providecommand\transparent[1]{%
    \errmessage{(Inkscape) Transparency is used (non-zero) for the text in Inkscape, but the package 'transparent.sty' is not loaded}%
    \renewcommand\transparent[1]{}%
  }%
  \providecommand\rotatebox[2]{#2}%
  \newcommand*\fsize{\dimexpr\f@size pt\relax}%
  \newcommand*\lineheight[1]{\fontsize{\fsize}{#1\fsize}\selectfont}%
  \ifx\svgwidth\undefined%
    \setlength{\unitlength}{67.23459213bp}%
    \ifx\svgscale\undefined%
      \relax%
    \else%
      \setlength{\unitlength}{\unitlength * \real{\svgscale}}%
    \fi%
  \else%
    \setlength{\unitlength}{\svgwidth}%
  \fi%
  \global\let\svgwidth\undefined%
  \global\let\svgscale\undefined%
  \makeatother%
  \begin{picture}(1,1.00000017)%
    \lineheight{1}%
    \setlength\tabcolsep{0pt}%
    \put(0,0){\includegraphics[width=\unitlength,page=1]{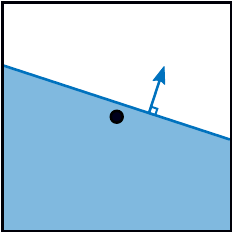}}%
    \put(0.71325294,0.5592951){\color[rgb]{0,0.44705882,0.74509804}\makebox(0,0)[lt]{\lineheight{1.25}\smash{\begin{tabular}[t]{l}$\+\eta_{c}$\end{tabular}}}}%
    \put(0,0){\includegraphics[width=\unitlength,page=2]{cutcell_divergence.pdf}}%
    \put(0.06049272,0.14162407){\color[rgb]{0.85490196,0.3254902,0.09803922}\makebox(0,0)[lt]{\lineheight{1.25}\smash{\begin{tabular}[t]{l}$\integral{\partial c^l \setminus I}{u^l_n}{S}$\end{tabular}}}}%
    \put(0.13997184,0.82765447){\color[rgb]{0.49411765,0.18431373,0.55686275}\makebox(0,0)[lt]{\lineheight{1.25}\smash{\begin{tabular}[t]{l}$\integral{\partial c^g \setminus I}{u^g_n}{S}$\end{tabular}}}}%
    \put(11.0319699,-8.75232437){\color[rgb]{0,0,0}\makebox(0,0)[lt]{\begin{minipage}{6.15543658\unitlength}\raggedright \end{minipage}}}%
  \end{picture}%
\endgroup%

  }
  \caption{Illustration of the cut-cell method applied to the approximation of the divergence constraint given by~\cref{eqn:interface:twofluid:cutcell:div_approx}.}
  \label{fig:interface:twofluid:cutcell}
\end{figure}
The boundary integral in~\cref{eqn:poisson:cons_constraint} can be written as a sum of two boundary integrals, where we sharply distinguish the integration over the liquid and gaseous parts of the boundary (see also~\cref{fig:interface:twofluid:cutcell_divergence})
\begin{equation}
  \integral{\partial c}{u_n}{S} = \integral{\partial c^l \setminus I}{u^l_n}{S} + \integral{\partial c^g \setminus I}{u^g_n}{S}.
\end{equation}
Hence for the sharp discretisation of the boundary integral we need to sharply identify which part of a face $f \subset \partial c$ is inside the liquid, and which part is in the gas phase.
To this end we let the face aperture $a^\pi \in \mathcal{F}^h$ be defined as
\begin{equation}
  a^{\pi,(n)}_f \defeq \frac{|f \cap \Omega^{\pi,(n)}|}{|f|}.
\end{equation}
Note that since a \emph{piecewise} reconstruction of the interface is used, the face aperture is averaged using the approximations from both sides of the face, as shown in~\cref{fig:interface:twofluid:cutcell_aperture}.

The face apertures allow for the following cut-cell~\citep{Crockett2011} approximation of~\cref{eqn:poisson:cons_constraint}
\begin{equation}\label{eqn:interface:twofluid:cutcell:div_approx}
  |c|\divh(\mean{a u})_c = 0.
\end{equation}
The cut-cell divergence operator is illustrated in~\cref{fig:interface:twofluid:cutcell_divergence}.
The advantage of using such a cut-cell approximation as opposed to using, e.g., the GFM, is that the cut-cell method naturally satisfies the discrete equivalent of Gauss's theorem (assuming $\jump{u_\eta} = 0$)
\begin{equation}\label{eqn:poisson:gauss_domain}
  \integral{\Omega}{\divergence \+u}{V} = \integral{\partial\Omega}{u_n}{S} \quad \xrightarrow[\text{equivalent}]{\text{discrete}} \quad \sum_{\mathcal{C}} \divh\mean{au} = \sum_{f\in\mathcal{F}_\Gamma} |f|u_f,
\end{equation}
where $\mathcal{F}_\Gamma \subset \mathcal{F}$ is the set of faces that lie on the boundary, and we assume that $\+n_f$ is outward pointing at the boundary (for notational convenience).
This is not just a property that is `nice to have', but is required to ensure that the resulting linear system of equations has a right-hand side which is in the image of the linear operator, and thus has a solution.
The image of the linear operator is nontrivial because this linear operator has a nontrivial null space spanned by $p \equiv 1$ due to the Neumann boundary condition imposed on the pressure.

In the remainder of this section we will discuss the discretisation of the gradient operator in the presence of the jump conditions given by~\cref{eqn:poisson:valjump,eqn:poisson:gradjump}.

\subsection{The jump interpolant}
\begin{figure}
  \centering
  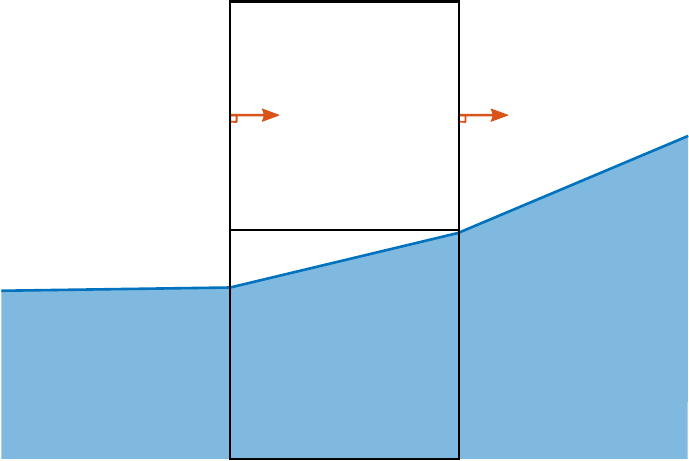
  \caption{Illustration of the jump interpolants~\cref{eqn:poisson:jump_interpolant:jump_interpolant,eqn:poisson:jump_interpolant:jump_interpolant_beta}.
  The set of neighbouring faces whose face normal is perpendicular to that of $f$ is given by $\mathcal{F}^{\perp_1}(f) = \{h_1, h_2, h_3, h_4\}$, whereas the set of neighbouring faces whose face normal coincides with that of $f$ is given by $\{h^g, h^l\}$.
  Note that the face $h_1$ does not contribute to the jump interpolant since the corresponding weight vanishes $w^{\perp_1}_{f,h_1} = 0$ (see~\cref{eqn:poisson:jump_interpolant:weight}).
  }\label{fig:poisson:interpolant:example}
\end{figure}
The imposition of~\cref{eqn:poisson:gradjump} is not straightforward when a staggered variable arrangement is used: the projection of the gradient in the interface normal direction necessarily requires interpolation since only the face normal component of the gradient is defined at any face.
We will now define the jump interpolant $\jumpOper$ which, for a staggered velocity field $u^\pi \in \mathcal{F}^h$, approximates the vector-valued jump at the face centroid
\begin{equation} 
  \jumpOper: [\mathcal{F}^h]^2 \rightarrow [\mathcal{F}^h]^d, \quad (\jumpOper(u^g, u^l))_f \approx \jump{\+u}(\+x_f).
\end{equation}

The staggered velocity $u^{\pi,(n)}_f$ is defined if and only if the corresponding staggered volume fraction is non-zero $\volfracstag^{\pi,(n)}_f > 0$, and therefore a mixed face is defined as one for which both staggered volume fractions are non-zero.
It follows that for a mixed face $f \in \mathcal{F}$ we can easily compute the face normal component of the jump as $\jump{u}_f$, resulting in the face normal part $\jump{u}_f \+n_f$ of the vector-valued jump.
The jump interpolant $\jumpOper$ is then defined by specifying how we approximate the parts of the vector-valued jump which are tangential to $\+n_f$: the face tangential parts.

For each face $f$ we define a subset $\mathcal{F}^{\perp_k}(f) \subset \mathcal{F}$, for $k=1,\ldots,d-1$, as the set of neighbouring faces whose face normal is perpendicular to that of $f$, as shown by the horizontal arrows in~\cref{fig:poisson:interpolant:example}.
These subsets provide a way to interpolate the face tangential parts of the velocity jump, resulting in the following definition of the jump interpolant for mixed faces
\begin{equation}\label{eqn:poisson:jump_interpolant:jump_interpolant}
  \jumpOper^1(u^{g},u^{l})_f \defeq \underbrace{\jump{u}_f \+n_f}_\text{face normal part} + \sum_{k=1}^{d-1} \underbrace{\sum_{h \in \mathcal{F}^{\perp_k}(f)} w^{\perp_k}_{f,h}\jump{u}_h \+n_h}_\text{$k$-th face tangential part},
\end{equation}
where face tangential parts are averaged using the weights $w^{\perp_k}_{f,h}$, resulting in a first-order accurate interpolant.
The weight $w^{\perp_k}_{f,h}$ is defined such that it is non-zero only for mixed faces $h$ (for which the jump $\jump{u}_h$ is defined)
\begin{equation}\label{eqn:poisson:jump_interpolant:weight}
  w^{\perp_k}_{f,h} = \frac{\volfracstag^g_h\volfracstag^l_h}{\sum_{h' \in \mathcal{F}^{\perp_k}(f)} \volfracstag^g_{h'}\volfracstag^l_{h'}},
\end{equation}
and moreover is such that the interpolant varies continuously in time, as a consequence of the (staggered) volume fractions varying continuously in time.

Given the jump interpolant $\jumpOper^1$ we are now able to compute the interface normal component of the jump in a staggered field.
We will use this in conjunction with the ghost fluid method (GFM), which will be introduced next.

\subsection{The ghost fluid method}
\begin{figure}
  \centering
  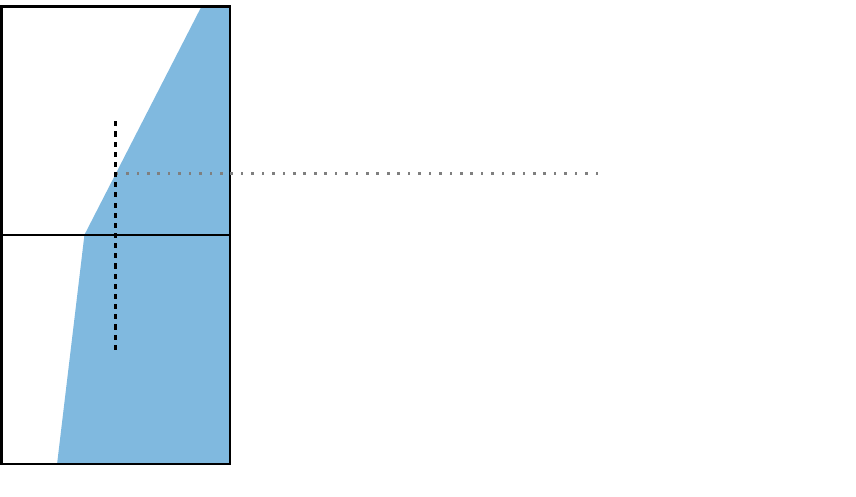
  \caption{The distance $\stagger\phi^\pi_f$ is defined as the fraction of the dashed line in the left-hand side figure that is contained in the $\pi$-phase, and is approximated using the sharp reconstruction of the interface.
  The right-hand side figure shows the pressure along the dashed line of the left-hand side figure, which has a discontinuity in both the value as well as the slope.}
  \label{fig:poisson:gfm:distance_fun}
\end{figure}
In~\citet{Liu2000} the GFM is introduced, which is a numerical method for computing a derivative in the presence of value and derivative jumps.
We denote the value and derivative jumps as
\begin{equation}\label{eqn:poisson:gfm:jumps}
  \zeta_f = \jump{p}(\stagger{\+x}^I_f), \quad \xi_f = \+n_f \cdot \jump{\frac{1}{\rho} \gradient p}(\stagger{\+x}^I_f),
\end{equation}
respectively, where $\stagger{\+x}^I_f$ is located on the line connecting $\+x_{c_1}$ and $\+x_{c_2}$ as well as on the interface $I$, see also~\cref{fig:poisson:gfm:distance_fun}.
We let $\stagger\phi^\pi$ denote the unsigned and nondimensional face projected distance (for which $\mean{\stagger\phi} = 1$) from the node in the $\pi$-face to the position at the interface $\stagger{\+x}^I_f$.

Imposing the jump conditions at first-order accuracy results in the GFM gradient (a derivation is given in~\cref{sec:app:gfm})
\begin{equation}\label{eqn:poisson:mdgfm:gfm}
  \sgradf{\pi} \defeq \frac{(\gradh p)_f + \zeta (\gradh\chi^l)_f + \gfmthing^\pi_f \xi_f}{\mean{\stagger\phi \rho}_f} \approx \frac{1}{\rho^\pi}\+n_f \cdot (\gradient p^\pi)(\+x_f),
\end{equation}
where $\chi^l \in \mathcal{C}^h$ denotes the liquid indicator function for which $\chi^l_c \in \{0, 1\}$, and 
\begin{equation}\label{eqn:poisson:gfm:rtilde}
  \gfmthing^l \defeq -\stagger\phi^g \rho^g, \quad \gfmthing^g \defeq \stagger\phi^l \rho^l,
\end{equation}
such that $\jump{\gfmthing} = \mean{\stagger\phi\rho}$ and therefore $\jump{\sgrad{}} = \xi$, as desired. 

The GFM requires the jump of the face normal derivative of the pressure, rather than the interface normal derivative given by~\cref{eqn:poisson:gradjump}.
We can write the jump of the face normal derivative in terms of the interface normal and tangential jumps as follows (for simplicity in 2D)
\begin{equation}
  \xi_f = \+n_f \cdot \squarepar{\normalstag_f \jump{\frac{1}{\rho} \partial_\eta p}(\+x^I_f) + \tangentstag_f \jump{\frac{1}{\rho} \partial_\tau p}(\+x^I_f)},
\end{equation}
where $\tangentstag_f$ denotes the vector which is orthogonal to the staggered interface normal $\normalstag_f$.
Neglecting the interface tangential component (which is multiplied by $\+n_f \cdot \tangentstag_f$) of the pressure gradient jump, as proposed by~\citet{Liu2000}, results in
\begin{equation}\label{eqn:poisson:mdgfm:liujump}
  \xi^\mtext{GFM} \defeq \frac{\+n \cdot \normalstag}{\dt} \normalstag \cdot \jumpOper(u^{g},u^{l})^{***},
\end{equation}
where for notational convenience we let $\jumpOper(u^{g},u^{l})^{***} = \jumpOper(u^{g,***},u^{l,***})$.
Note that~\cref{eqn:poisson:mdgfm:liujump} will be inconsistent since in general $\+n_f \cdot \tangentstag_f \neq 0$, and therefore results in undesired smearing of the jump condition given by~\cref{eqn:poisson:gradjump}.

In~\citet{Egan2020} the xGFM is proposed as a generalisation to the GFM which resolves the aforementioned inconsistency.
The xGFM solves a sequence of Poisson problems, where each such Poisson problem is discretised using the GFM method.
This approach will not be considered here. 
Instead we will propose a method to impose~\cref{eqn:intro:normal_smoothness} directly and consistently using a \emph{single} Poisson solve in~\cref{sec:poisson:mdgfm}, thereby proposing an alternative to xGFM for the generalisation of the GFM for imposing jump conditions of the form~\cref{eqn:poisson:gradjump}.


\subsection{A multidimensional variant of the ghost fluid method}\label{sec:poisson:mdgfm}
We will now introduce a multidimensional generalisation of the GFM (MDGFM) which consistently imposes~\cref{eqn:poisson:gradjump}, that is, without neglecting the interface tangential components of the pressure gradient jump.

Using the jump interpolant~\cref{eqn:poisson:jump_interpolant:jump_interpolant} we can directly (and consistently) impose~\cref{eqn:poisson:gradjump}, resulting in
\begin{equation}\label{eqn:poisson:mdgfm:gradient_jump}
  \normalstag \cdot \jumpOper\roundpar{\sgrad{g}, \sgrad{l}} = \frac{\normalstag \cdot \jumpOper(u^{g},u^{l})^{***}}{\dt},
\end{equation}
where the gradient $\sgrad{\pi}$ is given by~\cref{eqn:poisson:mdgfm:gfm} with unknown $p, \xi$.
Note that $\jump{\sgrad{}} = \xi$ and therefore, according to the definition of the jump interpolant~\cref{eqn:poisson:jump_interpolant:jump_interpolant}, the left hand-side of~\cref{eqn:poisson:mdgfm:gradient_jump} will depend only on the gradient jump $\xi$ and in particular not on the pressure $p$.
This implies that the gradient jump can be found independently from the pressure, which is remarkable since we have not imposed any condition on the interface tangential component of the gradient jump.

Regardless of this remarkable independence, we can use~\cref{eqn:poisson:mdgfm:gradient_jump} to express the jump $\xi_f$ in terms of its neighbouring values $\xi_h$
\begin{equation}\label{eqn:poisson:mdgfm:impose_gradient_jump_inverted_bounded_unbounded}
  \xi_f = \dt^{-1}{\jump{u}^{***}_f} + \sum_{k=1}^{d-1}\sum_{h \in \mathcal{F}^{\perp_k}(f)} w^{\perp_k}_{f,h} \roundpar{\dt^{-1}\jump{u}^{***}_h - \xi_h} \ratio{f,h},
\end{equation}
where the ratio of face normal components of the interface normal $\normalstag_f$ is given by
\begin{equation}
  \ratio{f,h} \defeq \frac{\normalstag_f \cdot \+n_h}{\normalstag_f \cdot \+n_f}.
\end{equation}
We find that the ratio $\ratio{f,h}$ tends to infinity $\ratio{f,h} \rightarrow \pm \infty$ if the interface normal $\normalstag_f$ is tangent to the face $f$.
The aforementioned unboundedness results in an ill-posed linear system of equations when solving for the jump $\xi$, making it impossible to satisfy~\cref{eqn:poisson:mdgfm:gradient_jump} in practice using the jump interpolant given by~\cref{eqn:poisson:jump_interpolant:jump_interpolant}, and therefore we now propose a slight modification of the jump interpolant.

The remarkable independence of the jump $\xi$ on the pressure $p$ suggests modifying the jump interpolant in a way that the jump becomes dependent on the pressure.
We propose the following redefinition of the jump interpolant (cf.~\cref{eqn:poisson:jump_interpolant:jump_interpolant})
\begin{equation}\label{eqn:poisson:jump_interpolant:jump_interpolant_beta}
  \jumpOper^\beta(u^{g},u^{l})_f \defeq \squarepar{\beta_f^{-1}\jump{u}_f + (1-\beta_f^{-1})(u^g_{h^g} - u^l_{h^l})} \+n_f + \sum_{k=1}^{d-1}\sum_{h \in \mathcal{F}^{\perp_k}(f)} w^{\perp_k}_{f,h} \jump{u}_h \+n_h,
\end{equation}
where $\beta_f \ge 0$ is some non-negative coefficient (when solving~\cref{eqn:poisson:jump_interpolant:jump_interpolant_beta} for $\xi_f$, resulting in~\cref{eqn:poisson:mdgfm:impose_gradient_jump_inverted_bounded_beta}, we find that $\beta_f$ appears rather than $\beta_f^{-1}$, and therefore $\beta_f = 0$ is permitted) and $h^\pi \in \mathcal{F}$ are neighbouring faces with $\+n_f = \+n_{h^\pi}$, $\rho^l\volfracstag^{l}_{h^{l}} > \rho^g\volfracstag^{g}_{h^{l}}$ and $\rho^g\volfracstag^{g}_{h^{g}} > \rho^l\volfracstag^{l}_{h^{g}}$, as shown in~\cref{fig:poisson:interpolant:example}.
We let either $h^g = f$ or $h^l = f$ depending on which of the corresponding masses is largest inside $\omega_f$.
Thus we average two different approximations of the face normal component of the jump.
Note that when $\beta \equiv 1$ we re-obtain our initial definition of the jump interpolant as given by~\cref{eqn:poisson:jump_interpolant:jump_interpolant}.

We will now construct the coefficient $\beta_f$ such that the jump $\xi_f$ is always well-defined, in the sense that any coefficient appearing in its definition is always bounded.
To this end, we again (as we did with~\cref{eqn:poisson:jump_interpolant:jump_interpolant} resulting in~\cref{eqn:poisson:mdgfm:impose_gradient_jump_inverted_bounded_unbounded}) impose~\cref{eqn:poisson:mdgfm:gradient_jump} and express the jump $\xi_f$ in terms of the neighbouring jump values, which now results in
\begin{multline}\label{eqn:poisson:mdgfm:impose_gradient_jump_inverted_bounded_beta}
  \xi_f^\mtext{MDGFM} \defeq \dt^{-1}{\jump{u}^{***}_f} + \sum_{k=1}^{d-1}\sum_{h \in \mathcal{F}^{\perp_k}(f)} w^{\perp_k}_{f,h} \roundpar{\dt^{-1}\jump{u}^{***}_h - \xi_h} \beta_f\ratio{f,h}\\ + (\beta_f-1)\squarepar{\dt^{-1}\roundpar{u^{g,***}_{h^g} - u^{l,***}_{h^l}} - \roundpar{\sgrad{g}_{h^g} - {\sgrad{l}_{h^l}}}}.
\end{multline}
The previously (when we used $\beta \equiv 1$) problematic term has now become $\beta_f\ratio{f,h}$. 
Note that whereas the reciprocal $\beta_f^{-1}$ appears in the redefinition of the jump interpolant in~\cref{eqn:poisson:jump_interpolant:jump_interpolant_beta}, the value $\beta_f$ itself appears in~\cref{eqn:poisson:mdgfm:impose_gradient_jump_inverted_bounded_beta}, thus showing that $\beta_f = 0$ is permitted.

Suppose that $\beta_f$ is a function of the absolute value of the face normal component of the interface normal
\begin{equation}\label{eqn:poisson:mdgfm:betaasfun}
  \beta_f \defeq \beta(\abs{\normalstag_f \cdot \+n_f}).
\end{equation}
We are now free to model the function $\beta(x)$.
First and foremost we require that $\beta_f\abs{\ratio{f,h}}$ is bounded at all times.
The ratio $\ratio{f,h}$ can be bounded by (with equality in 2D)
\begin{equation}
  \abs{\ratio{f,h}} = \frac{\abs{\normalstag_f \cdot \+n_h}}{\abs{\normalstag_f \cdot \+n_f}}  \le \frac{\sqrt{1-(\normalstag_f \cdot \+n_f)^2}}{\abs{\normalstag_f \cdot \+n_f}},
\end{equation}
where we made use of
\begin{equation}
  1 = \abs{\normalstag_f}^2_2 = (\normalstag_f \cdot \+n_f)^2 + \sum_{k=1}^{d-1} (\normalstag_f \cdot \+n_f^{\perp_k})^2,
\end{equation}
as well as $h \in \mathcal{F}^{\perp_k}(f)$ such that $\+n_h = \+n^{\perp_k}_f$ for some $k \in \{1, \ldots, d-1\}$.
Here $\+n^{\perp_k}_f$ denotes the $k$-th vector orthogonal to $\+n_f$.

Hence boundedness of $\beta_f\abs{\ratio{f,h}}$ is implied by boundedness of $\beta(x)\frac{\sqrt{1-x^2}}{x}$ for $x \in [0, 1]$.
Furthermore we note that the face normal component of the gradient jump should not be imposed at all if the interface normal is (nearly) orthogonal to the face normal ($\beta(x) \approx 0$ when $x \approx 0$), and finally we have no boundedness issues whenever the interface normal is (nearly) aligned with the face normal, and therefore the original interpolant should be obtained in such cases ($\beta(x) \approx 1$ when $x \approx 1$).
We propose to use
\begin{equation}\label{eqn:poisson:mdgfm:betafundef}
  \beta(x) \defeq 3x^2 - 2{x}^3,
\end{equation}
for which it holds that
\begin{equation}
  \beta(x)\frac{\sqrt{1-x^2}}{x} \le 0.87278\ldots, \quad \beta(0) = 0, \quad \beta'(0) = 0, \quad \beta(1) = 1, \quad \beta'(1) = 0,
\end{equation}
for $x \in [0, 1]$.

Due to the newly introduced term in~\cref{eqn:poisson:mdgfm:impose_gradient_jump_inverted_bounded_beta} the jump $\xi_f$ now also depends on the pressure, which implies that we must solve for $p$ and $\xi$ simultaneously. 
This makes sense because the interface tangential component of the gradient jump is not known a priori.

To summarise: we solve~\cref{eqn:overview:poisson:two} where the gradient $\sgrad{\pi}$ is given by~\cref{eqn:poisson:mdgfm:gfm} and the gradient jump $\xi$ is given by~\cref{eqn:poisson:mdgfm:impose_gradient_jump_inverted_bounded_beta}.
The latter definition of the gradient jump is a direct consequence of imposing~\cref{eqn:poisson:mdgfm:gradient_jump} using the jump interpolant defined by~\cref{eqn:poisson:jump_interpolant:jump_interpolant_beta} which, in turn, implicitly imposes continuity of the interface normal component of velocity, as stated in~\cref{eqn:intro:normal_smoothness}.

\subsection{Validation}
The pressure Poisson problem is considered (in 2D) where the exact pressure, which is unique up to a constant, is given by
\begin{equation}
  p = x - y + \chi^l\curlypar{\sigma\kappa + 40 \squarepar{(x-0.5)^2 + (y-0.5)^2 - R^2} \roundpar{x - y}},
\end{equation}
for $\sigma = 0.1$. 
The curvature $\kappa$ is approximated using the GHF method~\citep{Popinet2009}.
The liquid domain $\Omega^l$ is the interior of a circle with radius $R = 0.3$ centred in a unit square domain $\Omega = [0, 1]^2$.
An example approximate solution to the pressure Poisson problem is shown in~\cref{fig:poisson:validation:poisson:example}, where we have used a locally refined grid in order to illustrate that the proposed jump interpolant can also be used at grid refinement interfaces.

The `predictor' velocity field is given by
\begin{equation}
  \+u^{***} = \chi^l \begin{bmatrix}
    \partial_y \\ -\partial_x
  \end{bmatrix}
  \frac{\sin(2\pi ((x-0.5)^2 + (y-0.5)^2))}{2 \pi}
  + \frac{1}{\rho} \gradient p,
\end{equation}
such that the final velocity field $\+u = \+u^{***} - \frac{1}{\rho} \gradient p$ is divergence free and is continuous in the interface normal direction.
Note that the interface tangential component of the pressure gradient is given by (evaluated at the interface)
\begin{equation}
  \partial_\tau p^\pi = \tau_x - \tau_y,
\end{equation}
from which it follows that the scaled tangential derivative jump is in general non-zero 
\begin{equation}\label{eqn:poisson:validation:poisson:tangentjump}
  \jump{\frac{1}{\rho} \partial_\tau p} = \jump{\frac{1}{\rho}}(\tau_x - \tau_y).
\end{equation}
Recall that the GFM~\citep{Liu2000}, which resulted in the gradient jump given by~\cref{eqn:poisson:mdgfm:liujump}, neglects exactly this tangential derivative jump, and therefore the GFM is expected to be inconsistent for this test problem if $\jump{\frac{1}{\rho}} \neq 0$, or equivalently if $\ratio{\rho} \neq 1$, where the density ratio is denoted as $\ratio{\rho} \defeq \ratiofull{\rho}$.

\begin{figure}
  \captionbox{Approximate pressure using $\ratio{\rho} = 10^{-3}$ on a locally refined grid.
  The boundaries of the (potentially overlapping) gridblocks are shown, where each gridblock contains $8^2$ control volumes.
  \label{fig:poisson:validation:poisson:example} }
  [\twofigwidth]{
    \def\svgscale{1.3}
    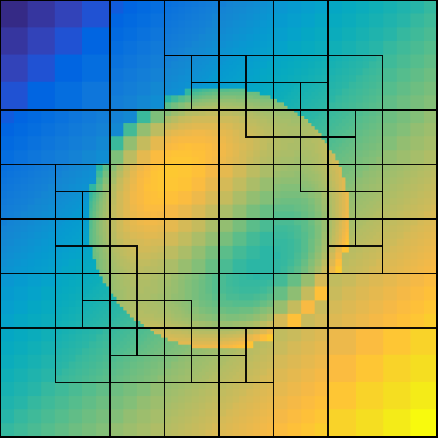
  }\hfill
  \captionbox{The $L^\infty$-norm of the velocity error for several density ratios. Here the solid and dashed lines are results using the MDGFM (solid) and GFM (dashed) method respectively.
  \label{fig:poisson:validation:poisson:convergence_g}}
  [\twofigwidth]{
    \def\tikzWidth{\textwidth*0.34}
    \def\tikzHeight{\textwidth*0.3}
    \inputtikzorpdf{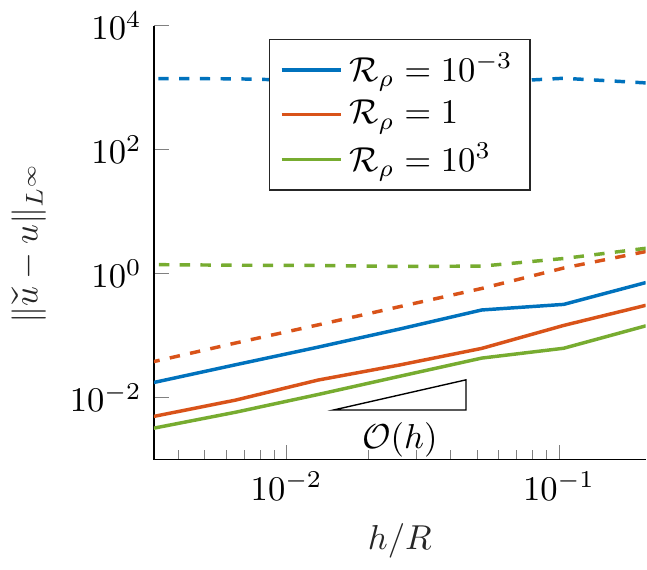}
  }
\end{figure}
We consider a uniform mesh for which $h = 2^{-L}$, for $L = 4, 5, \ldots, 10$, and show the resulting relative error in the $L^\infty$-norm of the velocity field, for several values of the density ratio $\ratio{\rho}$, in~\cref{fig:poisson:validation:poisson:convergence_g}.
The results show that the MDGFM is first-order accurate in terms of the gradient, and thus the resulting velocity field, regardless of the density ratio.
The accuracy of the distance function is essential for convergence: we found that using a PLIC reconstruction resulted in an insufficiently accurate distance function which resulted in a lack of convergence for $\ratio{\rho} = 10^{-3}$.
The results obtained here use a piecewise parabolic reconstruction of the interface (the PLVIRA method from~\citet{Remmerswaal2021} was used).

Since the tangential derivative jump given by~\cref{eqn:poisson:validation:poisson:tangentjump} does not vanish for $\ratio{\rho} \neq 1$, we find that the GFM results in a convergent velocity field only if the fluids are of equal density.
This result confirms that the GFM is not suitable for imposing a jump condition of the form~\cref{eqn:poisson:gradjump} and we therefore proceed with using our proposed MDGFM.

\subsection{The conservative nondimensional distance function}
In the derivation of the GFM~\citep{Liu2000} the distance function $\stagger\phi$ is indeed an actual distance function, as illustrated in~\cref{fig:poisson:gfm:distance_fun}.
Besides the local consistency of the gradient operator, global conservation properties are also of interest.
To this end we compute the contribution to linear momentum by the (MD)GFM gradient operator
\begin{equation}\label{eqn:poisson:momentum_cont}
  \sum_\setextrusion{F} \+n\mean{\volfracstag \rho \sgrad{}} = \sum_\setextrusion{F} \+n \frac{\mean{\volfracstag \rho} (\gradh p + \zeta \gradh \chi^l) + {\mean{\volfracstag\rho \gfmthing}\xi}}{\mean{\stagger\phi\rho}}.
\end{equation}
Whereas usually, when considering the \onefluid formulation in the absence of a value jump, we would find
\begin{equation}\label{eqn:poisson:momentum_usual}
  \sum_\setextrusion{F} \+n \gradh p \stackrel{\eqref{eqn:notation:adjointness}}{=} -\sum_{\mathcal{C}} \divh\+n = \+0,
\end{equation}
we now find that none of the contributions in the right-hand side of~\cref{eqn:poisson:momentum_cont} vanish.
The fact that we do not end up with a result such as~\cref{eqn:poisson:momentum_usual} is due to the fact that the GFM uses a definition of mass, given by $\mean{\stagger\phi \rho}$, which is different from the definition of mass used by the convection operator, which was given by $\mean{\volfracstag\rho}$.
Therefore we propose to use a novel definition of distance, given by $\stagger\phi^\pi = \volfracstag^\pi$, such that the contribution to linear momentum is now given by
\begin{equation}
  \sum_\setextrusion{F} \+n \mean{\volfracstag^{(n+1)}\rho(u^{(n+1)} - u^{***})} = -\dt \sum_\setextrusion{F} \+n \mean{\volfracstag \rho \sgrad{}} = -\dt \sum_\setextrusion{F} \+n \zeta \gradh\chi^l. 
\end{equation}
Hence only the value jump $\zeta$ (potentially) contributes to the change in linear momentum.

Simply using a different distance function affects the approximation accuracy, however we note that~\cref{eqn:intro:normal_smoothness} is still consistently imposed, and this was exactly the purpose of developing the MDGFM.
We will use this alternative choice for the distance function, and we refer to the resulting methods as the conservative GFM (CGFM) and the conservative MDGFM (CMDGFM) respectively.

To summarise, the CMDGFM gradient operator is given by
\begin{equation}
  \sgrad{\pi} = \frac{\gradh p + \zeta \gradh\chi^l + \gfmthing^\pi \xi}{\mean{\volfracstag \rho}}, \quad
  \gfmthing^l = -\volfracstag^g \rho^g, \quad 
  \gfmthing^g = \volfracstag^l \rho^l,
\end{equation}
where the jump $\xi$ is found by imposing~\cref{eqn:poisson:jump_interpolant:jump_interpolant_beta}, which results in~\cref{eqn:poisson:mdgfm:impose_gradient_jump_inverted_bounded_beta}, with $\beta \in \mathcal{F}^h$ given by~\cref{eqn:poisson:mdgfm:betaasfun,eqn:poisson:mdgfm:betafundef}.

\subsection{Comparison to the \onefluid formulation}
We now discuss the discretisation of the Poisson problem for the \onefluid formulation, as given by~\cref{eqn:overview:poisson:one}.
For the \onefluid formulation we use the mass weighted average of the CGFM gradient operator, which we denote as
\begin{equation}
  \sgradone \defeq \frac{\mean{\volfracstag \rho \sgrad{}}}{\mean{\volfracstag \rho}} = \frac{\gradh p + \zeta \gradh\chi^l}{\mean{\volfracstag \rho}},
\end{equation}
and no longer depends on the gradient jump $\xi$.
Using this gradient operator in~\cref{eqn:overview:poisson:one} results in the following velocity update
\begin{equation}\label{eqn:poisson:onefluid:onefluid}
  \onevelo^{(n+1)} = \onevelo^{***} - \dt \frac{\gradh p + \zeta \gradh\chi^l}{\mean{\volfracstag \rho}}.
\end{equation}

Interestingly, when using $\xi = \jump{u}^{***}/\dt$ in the \twofluid formulation, which results in the complete removal of the velocity jump~\citep{Desjardins2010,Vukcevic2017}, we find that the velocity field after the projection step is given by
\begin{equation}\label{eqn:poisson:onefluid:twofluid_removejump}
  u^{\pi,(n+1)} = u^{\pi,***} - \dt\frac{\gradh p + \zeta \gradh\chi^l + \gfmthing^\pi \jump{u^{***}}/\dt}{\mean{\volfracstag \rho}}  = \frac{\mean{\volfracstag \rho u^{***}}}{\mean{\volfracstag \rho}} - \dt\frac{\gradh p + \zeta \gradh\chi^l}{\mean{\volfracstag \rho}},
\end{equation}
where we emphasise that, due to the removal of the jump, the right-hand side no longer depends on the phase $\pi$.
Note that the right-hand side of~\cref{eqn:poisson:onefluid:twofluid_removejump} exactly coincides with the \onefluid formulation given in~\cref{eqn:poisson:onefluid:onefluid}, if we recall that the velocity $\onevelo$ was defined in terms of the velocities $u^\pi$ as
\begin{equation}\label{eqn:poisson:compare_to_one:onevelo}
  \onevelo \defeq \frac{\mean{\volfracstag \rho u}}{\mean{\volfracstag \rho}}.
\end{equation}
Furthermore, the numerical models for diffusion and gravity, which will be discussed in~\cref{sec:diffusion,sec:gravity} respectively, do not alter the jump.

It follows that the \onefluid formulation can alternatively be obtained by simply letting $\xi = \jump{u}^{***}/\dt$ within the \twofluid formulation.
The difference being that usually we first impose continuity of the velocity field $\jump{\+u} = 0$ and subsequently manipulate the analytical equations \emph{before} discretisation, whereas now we see that we can also interpret the \onefluid model as a special instance of the \twofluid model, where we have substituted $\jump{u} = 0$ \emph{after} discretisation.
This is a very important property, since, when comparing our proposed one- and \twofluid models in~\cref{sec:results}, it is exactly the influence of only the jump condition, as given by~\cref{eqn:intro:normal_smoothness}, that we want to consider.
Such an approach is in the same vein as~\citet{Veldman1990}: \emph{discretise first, substitute next}.

  \section{Diffusion}\label{sec:diffusion}
Thus far we have addressed~\cref{chal:mass_mom,chal:continuity} and therefore we now have a sharp model for the two-phase Euler equations. 
We however do not completely want to neglect viscosity, and in particular if the mesh is sufficiently fine, we want our \twofluid model to converge to the viscous \onefluid model which has a continuous velocity field (as given by~\cref{chal:refinement}).
This will be achieved by the diffusion model, but before we discuss this, we will introduce the diffusion model for the \onefluid formulation.

\subsection{\Onefluid formulation}
For the discretisation of diffusion we choose to use a diffuse (i.e. not sharp) approach, where we simply let the dynamic viscosity be a weighted (using the volume fraction) average of the dynamic viscosities per phase, and subsequently use standard single-phase operators for the discretisation of the strain tensor as well as divergence operator.
Using the implicit midpoint rule for time integration results in the following stress tensor $\tvar^\mu$ (see also~\cref{eqn:overview:diff:one})
\begin{equation}\label{eqn:diffusion:stress_tensor}
  \tvar^\mu \defeq \bar\mu^{(n+1)} \symmgrad\roundpar{\onevelo^* + \onevelo^{**}}, \quad \symmgrad u \defeq \frac{\stagger \gradh u + (\stagger \gradh u)^T}{2},
\end{equation}
where the gradient and strain\footnote{The diagonal components of the strain tensor are located on the centroids of the centred control volumes, whereas the off-diagonal components are located on the edges of the centred control volumes.
In particular, the $i,j$-th component, for $i \neq j$, is located on the edge $e$ whose tangent $\+t_e$ is normal to both the $i$-th as well as $j$-th coordinate direction.
It follows that the definition of the symmetric part of the gradient tensor $\stagger \gradh u$, resulting in the strain tensor $\symmgrad u$, does not require interpolation (on a rectilinear mesh) since the $i,j$-th component is collocated with the $j,i$-th component.} operator are denoted by $\stagger \gradh, \symmgrad: \mathcal{F}^h \rightarrow \mathcal{G}^h$, and we remark that $(\stagger \gradh u)^T$ is not to be confused with the operator adjoint $\stagger \gradh^T$.
The gradient operator is defined as the negative adjoint of $\stagger \divh$: $\stagger \gradh \defeq -\stagger \divh^T$
\begin{equation}\label{eqn:notation:notation:gradient_stag}
  (\stagger \gradh u)_g = -\frac{1}{\stagger h_g}\sum_{f \in \mathcal{F}^\omega(g)} \stagger\orientation_{f,g} u_f,
\end{equation}
where $\mathcal{F}^\omega(g)$ denotes the set of neighbouring staggered control volumes which have $g$ as part of their boundary and $\stagger h_g$ is such that $\stagger \gradh$ is exact for linear functions.

The volume fraction at the edges of a control volume are interpolated from the centred volume fractions, thereby allowing us to compute the spatially varying dynamic viscosity $\bar \mu \in \mathcal{G}^h$ as the weighted geometric average of the dynamic viscosity per phase
\begin{equation}\label{eqn:diffusion:generic_mean}
  \bar\mu \defeq (\mu^l)^{\volfrac^{l}} (\mu^g)^{\volfrac^{g}}.
\end{equation}
The weighted arithmetic average of the dynamic viscosity was found to lead to continuity of the velocity too quickly under mesh refinement for a viscous test problem using the \twofluid formulation, and we have therefore chosen to use the weighted geometric average instead.
We also use the weighted geometric average in the \onefluid formulation to ensure that any observed differences, between the one- and \twofluid formulations, are in fact due to the velocity jump condition at the interface, and not due to a different viscosity averaging.

\subsection{\Twofluid formulation}
In developing a \twofluid generalisation for the diffusion model we require the following: total (i.e. the sum over the phases) linear momentum should be conserved, and if the velocity field is continuous we want to recover exactly the same diffusion model as used for the \onefluid formulation.

A simple way to ensure that both of these conditions hold is to let the diffusion operator act on the equivalent velocity field one would obtain in the \onefluid model.
That is, we define (cf.~\cref{eqn:poisson:compare_to_one:onevelo})
\begin{equation}
  \onevelo^{*} \defeq \frac{\mean{\volfracstag^{(n+1)} \rho u^{*}}}{\mean{\volfracstag \rho}^{(n+1)}},
\end{equation}
and compute $\onevelo^{**}$ using the implicit midpoint rule as given by~\cref{eqn:overview:diff:one}, where $\tvar^\mu$ is defined by~\cref{eqn:diffusion:stress_tensor}.
Provided with the stress tensor $\tvar^\mu$, the velocity $u^{\pi,**}$ can be computed according to~\cref{eqn:overview:diff:two}.

This model is not based on consistent and sharp modelling of the diffusive stresses at the interface, but does result in the desired behaviour: the velocity discontinuity vanishes under mesh refinement for viscous flows. 
This will be demonstrated in~\cref{sec:results:stokes}.

  \section{Gravity}\label{sec:gravity}
The nonconservative gravity force (where $\chi^\pi(\+x) \in \{0, 1\}$ is the phase indicator, and $\+g$ the gravitational acceleration)
\begin{equation}\label{eqn:gravity:force}
  \+F = \mean{\chi \rho} \+g,
\end{equation}
is deceptively simple, but its discretisation is nontrivial nonetheless.
Besides consistency, well-balancedness is of importance, which refers here to the existence of a steady state solution.
That is, we consider the existence of some pressure $p^\dagger$, such that
\begin{equation}
  \gradient p^{\dagger} = \+F.
\end{equation}
Here we consider a steady state in the absence of surface tension, and thus consider a linear interface given by the zero level set of
\begin{equation}\label{eqn:gravity:levelset}
  q(\+x) = \+N \cdot \+x - S, \quad \+N = -\frac{\+g}{\abs{\+g}_2},
\end{equation}
where we permit the gravitational acceleration $\+g$ to be in any direction.
It follows that the following pressure 
\begin{equation}\label{eqn:gravity:steady_pressure}
  p^{\dagger}(\+x) = -\mean{\chi(\+x) \rho} \abs{\+g}_2 q(\+x),
\end{equation}
results in an exact balance between pressure and the gravity force.
A discretisation of the gravity force~\cref{eqn:gravity:force} for which such an exact balance always (for any $\+g$) exists is called well-balanced~\citep{Popinet2018}.

Note that the pressure given by~\cref{eqn:gravity:steady_pressure} is continuous, despite the discontinuous density, since $q$ vanishes at the interface.

\subsection{Overview}
Integration over a staggered control volume of the face normal component of the gravity force~\cref{eqn:gravity:force} results in
\begin{equation}\label{eqn:gravity:overview:straight}
  \frac{1}{|\omega_f|}\+n_f \cdot \integral{\omega_f}{\+F}{V} = \mean{\volfracstag \rho}_f \+n_f \cdot \+g,
\end{equation}
where we make use of $|c|\volfrac^{l}_c = \integral{c}{\chi^l(\+x)}{V}$ (cf.~\cref{eqn:overview:fractiondef}).
This is the most straightforward discretisation of the gravity force.
It turns out, however, that this model is ill-balanced (i.e. not well-balanced) whenever the interface is not aligned with the grid.

The model can be made well-balanced by modifying the definition of mass that appears inside the gravity force $\+F$~\citep{Wemmenhove2015a}.
The resulting `gravity consistent averaging' of the density is such that the discrete equivalent of $\curl \+F = \+0$ holds in a steady state, which implies that $\+F = \gradient p^*$ for some scalar potential $p^*$. 
Modification of the mass however implies that the mass appearing in the pressure gradient differs from the one used in the definition of momentum, and thus linear momentum is no longer conserved by the pressure gradient, which is undesirable.

Gravity can also be modelled by analytically subtracting the hydrostatic pressure from the pressure~\citep{Popinet2018}
\begin{equation}
  P^\pi = p^\pi - \rho^\pi \+g \cdot \+x,
\end{equation}
which results in a modified jump condition on the newly defined reduced pressure
\begin{equation}\label{eqn:gravity:reduced_pressure}
  \jump{P} = \jump{p} - \jump{\rho} \+g \cdot \+x = -\sigma\kappa - \jump{\rho} \+g \cdot \+x, \quad \+x \in I.
\end{equation}
This jump condition can then be taken into account as a value jump in the C(MD)GFM method.
This model is well-balanced~\citep{Popinet2018} and reproduces the steady state~\cref{eqn:gravity:steady_pressure} exactly provided that the globally linear interface is exactly reconstructed.

Each of the three methods (firstly~\cref{eqn:gravity:overview:straight}, secondly the gravity consistent method from~\citet{Wemmenhove2015a} and thirdly the reduced pressure approach given by~\cref{eqn:gravity:reduced_pressure}) are consistent, but the first lacks well-balancedness, and all of them lack conservation properties.
In what follows we therefore propose a new well-balanced and mimetic gravity model, which yields the correct contribution to linear momentum, as well as a correct exchange between kinetic and gravitational potential energy.

\subsection{A mimetic gravity model}
We propose a new gravity model based on the consistent evolution of the gravitational potential energy density $\gravpot \in \mathcal{C}^h$, which we define as
\begin{equation}\label{eqn:gravity:mimetic:potential}
  \gravpotc \defeq -\frac{\+g \cdot \mean{\rho \+M_{1,c}}}{|c|},
\end{equation}
where $\+M_{1,c}^\pi$ is the first moment of the $\pi$-phase
\begin{equation}
  \+M_{1,c}^\pi \defeq \integral{c^\pi}{\+x}{V}.
\end{equation}

We consider the \onefluid formulation, and impose that the exchange of kinetic and gravitational potential energy balances exactly. 
That is, we impose
\begin{equation}\label{eqn:gravity:mimetic:energy_balance}
  \frac{d}{dt} \sum_{\mathcal{C}} \gravpot + \sum_\setextrusion{F} \onevelo F = 0,
\end{equation}
where we recall that the integration functionals are defined in~\cref{eqn:notation:integration}.
Subsequently we derive an expression for the, thus far unknown, gravity model $F$.

It follows that we must model the temporal evolution of the gravitational potential energy density, given by~\cref{eqn:gravity:mimetic:potential}, for which the only time dependent part is given by the first moment.
The time derivative of the first moment is given by
\begin{equation}\label{eqn:gravity:mimetic:moment_evolution}
   \frac{d}{dt} \+M_{1,c}^\pi = \frac{d}{dt} \integral{c^\pi}{\+x}{V} = \integral{I_c}{\+x u_{\eta^\pi}}{S},
\end{equation}
where Reynolds' transport theorem~\citep{Reynolds1903} was used.
Subsequently, we approximate the interface integral on the right-hand side by
\begin{equation}\label{eqn:gravity:mimetic:rhs_approx}
  \integral{I_c}{\+x u_{\eta^\pi}}{S} \approx\+x^I_c \integral{I_c}{u_{\eta^\pi}}{S},
\end{equation}
where $\+x^I \in [\mathcal{C}^h]^d$ denotes the centroid of the interface within the corresponding control volume.
If the control volume $c$ does not contain part of the interface then $\+x^I_c$ equals the centroid of the control volume instead: $\+x^I_c = \+x_c$ if $I_c = \emptyset$.
For a solenoidal velocity field we find that
\begin{equation}\label{eqn:gravity:mimetic:divergence_approx}
  \integral{I_c}{u_{\eta^\pi}}{S} = -\integral{\partial c^\pi \setminus I_c}{u_\eta}{S} \approx -|c|\divh(a^\pi \onevelo)_c,
\end{equation}
where the divergence operator $\divh$ is as defined in~\cref{eqn:notation:divergence}.
Combining~\cref{eqn:gravity:mimetic:moment_evolution,eqn:gravity:mimetic:rhs_approx,eqn:gravity:mimetic:divergence_approx} shows that we can model the temporal evolution of the first moment in the following way
\begin{equation}\label{eqn:gravity:mimetic:moment_model}
  \frac{d}{dt} \+M_{1,c}^\pi = -|c|\divh(a^\pi \onevelo)_c\+x^I_c.
\end{equation}

Integration of the time derivative of~\cref{eqn:gravity:mimetic:potential}, using the approximation given by~\cref{eqn:gravity:mimetic:moment_model}, results in the rate of change of the total gravitational potential energy
\begin{equation}
  \frac{d}{dt} \sum_{\mathcal{C}} \gravpot = \sum_{\mathcal{C}}  \divh(\mean{a \rho}\onevelo) \+g \cdot\+x^I = -\sum_\setextrusion{F}  \mean{a \rho} \onevelo \gradh(\+g \cdot\+x^I),
\end{equation}
where we have utilised the skew adjointness relation between the divergence and gradient operator, see~\cref{eqn:notation:adjointness}.
At this point we impose~\cref{eqn:gravity:mimetic:energy_balance} from which we define the mimetic gravity model (MGM) as
\begin{equation}\label{eqn:gravity:model}
  F \defeq \mean{a \rho} \gradh(\+g \cdot\+x^I).
\end{equation}

Well-balancedness of the MGM is stated in the following lemma for which a proof can be found in~\cref{sec:app:gravity}.
\begin{restatable}[Well-balancedness of the MGM]{lemma}{lemmagravitybalance}\label{lem:app:gravity:primal}
  If the interface is given by the zero level set of the function $q$ defined in~\cref{eqn:gravity:levelset} then the MGM, as given by~\cref{eqn:gravity:model}, is well-balanced
  \begin{eqnarray}\label{eqn:app:gravity:steady_state}
    \gradh p^\dagger = F,
  \end{eqnarray}
  where the steady state pressure is given by
  \begin{eqnarray}\label{eqn:app:gravity:mgm_sol}
    p^\dagger = -\mean{\chi \rho} \abs{\+g}_2 q(\+x^I).
  \end{eqnarray}
\end{restatable}

Note that the derivation ensures that the energy exchange is exact in semi-discrete form, and in~\cref{sec:validation:gravity} we show that by mimicking the correct energy exchange we obtain a much improved exchange between gravitational potential and kinetic energy, when compared to the gravity model based on the reduced pressure.
In~\cref{sec:app:gravity} we moreover show that the MGM yields the correct contribution to linear momentum.
 
  \section{Validation}\label{sec:results}
We will now validate the proposed one- and \twofluid models by applying them to a multitude of two-phase flow problems.
In particular, the well-balancedness of the models is demonstrated in~\cref{sec:results:balance}, whereas the simulation of a capillary wave, a Kelvin--Helmholtz instability and a Rayleigh--Taylor instability is considered in~\cref{sec:results:cap_wave,sec:results:kh,sec:results:rt} respectively.
Subsequently we consider a Plateau--Rayleigh instability in~\cref{sec:results:pr}, followed by the simulation of breaking waves in~\cref{sec:results:stokes,sec:results:impact}.
Both formulations have been implemented in the ComFLOW~\citep{kleefsman2005volume,Wemmenhove2015a,VanderPlas2017} two-phase Navier--Stokes solver.
A block-based adaptive grid is used, where we consider several types of adaptive mesh refinement (AMR).
The simplest form of mesh refinement that we consider refines only near the phase interface: a mesh block $\mathcal{B} \subset \mathcal{C}$ is refined (i.e. every control volume $c$ is split into $2^d$ control volumes) if this block contains both phases
\begin{equation}
  \max_{c\in\mathcal{B}}\volfrac^\pi_c \ge \varepsilon^{\volfrac}, \quad\text{for } \pi = l \text{ and } \pi = g,
\end{equation}
where e.g. $\varepsilon^{\volfrac} = 10^{-6}$.
If this criterion no longer holds true for some block $\mathcal{B}$ then it is coarsened.
We refer to this type of mesh refinement as $\volfrac\amr{\level}$-AMR, where $\level$ is the maximum refinement level of this criterion on top of some underlying uniform grid.
The underlying uniform grid has mesh width $\basemesh$, and the maximum refinement level results in $h = 2^{-\level}\basemesh$.
Whenever blocks are refined and/or coarsened it is ensured that the difference in refinement level between neighbouring blocks is at most one.
More advanced refinement criteria will be introduced later on when they are needed.

For validation we will frequently compare (analytical) reference solutions to approximate numerical solutions, and to this end we will denote an approximation of some quantity $\varphi$ as $\approximate{\varphi}$.
We will often study convergence under mesh refinement $h \rightarrow 0$, where time step refinement $\dt \rightarrow 0$ is implied by the time step constraints~\cref{eqn:overview:onefluid:dr:cfl,eqn:overview:grav_timestep,eqn:overview:st_timestep} discussed in~\cref{sec:overview}.

The density and dynamic viscosity ratio are denoted by
\begin{equation}
  \ratio{\rho} \defeq \frac{\rho^g}{\rho^l}, \quad \ratio{\mu} \defeq \frac{\mu^g}{\mu^l}.
\end{equation}

\subsection{Well-balancedness}\label{sec:results:balance}
We start by confirming that the \twofluid model has the same steady state solutions as the \onefluid model.
The existence of a steady state solution, as discussed in e.g.~\cref{sec:gravity}, does not imply its stability, and therefore such a validation really is nontrivial.

\subsubsection{Surface tension}
Surface tension is modelled using the pressure value jump condition in~\cref{eqn:poisson:mdgfm:gfm}, where the value jump is given by the Young--Laplace equation~\eqref{eqn:poisson:valjump}.
The curvature is approximated using the generalised height-function (GHF) method~\citep{Popinet2009}.
A floating droplet (or circle in 2D) of radius $R$ in the absence of gravity yields a steady state solution to the two-phase Navier--Stokes equations, where the steady state pressure $p^\dagger$ is piecewise constant (as follows from the Young--Laplace equation~\eqref{eqn:poisson:valjump}) and given by
\begin{equation}
  p^\dagger = \chi^l \sigma \frac{d-1}{R}.
\end{equation}
For an exactly computed curvature, given by $\kappa = \frac{d-1}{R}$, we find that $p = p^\dagger$ yields an exact balance.
Since our curvature is approximated using the GHF method, we don't expect an immediate and exact balance.
However, via viscous dissipation an exact balance can possibly be reached, but this depends on the numerical stability of the steady state.

We simulate a floating droplet using $\volfrac\amr{2}$-AMR with an interface resolution of $h / R \approx 1/13$, and compute the kinetic energy as function of time
\begin{equation}
  E_k \defeq  \half \sum_\setextrusion{F} \mean{\volfracstag \rho u^2}.
\end{equation}
Several Laplace numbers are considered
\begin{equation}\label{eqn:results:laplaceR}
  \laplacenr = \frac{\sigma \rho^l R}{(\mu^l)^2},
\end{equation}
where the density and dynamic viscosity ratio are given by $\ratio{\rho} = \ratio{\mu}^2 = 10^{-3}$, which ensures that the Laplace number is equal for both phases.
Time is nondimensionalised using the oscillation period $T_\sigma$ for an inviscid droplet of mean radius $R$ with infinitesimal harmonic perturbations of wavenumber $k$.
The oscillation period depends on the dimensionality of the droplet, and is given by~\citep[Art. 273 \& 275]{Lamb1932}
\begin{equation}
  T_{\sigma,\mtext{2D}}(k) = 2\pi \sqrt{\frac{\squarepar{\rho^l + \rho^g}R^3}{\sigma k(k+1)(k-1)}}, \quad
  T_{\sigma,\mtext{3D}}(k) = 2\pi \sqrt{\frac{\squarepar{(k+1) \rho^l + k \rho^g}R^3}{\sigma k(k+1)(k-1)(k+2)}}.
\end{equation}

\begin{figure}
  \centering
  \inputtikzorpdf{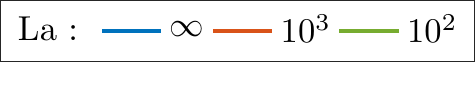}

  \subcaptionbox{A 2D droplet.
  \label{fig:results:floating_droplet2d_r2}}
  [\twofigwidth]{
    \def\tikzWidth{\textwidth*0.325}
    \def\tikzHeight{\textwidth*0.3}
    \inputtikzorpdf{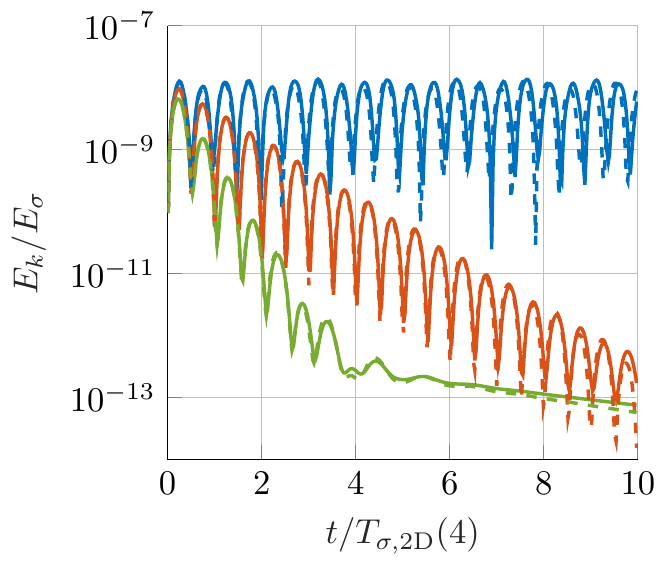}
  }\hfill
  \subcaptionbox{A 3D droplet.
  \label{fig:results:floating_droplet3d_r2}}
  [\twofigwidth]{
    \def\tikzWidth{\textwidth*0.325}
    \def\tikzHeight{\textwidth*0.3}
    \inputtikzorpdf{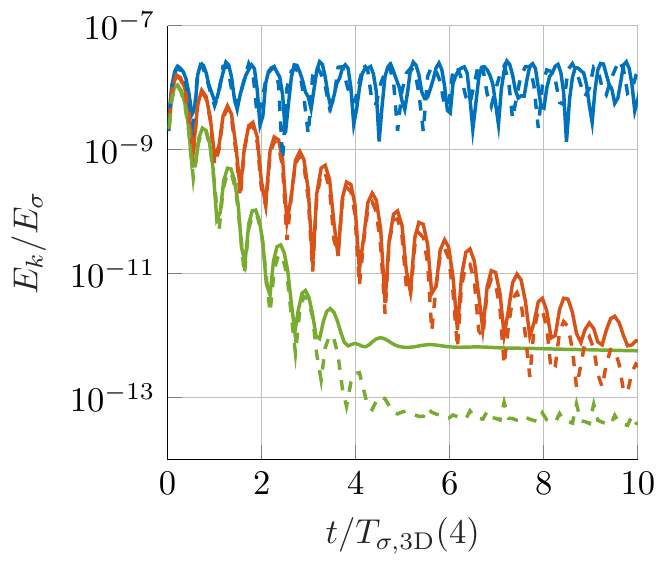}
  }
  \caption{Kinetic energy evolution of a floating droplet with $\ratio{\rho} = 10^{-3}$ for several Laplace numbers using $\volfrac\amr{2}$-AMR with $h / R \approx 1 / 13$.
  The solid and dashed lines correspond to the two- and \onefluid respectively.}
  \label{fig:results:floating_droplet}
\end{figure}
The results are shown in~\cref{fig:results:floating_droplet}, where we have included results from the \onefluid model for comparison.
We note that the ability in reaching a steady state solution is not affected by the new jump condition (given by~\cref{eqn:intro:normal_smoothness}) used in the \twofluid model.
Furthermore, we find stable oscillations of the correct frequency for $\laplacenr = \infty$, where we moreover correctly find that the kinetic energy does not dissipate due to the absence of viscous dissipation.

\subsubsection{Gravity}\label{sec:validation:gravity}
Gravity can be exactly balanced by a hydrostatic pressure, provided that the interface is flat (in the absence of surface tension) and with an interface normal which opposes the direction of gravity.
See also~\cref{fig:results:rotated_wave}, with $\delta_0 = 0$.
\Cref{lem:app:gravity:primal} implies that the MGM exactly yields this balance.
Well-balancedness however does not imply numerical stability of this steady state solution and therefore we test the stability of such a solution using a simulation where the initial condition is \emph{near} the steady state, i.e., we let $\delta_0 > 0$ in~\cref{fig:results:rotated_wave}.

\begin{figure}
  \captionbox{Initial condition of the rotated sinusoidal interface of wavelength $\lambda$ and amplitude $\waveamplitude_0$, as given by~\cref{eqn:results:gravity:initial_interface}.\label{fig:results:rotated_wave}}
  [0.35\textwidth]{
    \centering
    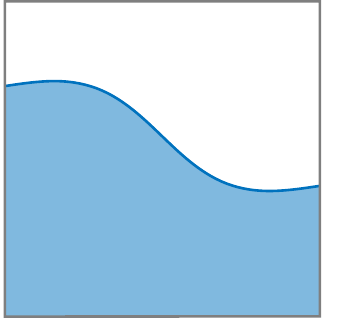
  }\hfill
  \captionbox{Initial condition of the sinusoidal interface of wavelength $\lambda$ and amplitude $\waveamplitude_0$ on a periodic domain, which is used in~\cref{sec:results:cap_wave,sec:results:rt,sec:results:kh}.
  Note that for the Rayleigh--Taylor test case we swap the role of the gas and liquid such that the heavier of the two fluids is on top.
  \label{fig:results:standing_wave}}
  [0.6\textwidth]{
    \centering
    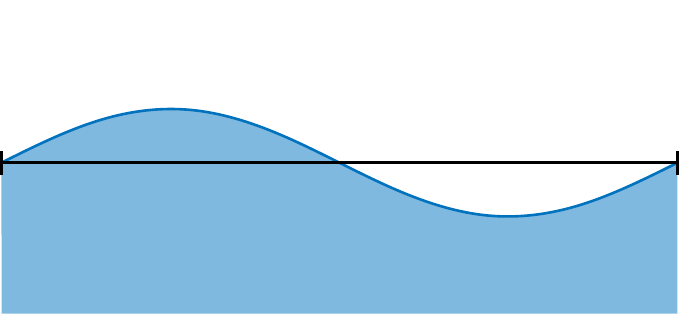
  }
\end{figure}
We consider a square domain $\Omega = [0, 2\pi]^2$ with an initially perturbed fluid interface given by
\begin{equation}\label{eqn:results:gravity:initial_interface}
  \+N \cdot \+x = 4 + \waveamplitude_0 \cos\squarepar{\frac{2\pi \+T \cdot \+x}{\lambda}}, \quad \+T \perp \+N,
\end{equation}
where $\waveamplitude_0 = 10^{-2}$ and $\lambda$ denotes the wave length
\begin{equation}
  \lambda = \frac{2\pi}{\sin\vartheta}.
\end{equation}
The angle $\vartheta = 3\pi/8$ determines the normal direction of the unperturbed interface, see also~\cref{fig:results:rotated_wave}.
The simulations are performed using $\volfrac\amr{2}$-AMR with $h / \lambda \approx 1/35$.
We consider several values of the Galilei number
\begin{equation}
  \galilei = \frac{\abs{\+g}_2 \lambda^3 (\rho^l)^2}{(\mu^l)^2},
\end{equation}
where we let the density and dynamic viscosity ratio be given by $\ratio{\rho} = \ratio{\mu} = 10^{-3}$, such that the Galilei number is equal for both phases.

\begin{figure}
  \subcaptionbox{Evolution of kinetic energy for an initially perturbed interface for the one- and \twofluid model, corresponding to the dashed and solid lines respectively.\label{fig:poisson:results:gravity_balance_kinetic_N32}}
  [\twofigwidth]{
    \def\tikzWidth{\textwidth*0.3}
    \def\tikzHeight{\textwidth*0.3}
    \inputtikzorpdf{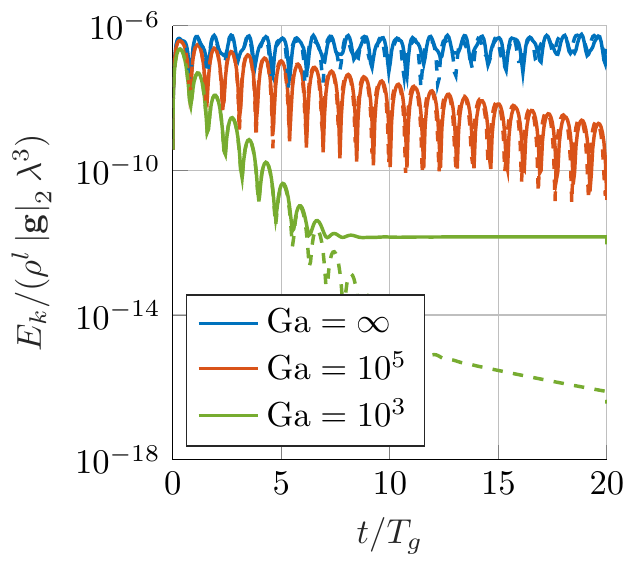} 
  }\hfill
  \subcaptionbox{Evolution of kinetic as well as gravitiational potential energy for $\galilei = \infty$, where we consider the MGM (solid) and reduced pressure (dashed) gravity models with the \twofluid formulation.\label{fig:poisson:results:gravity_balance_exchange_N32}}
  [\twofigwidth]{
    \def\tikzWidth{\textwidth*0.34}
    \def\tikzHeight{\textwidth*0.3}
    \inputtikzorpdf{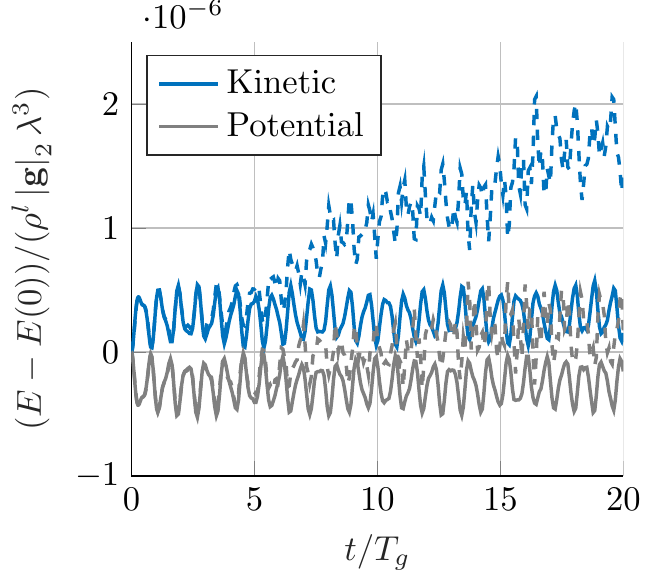} 
  }
  \caption{Energy evolution for testing the well-balancedness of the gravity models.
  Here the interface resolution is $h / \lambda \approx 1 / 35$.}
\end{figure}
The evolution of kinetic energy for several Galilei numbers is shown in~\cref{fig:poisson:results:gravity_balance_kinetic_N32}.
We observe that the results of the one- and \twofluid model hardly differ, except in the ability to reach the steady state for $\galilei = 10^3$.
It is unclear why this is the case.

In~\cref{fig:poisson:results:gravity_balance_exchange_N32} we show the evolution of kinetic as well as gravitational potential energy for the MGM as well as the gravity model based on the reduced pressure (see~\cref{eqn:gravity:reduced_pressure}), for $\galilei = \infty$.
Here we find, even though~\cref{eqn:gravity:mimetic:moment_model} is not satisfied exactly, that the MGM results in a much better exchange of kinetic and gravitational potential energy as when compared to the model based on the reduced pressure.

\subsection{Capillary wave}\label{sec:results:cap_wave}
We consider the motion of the interface between two viscous fluids in the presence of surface tension, for which an analytical solution for the amplitude $\waveamplitude(t)$ is found in~\citet{Prosperetti1981}.
See also~\cref{fig:results:standing_wave}.
The Laplace number is given by
\begin{equation}
  \mtext{La} = \frac{\sigma \rho^l \lambda}{(\mu^l)^2}.
\end{equation}
We vary the Laplace number, $\mtext{La} \in \{3\times 10^3, \infty\}$, and let $\ratio{\rho} = \ratio{\mu} = 1$ (no gravity).
The initial amplitude is taken as $\waveamplitude_0 = 10^{-2} \lambda$ with $\lambda = 2\pi$.
Time will be nondimensionalised using the capillary time scale
\begin{equation}\label{eqn:results:cap_wave:cap_timescale}
  T_\sigma = \sqrt{\frac{(\rho^l + \rho^g) \lambda^3}{2\pi \sigma}}.
\end{equation}

\begin{figure}
  \centering
  \inputtikzorpdf{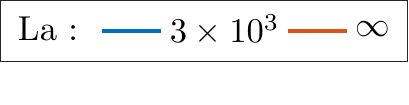}

  \captionbox{Example temporal evolution of the interface amplitude using the \twofluid model (markers) and analytical solution (solid lines) for two values of the Laplace number.
  The resolution at the interface is given by $h = \lambda/64$.
  \label{fig:results:capillary_wave_example_Dr1_N64}}
  [\twofigwidth]{
    \def\tikzWidth{\textwidth*0.325}
    \def\tikzHeight{\textwidth*0.3}
    \inputtikzorpdf{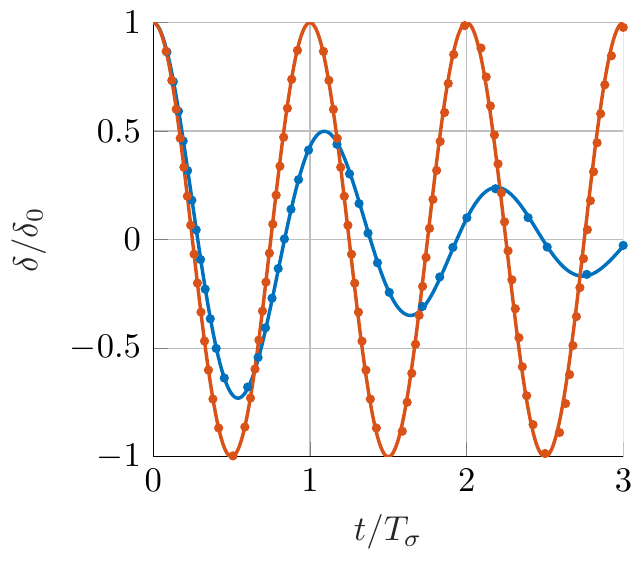}
  }\hfill
  \captionbox{Convergence of the interface amplitude in the nondimensionalised $L^2([0, T])$ norm for La $ = 3\times 10^3, \infty$ and $\ratio{\rho} = \ratio{\mu} = 1$.
  We compare the \twofluid model (solid line) to the \onefluid model (dashed line) as well as a result from~\citet{Popinet2009} (dotted line).
  \label{fig:results:capillary_wave_convergence_Dr1}}
  [\twofigwidth]{
    \def\tikzWidth{\textwidth*0.325}
    \def\tikzHeight{\textwidth*0.3}
    \inputtikzorpdf{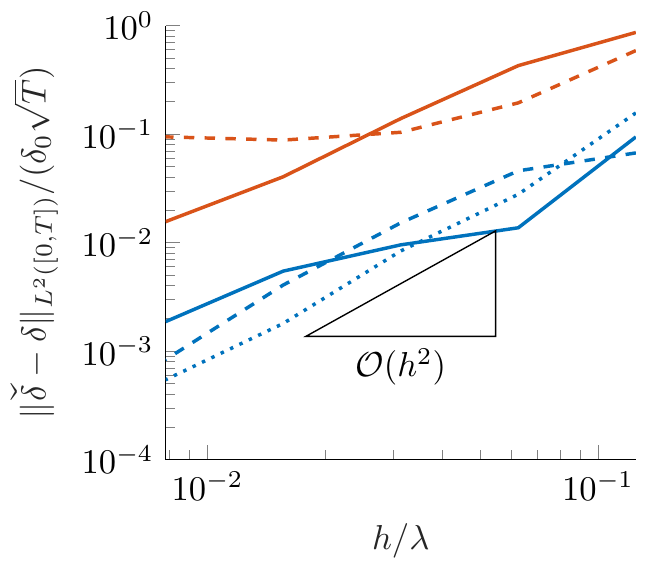}
  }
\end{figure}
We use $\volfrac\amr{L}$-AMR with $\basemesh = 2^{-3}\lambda$.
In~\cref{fig:results:capillary_wave_example_Dr1_N64} we show the temporal evolution of the interface amplitude for the two Laplace numbers, using $L = 3$.
Furthermore, for each of the Laplace numbers, we vary the refinement level $L$ and for each approximate solution we compute the normalised $L^2$-norm (in time) of the error in the amplitude $\waveamplitude$. 
The time interval used in the $L^2$-norm is given by $T = \frac{25T_\sigma}{2\pi}$.
The results are shown in~\cref{fig:results:capillary_wave_convergence_Dr1} where we compare to the \onefluid model as well as a numerical result obtained from~\citet{Popinet2009}.

For the viscous case ($\laplacenr = 3\times 10^3$) we find that the results of the \twofluid model, \onefluid model as well as the result obtained from~\citet{Popinet2009} are of similar, and second-order, accuracy.
To the contrary, the inviscid case ($\laplacenr = \infty$) clearly shows the advantage of using the \twofluid model: here the \onefluid model does not converge because the velocity field is erroneously assumed to be continuous, whereas the \twofluid model still converges with second-order accuracy.

\subsection{Kelvin--Helmholtz instability}\label{sec:results:kh}
We now consider two fluids which initially have a velocity discontinuity tangential to the interface: $U^g = \+e_x \cdot \+u^g, U^l = \+e_x \cdot \+u^l$ with $\jump{U} \neq 0$, in the absence of gravity and viscosity, see also~\cref{fig:results:standing_wave}.
The Weber number, which is a nondimensional measure of the importance of inertia relative to that of capillary forces,  is given by
\begin{equation}
  \weber = \frac{\rho^g \rho^l \jump{U}^2 \lambda}{2\pi \sigma (\rho^g + \rho^l)}.
\end{equation}
For small amplitude perturbations of a given wavenumber $k = \frac{2\pi}{\lambda}$, the evolution of the interface amplitude is as follows~\citep[Art. 267 \& 268]{Lamb1932}
\begin{equation}\label{eqn:results:kh:amplitude}
  \waveamplitude(t) = {\frac{e^{-i\aleph_+ t} + e^{-i\aleph_- t}}{2}}\waveamplitude_0,
\end{equation}
where the frequency $\aleph_\pm$ can be found from a dispersion relation.
This dispersion relation yields an expression for the frequency $\aleph_\pm$ as function of the wavenumber, and is given by (recall that the capillary time scale is given by~\cref{eqn:results:cap_wave:cap_timescale})
\begin{equation}\label{eqn:results:kh:dispersion}
  \aleph_\pm = k \frac{\mean{\rho U}}{\mean{\rho}} \pm \frac{2\pi}{T_\sigma}\sqrt{1 - \weber}.
\end{equation} 
It follows that if $\weber > 1$ we find that $\Im \aleph_\pm \neq 0$ and therefore results in exponential growth of the interface amplitude given by~\cref{eqn:results:kh:amplitude}.
Here $\Re \aleph_\pm$ and $\Im \aleph_\pm$ denote the real and imaginary part of $\aleph_\pm$ respectively.

We use $\volfrac\amr{L}$-AMR with $\basemesh = 2^{-3}\lambda$ and a domain of size $[0, \lambda] \times [-1.5\lambda, 1.5\lambda]$, the initial amplitude is given by $\waveamplitude_0 = 10^{-2}$ and the wavelength is given by $\lambda = 2\pi$.

\subsubsection{Linear regime}
\begin{figure}
  \captionbox{The analytical dispersion relation~\cref{eqn:results:kh:dispersion} (solid lines) compared to the numerical approximation (markers) for $\ratio{\rho} = 1$.
  Here an interface resolution of $h = \lambda/32$ was used.
  \label{fig:results:kelvin_helmholtz_dispersion}}
  [\twofigwidth]{
    \def\tikzWidth{\textwidth*0.35}
    \def\tikzHeight{\textwidth*0.3}
    \inputtikzorpdf{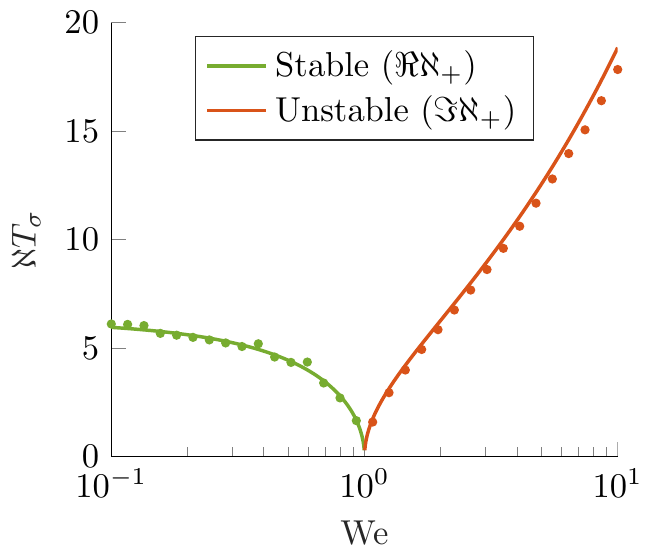}
  }\hfill
  \captionbox{Convergence of the interface amplitude for the Kelvin--Helmholtz problem at $x = 0$ in the nondimensionalised $L^2([0, T])$-norm.
  The density ratio was given by $\ratio{\rho} = 10^{-3}$.
  \label{fig:results:kelvin_helmholtz_convergence_dr1e-3}}
  [\twofigwidth]{
    \def\tikzWidth{\textwidth*0.325}
    \def\tikzHeight{\textwidth*0.3}
    \inputtikzorpdf{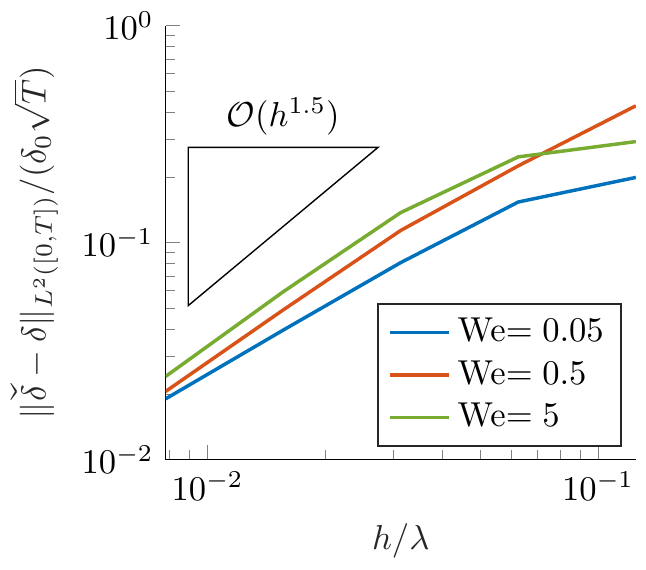}
  }
\end{figure}
We start by considering an equal density ratio $\ratio{\rho} = 1$ and equal, up to sign, initial velocity fields $U^g = \jump{U}/2 = -U^l$, for several values of the Weber number.
This choice of density ratio combined with the velocity field implies that $\mean{\rho U} = 0$ such that $\aleph_+ + \aleph_- = 0$, which is the mode that corresponds to our initially piecewise constant velocity field.
It follows that the wave will not travel, and either oscillates or grows exponentially fast in amplitude.

In~\cref{fig:results:kelvin_helmholtz_dispersion} we show the dispersion relation given by~\cref{eqn:results:kh:dispersion} together with its numerical approximation $\approximate{\aleph}$, using $L = 2$ which results in $h = \lambda / 32$. 
The numerical frequency $\approximate\aleph$ is obtained from the approximate interface amplitude $\approximate{\delta}$, by optimising the following cost function
\begin{equation}
  f(\approximate\aleph) = \left\|\approximate{\waveamplitude(t)} - {\frac{e^{-i\approximate{\aleph}_+ t} + e^{-i\approximate{\aleph}_- t}}{2}}\waveamplitude_0\right\|^2_{L^2([0,T])}.
\end{equation}
Here we let $T = T_\sigma$ if the interface is \emph{found} to be stable, and $T = \log(6) / \Im(\aleph_-)$ if it is not\footnote{This results in $|\waveamplitude(T)| \le \waveamplitude_0 \cosh(\Im \aleph T) \approx 3\waveamplitude_0$ which ensures we are still in the linear regime.}.
Excellent agreement is observed between the numerical and analytical dispersion relations.

In practice our density ratio will not be one, and therefore we now consider a more realistic density ratio of $\ratio{\rho} = 10^{-3}$.
Using the same initial velocity fields we now find $\mean{\rho U} \neq 0$ such that the wave will travel with a non-zero velocity through the periodic domain.
We consider three values of the Weber number, given by $\weber \in \{0.05, 0.5, 5\}$, for which $\weber = 5$ results, both analytically and numerically, in an unstable interface, see~\cref{eqn:results:kh:dispersion}.
The accuracy of the amplitude as function of time is measured by the nondimensionalised $L^2([0,T])$-norm.
Results are shown in~\cref{fig:results:kelvin_helmholtz_convergence_dr1e-3}, which shows that for each of the Weber numbers convergence is obtained under mesh refinement.
We find convergence of order $\mathcal{O}(h^{1.5})$ for all Weber numbers, which shows that the first-order time integration error dominates as $\dt \propto h^{1.5}$, see also~\cref{eqn:overview:st_timestep}.

\subsubsection{Nonlinear regime}
\begin{figure}
  \centering
  \inputtikzorpdf{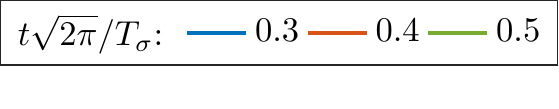}

  \subcaptionbox{Example interface profile (top) and vortex sheet strength $\jump{u_\tau}$ (bottom) using $h = \lambda/256$.
  The solid lines and markers correspond to the approximate and BIM solution respectively.
  \label{fig:results:kh:nonlinear:example}}
  [\twofigwidth]{
    \def\tikzWidth{\textwidth*0.325}
    \def\tikzHeighttop{\textwidth*0.24}
    \def\tikzHeightbottom{\textwidth*0.17}
    \inputtikzorpdf{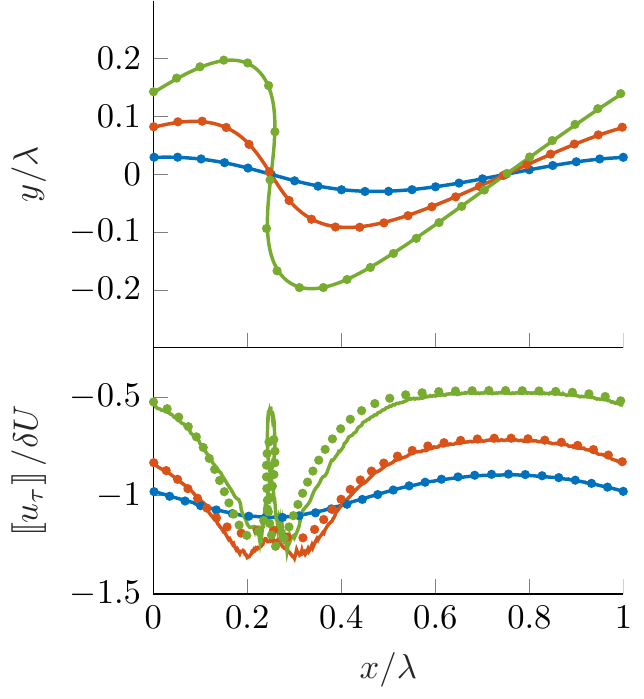}
  }\hfill
  \subcaptionbox{Area of the symmetric difference of the approximate liquid domain compared to the BIM solution at three different time instances.
  \label{fig:results:kh:nonlinear:convergence}}
  [\twofigwidth]{
    \def\tikzWidth{\textwidth*0.325}
    \def\tikzHeight{\textwidth*0.35}
    \inputtikzorpdf{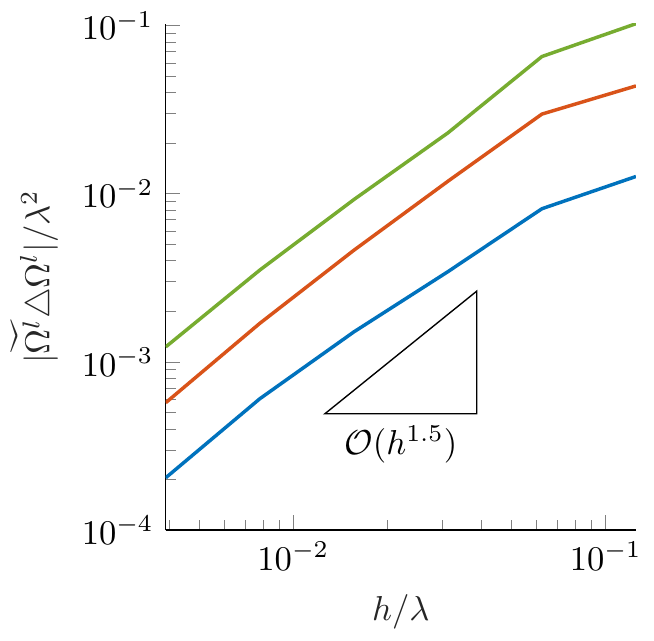}
  }
  \caption{Validation of the nonlinear regime for the Kelvin--Helmholtz problem with $\ratio{\rho} = 1$ and $\weber = \frac{5}{\pi}$.
  The BIM solution is obtained using Method III from~\citet{Baker1998}.}
\end{figure}
When the interface becomes unstable the {linear} analysis does not predict evolution of the interface beyond the stage where the interface amplitude $\waveamplitude(t)$ becomes much larger than the initial amplitude $\waveamplitude_0$.
To this end we have implemented the boundary integral method (BIM) from~\citet{Baker1998} (therein referred to as Method III), which we use for validation of the nonlinear evolution of the Kelvin--Helmholtz instability.
We let $\ratio{\rho} = 1$ and $\weber = \frac{5}{\pi}$.
In~\cref{fig:results:kh:nonlinear:example} we show the interface profile at three time instances, together with the vortex sheet strength (the velocity discontinuity $\jump{u_\tau}$).
The markers correspond to a high-resolution BIM solution. 
We find excellent agreement in both the interface profile as well as the vortex sheet strength.
Furthermore, in~\cref{fig:results:kh:nonlinear:convergence}, we have quantified the convergence of the approximate interface profile to the BIM solution by computing the area of the symmetric difference\footnote{The symmetric difference between two sets $A,B \subset \mathbb{R}^d$ is defined as\begin{equation}
  A \symmdiff B \defeq (A \cup B) \setminus (A \cap B),
\end{equation}
and the area of the symmetric difference is a measure for how close the two sets are.} between the two solutions.
Convergence of order $\mathcal{O}(h^{1.5})$ is observed, which coincides with the previously observed accuracy in the linear regime.

In~\cref{fig:results:kh:nl_pressure} we illustrate the clearly discontinuous pressure field on two time instances.
\begin{figure}
  \captionbox{The pressure field resulting from the Kelvin--Helmholtz problem with $\weber = \frac{5}{\pi}$ and $\ratio{\rho} = 1$, at $t\sqrt{2\pi}/T_\sigma = 0.4$ and $0.5$ respectively.\label{fig:results:kh:nl_pressure}}
  [\twofigwidth]{
    \scalebox{-1}[1]{\includegraphics[scale=.22,trim=0 80 0 80 0,clip=true]{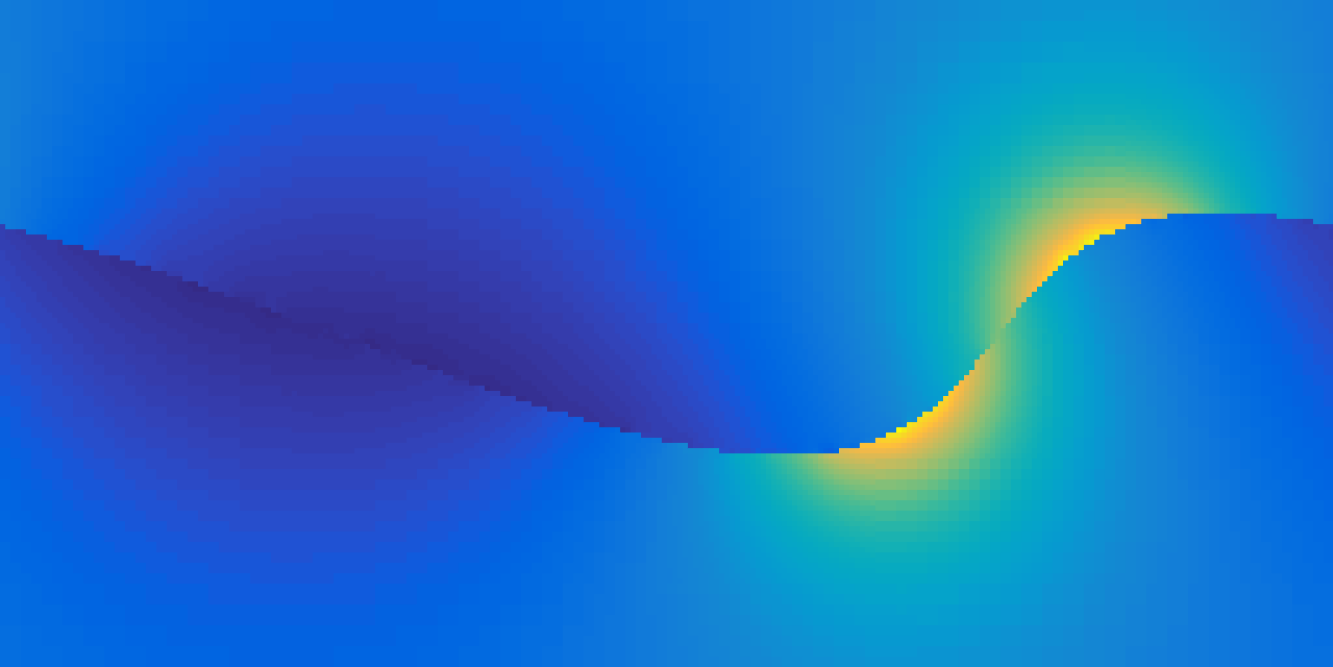}}
    \scalebox{-1}[1]{\includegraphics[scale=.22]{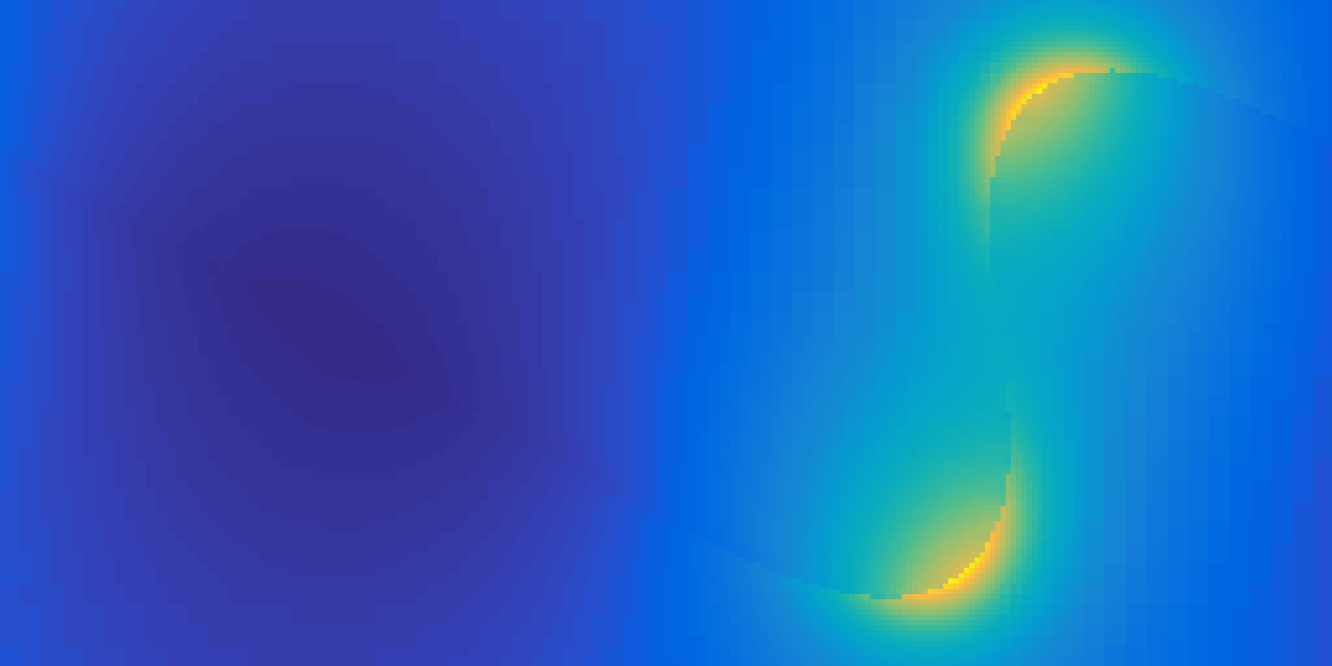}}
  }\hfill
  \captionbox{The analytical dispersion relation~\cref{eqn:results:rt:dispersion} (solid lines) compared to the numerical \twofluid approximation (markers) for $\ratio{\rho} = \{\frac{8}{9}, \frac{1}{3}, 10^{-3}\}$, corresponding to the large, medium and small size markers respectively (which all overlap).
  Here an interface resolution of $h = \lambda/32$ was used.\label{fig:results:rt:dispersion}}
  [\twofigwidth]{
    \def\tikzWidth{\textwidth*0.35}
    \def\tikzHeight{\textwidth*0.3}
    \inputtikzorpdf{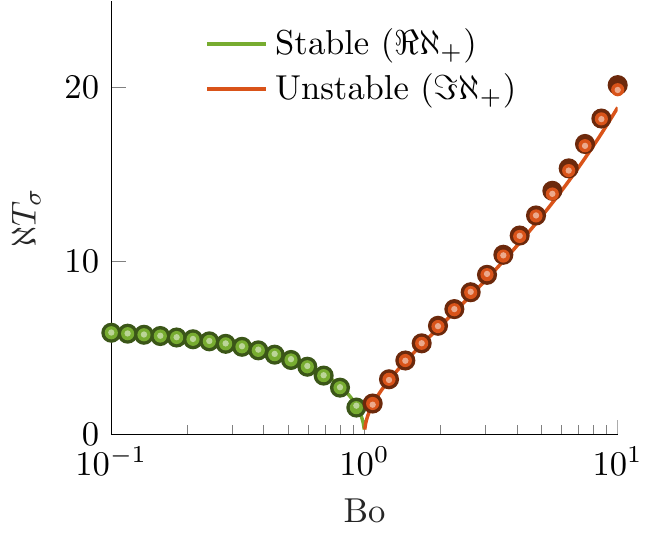}
  }
\end{figure}

\subsection{Rayleigh--Taylor instability}\label{sec:results:rt}
We now consider the effect which gravity has on the stability of the interface.
Again, we consider the motion of the interface between two inviscid fluids, where the interface is initially as shown in~\cref{fig:results:standing_wave}, but now with the heavier liquid on top. 
The dispersion relation for the interface motion in the presence of gravity and surface tension is given by~\citep[Art. 267]{Lamb1932} (cf.~\cref{eqn:results:kh:dispersion})
\begin{equation}\label{eqn:results:rt:dispersion}
  \aleph_\pm = \pm \frac{2\pi}{T_\sigma} \sqrt{1 - \bond},
\end{equation}
where the Bond number is given by
\begin{equation}
  \bond = \roundpar{\frac{T_\sigma}{T_g}}^2 = \frac{(\rho^l - \rho^g)\abs{g}\lambda^2}{4\pi^2\sigma}.
\end{equation}
Hence the gravitational time scale is given by
\begin{equation}\label{eqn:results:rt:gravity_timescale}
  T_g = \sqrt{\frac{\rho^l + \rho^g}{\rho^l - \rho^g}\frac{2\pi\lambda}{\abs{g}}}.
\end{equation}
We find that, according to~\cref{eqn:results:rt:dispersion}, the interface is unstable w.r.t. an infinitesimal perturbation if $\bond > 1$.

\subsubsection{Linear regime}
We validate the dispersion relation given by~\cref{eqn:results:rt:dispersion} for several density ratios given by $\ratio{\rho} \in \{\frac{8}{9}, \frac{1}{3}, 10^{-3}\}$.
The approach we take is identical to that presented in~\cref{sec:results:kh} for obtaining~\cref{fig:results:kelvin_helmholtz_dispersion}.
Results are shown in~\cref{fig:results:rt:dispersion} which shows excellent agreement for all density ratios.

\subsubsection{Nonlinear regime}
\begin{figure}
  \centering
  \inputtikzorpdf{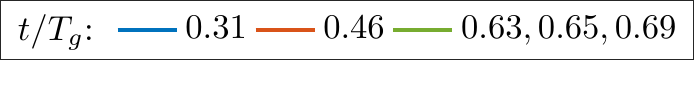}

  \def\rtfigwidth{0.1375\textwidth}
  \def\tikzWidth{\threefigwidth}
  \def\tikzHeighttop{\threefigwidth} 
  \def\tikzHeightbottom{0.67*\threefigwidth}
  \definecolor{thecolor}{rgb}{1.0,1.0,1.0}
  \definecolor{theothercolor}{rgb}{0.00000,0.44700,0.74100}
  \definecolor{mycolor1}{rgb}{0.00000,0.44700,0.74100}%
  \definecolor{mycolor2}{rgb}{0.85000,0.32500,0.09800}%
  \definecolor{mycolor3}{rgb}{0.46600,0.67400,0.18800}%
  \def\thelinewidth{1.}
  \def\themarksize{0.8pt}
  \subcaptionbox{$\ratio{\rho} = \frac{8}{9}$.}
  [0.33333\textwidth]{ 
    \inputtikzorpdf{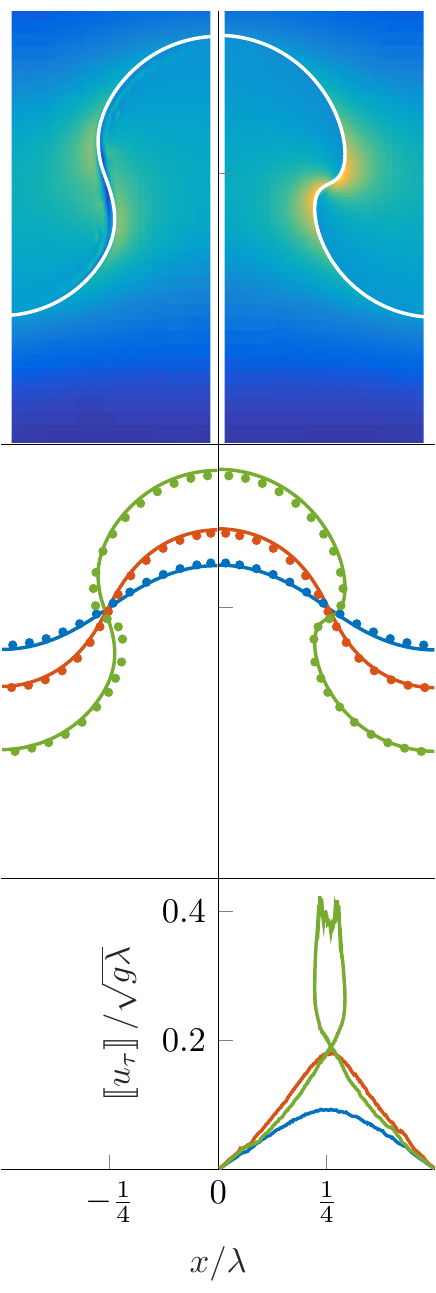} 
  }\hfill
  \subcaptionbox{$\ratio{\rho} = \frac{1}{3}$.}
  [0.33333\textwidth]{
    \inputtikzorpdf{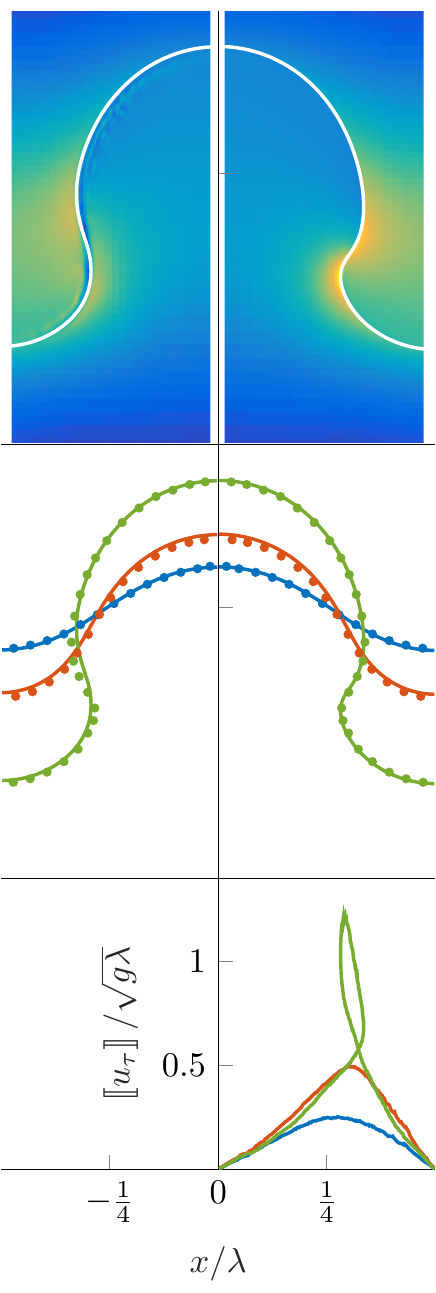}
  }\hfill
  \subcaptionbox{$\ratio{\rho} = 10^{-3}$.}
  [0.33333\textwidth]{
    \inputtikzorpdf{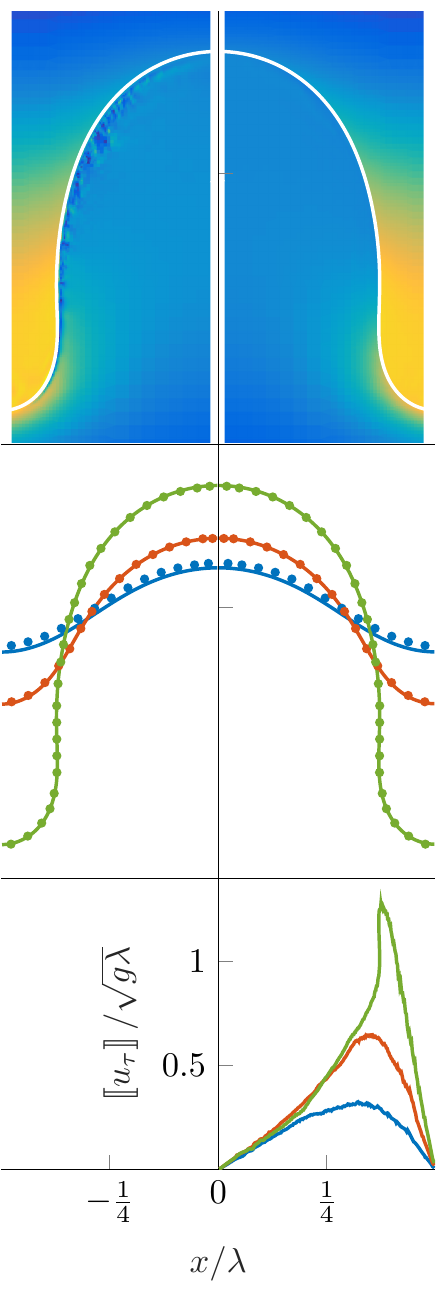}
  }
  \caption{The velocity magnitude at the final time (top), interface profile (middle) and vortex sheet strength $\jump{u_\tau}$ (bottom) for the Rayleigh--Taylor problem for different density ratios.
  The left- and right-hand side results correspond to the one- and \twofluid model respectively (using $\volfrac\amr{4}$-AMR resulting in $h / \lambda = 1/256$), whereas the markers (for the interface profile only) correspond to the reference solutions found in~\citet{Pullin1982}.
  The velocity magnitude at the third and final time instance is clipped to $\abs{\+u}_2 \in [0, 0.5]$, $[0, 1.25]$ and $[0, 1.5]$ for the three density ratios respectively.}
  \label{fig:cd:validation:rt:nonlinear:convergence}
\end{figure}
For the nonlinear regime we compare to numerical results from~\citet{Pullin1982} which were obtained using a BIM. 
Therein $\bond = 3$ and three density ratios $\ratio{\rho} \in \{\frac{8}{9}, \frac{1}{3}, 10^{-3}\}$ were considered.
We use $\volfrac$-AMR with a domain of size $[0, 2\pi] \times [-3 \pi, 3\pi]$ which is sufficiently large for finite-depth effects to be negligible.
The initial amplitude is taken as $\waveamplitude_0 = 10^{-2}$.
The velocity fields are initially zero, which corresponds to the initial velocity field resulting from the sum of the two modes $\aleph_\pm$.

The results are shown in~\cref{fig:cd:validation:rt:nonlinear:convergence} where we compare the approximate interface profile (solid lines) to the BIM solution (markers) at three time instances, for each of the three density ratios. 
We consider the one- and \twofluid solutions in the left- and right-hand side of each of the figures respectively.

For the \twofluid model we find excellent agreement with the reference solution, for all density ratios.
The \onefluid model yields a less accurate interface profile, especially for the density ratio closest to one.
We also show the vortex sheet strength $\jump{u_\tau}$ for the \twofluid solution; unfortunately no reference data was available.

The velocity magnitude at the final time instance is included for each of the density ratios.
This clearly shows the discontinuity in the velocity that has developed in the velocity field of the \twofluid solution.
For the \onefluid model no such discontinuity can develop, and therefore the velocity field shows a transition region where unphysical oscillations are present, in particular for the large density ratio solution. 

\subsection{Plateau--Rayleigh instability}\label{sec:results:pr}
\begin{figure}
  \captionbox{Example interface contours resulting from a Plateau--Rayleigh instability, for $\laplacenr = 23.834$ using $\kappa\amr{6}$-AMR with the \twofluid model. 
  The time corresponds to $t/T_\sigma \approx 14.72, 17.91, 19.02$ from left to right respectively.
  \label{fig:results:plateau_rayleigh:example_global}}
  [\textwidth]{
    \centering
    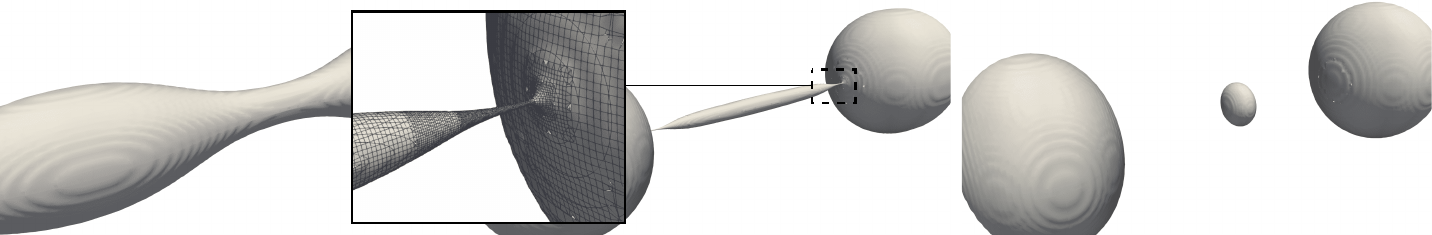
  }
\end{figure}
A liquid jet (a cylindrical column of water), in the absence of gravity, can become unstable due to surface tension and evolve into a sequence of droplets, as shown in~\cref{fig:results:plateau_rayleigh:example_global}.
Such a Plateau--Rayleigh instability happens when the resulting surface area of the droplets is smaller than that of the original liquid jet.

With the Plateau--Rayleigh instability we find that only a small part of the interface is greatly deformed (see~\cref{fig:results:plateau_rayleigh:example_global}), suggesting a localised mesh refinement criterion based on the interface curvature.
The refinement criterion $\kappa\amr{\level}$-AMR is such that the following condition is ensured to hold
\begin{equation}
  \varepsilon_\mtext{coarsen}^{\kappa} \le \max_{c \in \mathcal{B}} h_c |\kappa_c| \le \varepsilon_\mtext{refine}^{\kappa},
\end{equation}
where $\varepsilon_\mtext{coarsen}^{\kappa}$ and $\varepsilon_\mtext{refine}^{\kappa}$ are nondimensional tolerances, and $h_c = |c|^{1/d}$.
This means that if $h_c |\kappa_c| > \varepsilon_\mtext{refine}^{\kappa}$ for at least one $c \in \mathcal{B}$ then the block is refined, and if $h_c |\kappa_c| < \varepsilon_\mtext{coarsen}^{\kappa}$ for all $c \in \mathcal{B}$ then the block is coarsened.
Coarsening will likely double the nondimensional curvature and therefore we must ensure that
\begin{equation}
  2\varepsilon_\mtext{coarsen}^{\kappa} < \varepsilon_\mtext{refine}^{\kappa},
\end{equation}
such that the coarsened block isn't directly eligible for refinement in the next time step.
We use $\varepsilon_\mtext{refine}^{\kappa} = 0.2$ and $\varepsilon_\mtext{coarsen}^{\kappa} = 0.06$.

This test case was also considered in~\citet{Popinet2009}.
Therein the interface along the axis of the cylinder, the $x$-axis, is perturbed according to
\begin{equation}
  r(t, x) = R(1 + \epsilon e^{i(\aleph t + kx)}),
\end{equation}
where $r(x, t)$ denotes the radius of the jet and $k$ denotes the wavenumber of the perturbation.
The dispersion relation for perturbations on a liquid jet, for an inviscid fluid without the presence of a surrounding gas, is given by~\citep{rayleigh1892xvi}
\begin{equation}\label{eqn:st:validation:pr:inviscid}
  (\aleph T_\sigma)^2 + kR\frac{I_1(kR)}{I_0(kR)}\roundpar{1 - (kR)^2} = 0,
\end{equation}
where $I_n$ are the (non-negative) modified Bessel functions of the first kind.
It follows that the jet is stable w.r.t. perturbations of wavenumber $k$ provided that
\begin{equation}
  kR \ge 1 \quad \iff \quad \lambda \le 2\pi R.
\end{equation}
That is, the inviscid jet is stable if the wavelength of the perturbation is at most the circumference of the jet.
For finite values of the Laplace number (defined as in~\cref{eqn:results:laplaceR}) the following approximate generalisation of~\cref{eqn:st:validation:pr:inviscid} was proposed in~\citet{Eggers2008}
\begin{equation}\label{eqn:st:validation:pr:viscous}
  (\aleph T_\sigma)^2 + \half (kR)^2 \roundpar{1 - (kR)^2} + (\aleph T_\sigma) i(kR)^2 {\laplacenr}^{-\half} = 0.
\end{equation}

We use the \twofluid model with the fluid properties as given in~\citep{Popinet2009}. 
Importantly, the Laplace number is given by $\laplacenr = 23.834$ and $kR = 0.2\pi$ such that the liquid jet should become unstable.
We nondimensionalise time using the capillary timescale, which is given by
\begin{equation}
  T_\sigma = \sqrt{\frac{\rho^l R^3}{\sigma}}.
\end{equation}
The temporal evolution of the interface is shown in~\cref{fig:results:plateau_rayleigh:example_global}.
Indeed we find that the liquid jet becomes unstable, and eventually results in the pinch-off of a sequence of droplets.
We find that the $\kappa$-AMR refinement criterion successfully refines the interface where needed, resulting in an efficient refinement strategy.
Furthermore, the results illustrate that topological changes of the interface are no problem for the jump conditions discussed in~\cref{sec:poisson}.

\begin{figure} 
  \captionbox{The nondimensional amplitude of the perturbation as function of time for three Laplace numbers.
  We compare our \twofluid $\kappa(6)$-AMR solution (solid line) to theoretical estimates (markers).
  \label{fig:results:plateau_rayleigh:amplitude}}
  [0.4\textwidth]{
    \def\tikzWidth{\textwidth*0.28}
    \def\tikzHeight{\textwidth*0.3}
    \inputtikzorpdf{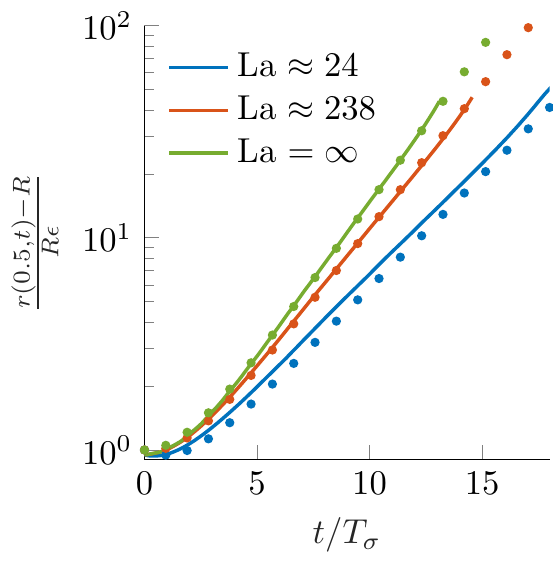}
  }\hfill
  \captionbox{The neck thinness (the minimal neck thickness) as function of $\timpact = (t^\mtext{breakup} - t) / T_\sigma$.
  We compare our \twofluid $\kappa(6)$-AMR solution (solid line) to a numerical result from Gerris~\citep{Popinet2009} (dashed line) as well as theoretical slopes for the viscous ($r / R < 1 / \mtext{La}$) and capillary ($r / R > 1 / \mtext{La}$) dominating regime.
  \label{fig:results:plateau_rayleigh:neck_thinness}}
  [0.55\textwidth]{
    \def\tikzWidth{\textwidth*0.4}
    \def\tikzHeight{\textwidth*0.3}
    \inputtikzorpdf{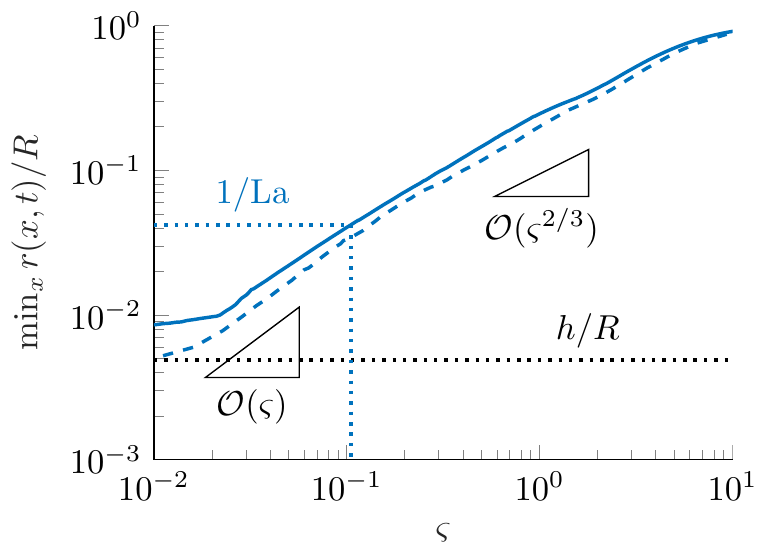} 
  }
\end{figure} 
The accuracy of the solution is quantified by comparing the growth of the perturbation amplitude to the theoretical dispersion relations~\cref{eqn:st:validation:pr:inviscid,eqn:st:validation:pr:viscous}. 
Results are shown in~\cref{fig:results:plateau_rayleigh:amplitude}, where excellent agreement is found for all three Laplace numbers.
Furthermore, in~\cref{fig:results:plateau_rayleigh:neck_thinness} we show the minimal radius of the jet as function of the time before breakup $\timpact$ for $\laplacenr = 23.834$. 
We also show a reference solution from~\citep{Popinet2009}, and theoretical slopes in the regions where capillarity ($r / R > 1 / \laplacenr$) and viscosity ($r / R < 1 / \laplacenr$) dominate~\citep{Popinet2009}.
For the pinch-off process we also find excellent agreement with the reference solution as well as the theoretical slopes.

\subsection{Third-order Stokes wave}\label{sec:results:stokes}
We consider the simulation of a third-order Stokes wave~\citep{Deike2015}, which, depending on the initial wave steepness $\epsilon$ and Bond number $\bond$, results in a non-breaking gravity wave, a parasitic-capillary wave, a spilling breaker, or a plunging breaker.
For details about the initial conditions and domain size we refer to~\citet{Deike2015}.
We consider $(\epsilon, \bond) = (0.55, 10^3)$ which yields a plunging breaker, as illustrated in~\cref{fig:results:stokes3:vorticity,fig:results:stokes3:3d} where we show example simulations using our proposed \twofluid model.

Contrary to~\cite{Deike2015} we do not initialise the velocity field in the gas phase, hence we let $u^g = 0$ at $t = 0$.
The density ratio is given by $\ratio{\rho} \approx 1.17 \times 10^{-3}$.
In the notation of~\cite{Deike2015}, we let the Reynolds and Bond number be given by $\reynolds=4 \times 10^4$ and $\bond=10^3$, resulting in `easy to resolve' shear layers ($\delta_\mtext{SL}(\lambda) = 0.1 \lambda$) and interface length scales.
Hence this type of wave is appropriate for testing whether our model will converge to the desired Navier--Stokes solution for a sufficiently fine mesh.
Here time has been nondimensionalised in terms of the wave period $T_g$ as given in~\cref{eqn:results:rt:gravity_timescale}.

Since this is a rather viscous problem, we aim at actually resolving the shear layer, thereby testing whether or not we have succeeded in solving~\cref{chal:refinement}: `how to ensure that the viscous one-velocity solution is obtained under mesh refinement?'
Resolving the shear layer implies that we must now also refine our mesh away from the interface, and to this end we use the vorticity based refinement criterion denoted by $\bar\omega\amr{\level}$-AMR.
The criterion is given by
\begin{equation}
  \varepsilon_\mtext{coarsen}^{\bar\omega} \le \max_{c \in \mathcal{B}} h_c |\bar\omega|_c \le \varepsilon_\mtext{refine}^{\bar\omega},
\end{equation}
where $2\varepsilon_\mtext{coarsen}^{\bar\omega} < \varepsilon_\mtext{refine}^{\bar\omega}$ are dimensional tolerances and $|\bar\omega| \in \mathcal{C}^h$ is an approximation of the vorticity magnitude
\begin{equation}
  |\bar\omega|_c \approx \abs{\curl \bar{\+u}(\+x_c)}_2.
\end{equation}
Note that the vorticity is computed using the equivalent velocity field of the \onefluid formulation as defined in~\cref{eqn:poisson:compare_to_one:onevelo}.

\begin{figure}
  \captionbox{Vorticity (clipped to $\omega_y \in [-100, 100]$) and interface profile on time instances $t / T_g \approx 0.38, 0.52, 0.66, 0.8$.
  The result was obtained using the \twofluid model and $\bar\omega\amr{6}\volfrac\amr{6}$-AMR.
  \label{fig:results:stokes3:vorticity}}
  [\twofigwidth]{
\begingroup%
  \makeatletter%
  \providecommand\color[2][]{%
    \errmessage{(Inkscape) Color is used for the text in Inkscape, but the package 'color.sty' is not loaded}%
    \renewcommand\color[2][]{}%
  }%
  \providecommand\transparent[1]{%
    \errmessage{(Inkscape) Transparency is used (non-zero) for the text in Inkscape, but the package 'transparent.sty' is not loaded}%
    \renewcommand\transparent[1]{}%
  }%
  \providecommand\rotatebox[2]{#2}%
  \newcommand*\fsize{\dimexpr\f@size pt\relax}%
  \newcommand*\lineheight[1]{\fontsize{\fsize}{#1\fsize}\selectfont}%
  \ifx\svgwidth\undefined%
    \setlength{\unitlength}{195.4578147bp}%
    \ifx\svgscale\undefined%
      \relax%
    \else%
      \setlength{\unitlength}{\unitlength * \real{\svgscale}}%
    \fi%
  \else%
    \setlength{\unitlength}{\svgwidth}%
  \fi%
  \global\let\svgwidth\undefined%
  \global\let\svgscale\undefined%
  \makeatother%
  \begin{picture}(1,1.02939871)%
    \lineheight{1}%
    \setlength\tabcolsep{0pt}%
    \put(0,0){\includegraphics[width=\unitlength,page=1]{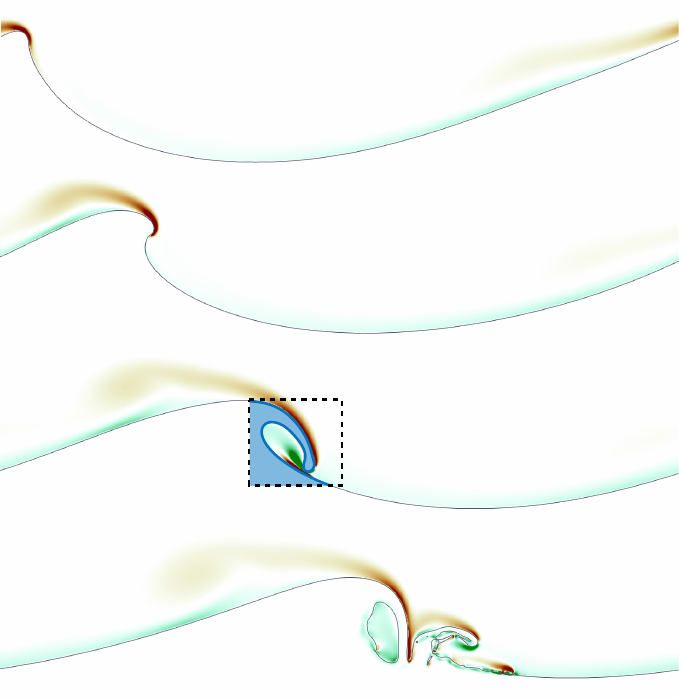}}%
    \put(0.32825654,0.25295794){\color[rgb]{0,0,0.09019608}\makebox(0,0)[lt]{\lineheight{1.25}\smash{\begin{tabular}[t]{l}\Cref{fig:results:stokes3:convergence}\end{tabular}}}}%
    \put(0,0){\includegraphics[width=\unitlength,page=2]{stokes3_vort.pdf}}%
    \put(0.78099551,0.55195338){\color[rgb]{0,0,0}\makebox(0,0)[lt]{\lineheight{1.25}\smash{\begin{tabular}[t]{l}$\+x_*$\end{tabular}}}}%
    \put(0,0){\includegraphics[width=\unitlength,page=3]{stokes3_vort.pdf}}%
    \put(0.33041038,0.70326458){\color[rgb]{0,0,0.09019608}\makebox(0,0)[lt]{\lineheight{1.25}\smash{\begin{tabular}[t]{l}\Cref{fig:results:stokes3:bl}\end{tabular}}}}%
    \put(0,0){\includegraphics[width=\unitlength,page=4]{stokes3_vort.pdf}}%
    \put(0.86142466,0.70433254){\color[rgb]{0.85098039,0.3254902,0.09803922}\makebox(0,0)[lt]{\lineheight{1.25}\smash{\begin{tabular}[t]{l}$w$\end{tabular}}}}%
  \end{picture}%
\endgroup%

  }\hfill
  \captionbox{Example 3D simulation using the \twofluid model with $\volfrac\amr{5}$-AMR on time instances $t / T_g \approx 0.5,  0.7,  0.9$. 
  The colouring corresponds to the magnitude of the liquid velocity clipped to the interval $\abs{\+u^l}_2 \in [0, 0.9]$.\label{fig:results:stokes3:3d}}
  [\twofigwidth]{
    \includegraphics[scale=0.103,trim=0 170 0 250,clip=true]{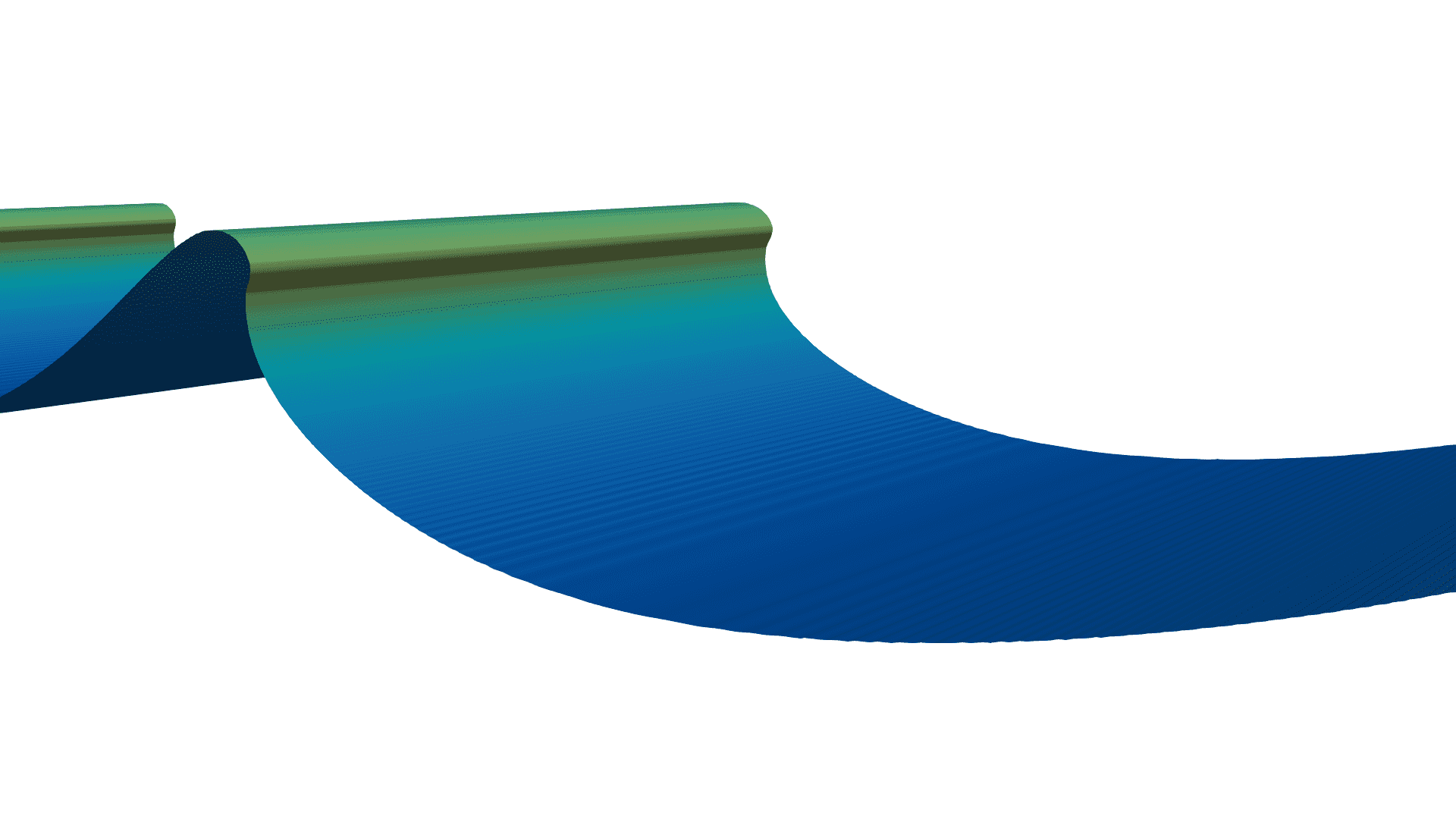}
    \includegraphics[scale=0.103,trim=0 170 0 250,clip=true]{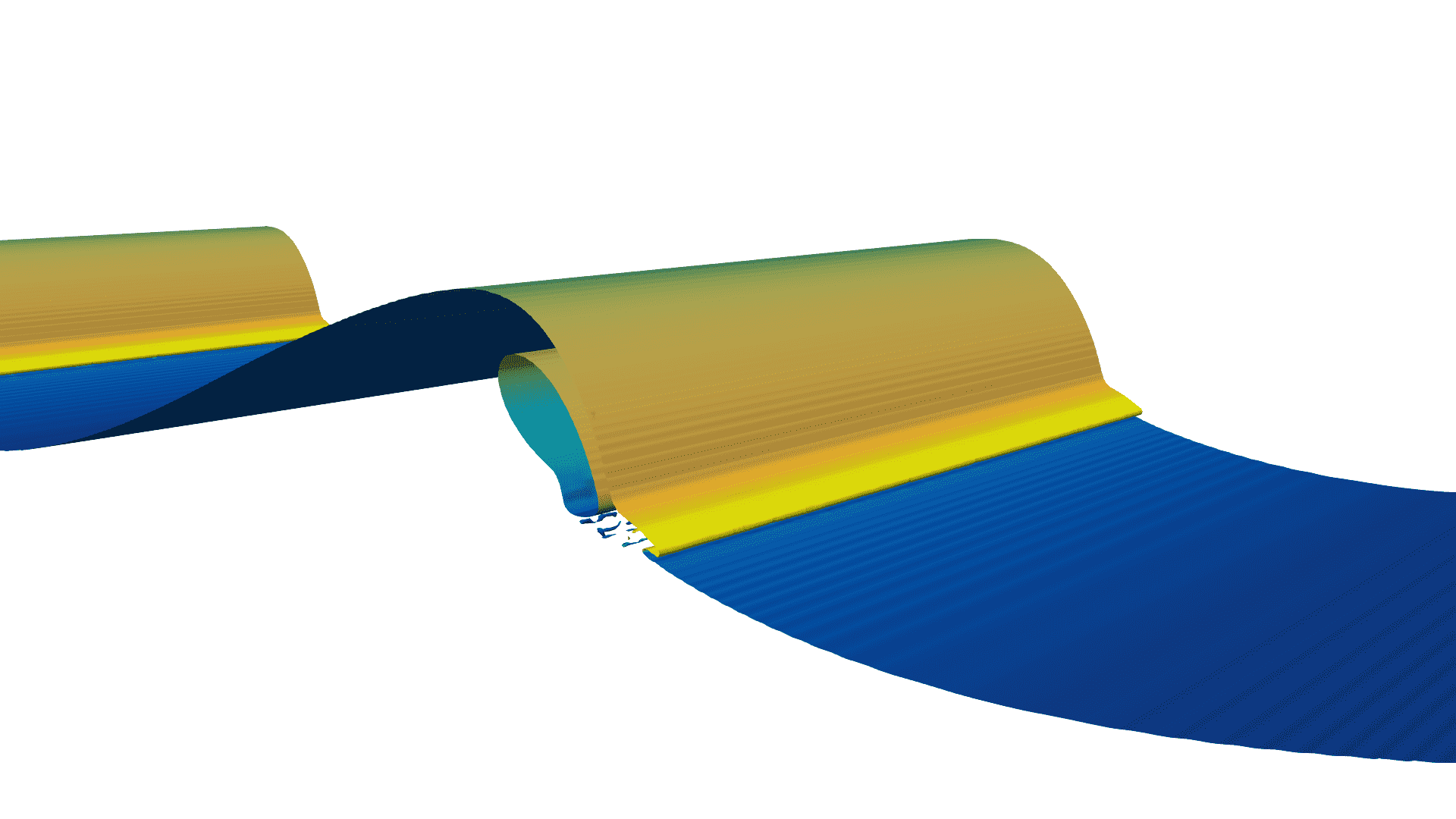}
    \includegraphics[scale=0.103,trim=0 170 0 250,clip=true]{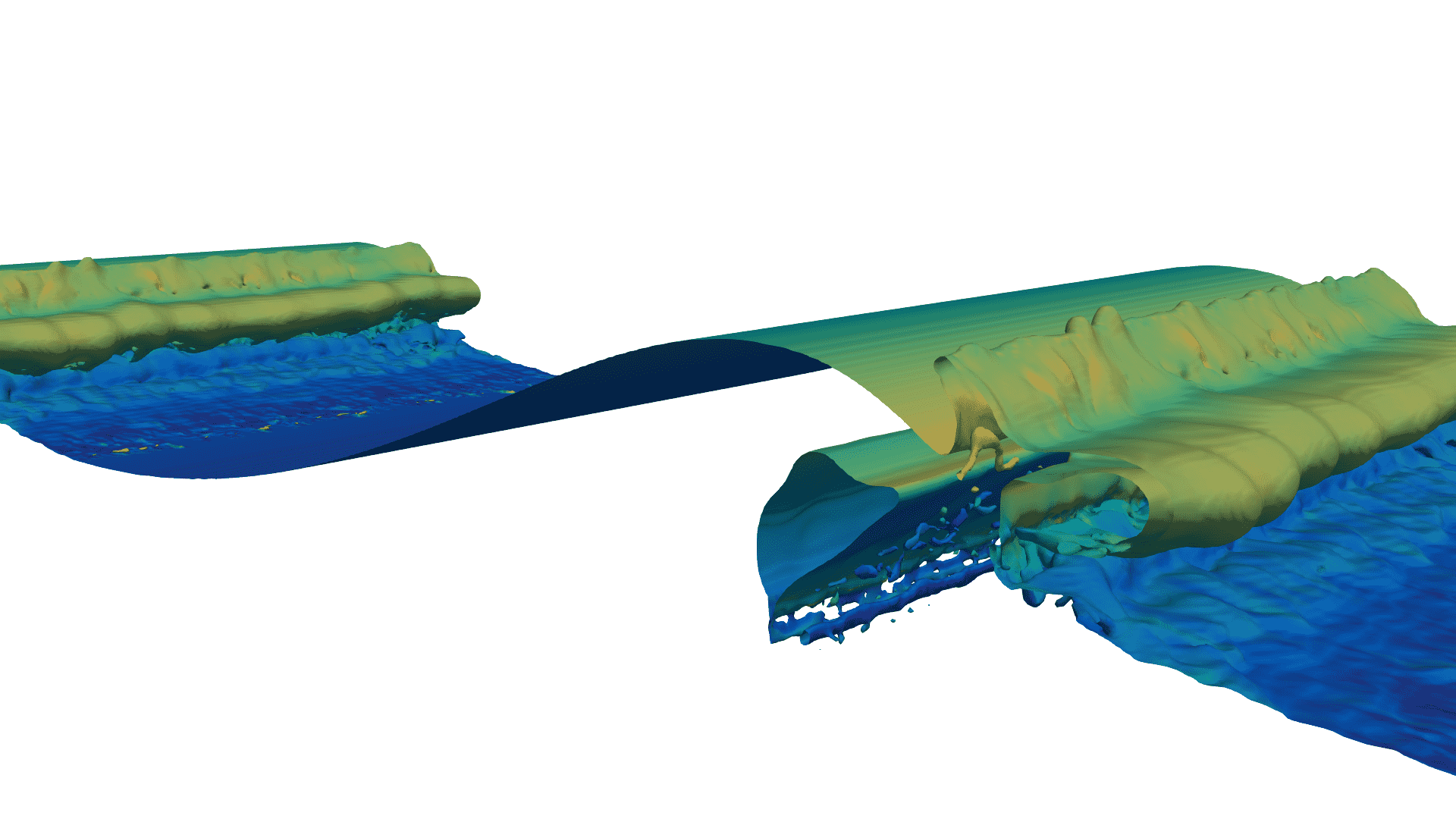}
  }

  \centering
  \inputtikzorpdf{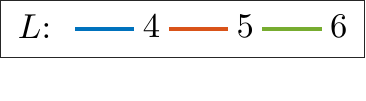}

  \captionbox{Convergence of the interface profile at $t / T_g \approx 0.66$.
  The dashed and solid lines correspond to the one- and \twofluid model respectively.
  The black solid line is the reference solution, which is obtained using $L = 7$.
  \label{fig:results:stokes3:convergence}}
  [\twofigwidth]{
    \def\tikzWidth{\textwidth*0.35}
    \def\tikzHeight{\textwidth*0.35}
    \inputtikzorpdf{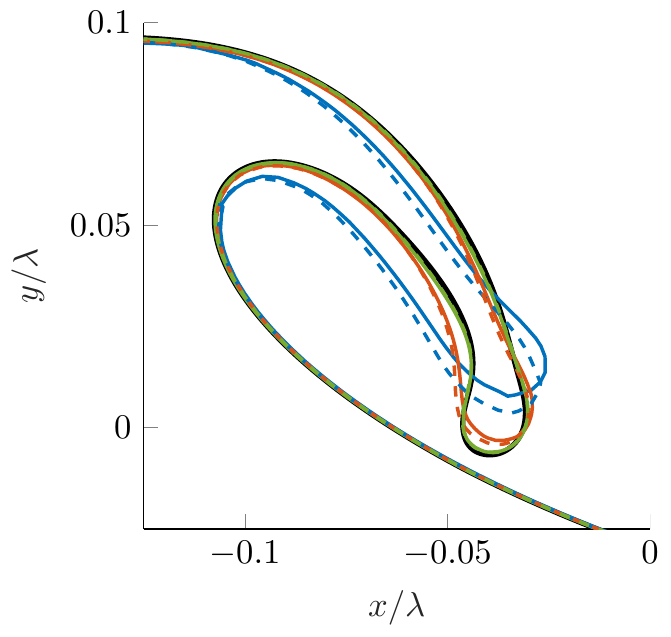}
  }\hfill
  \captionbox{Convergence of interface tangential velocity $w$.
  See also~\cref{fig:results:stokes3:vorticity}.
  Here the blue, red and green lines correspond to $h / \delta_\mtext{SL}(\lambda) \approx 1/25, 1/50, 1/100$.
  The markers indicate the velocity per phase at the interface.\label{fig:results:stokes3:bl}}
  [\twofigwidth]{
    \def\tikzWidth{\textwidth*0.4}
    \def\tikzHeight{\textwidth*0.35}
    \inputtikzorpdf{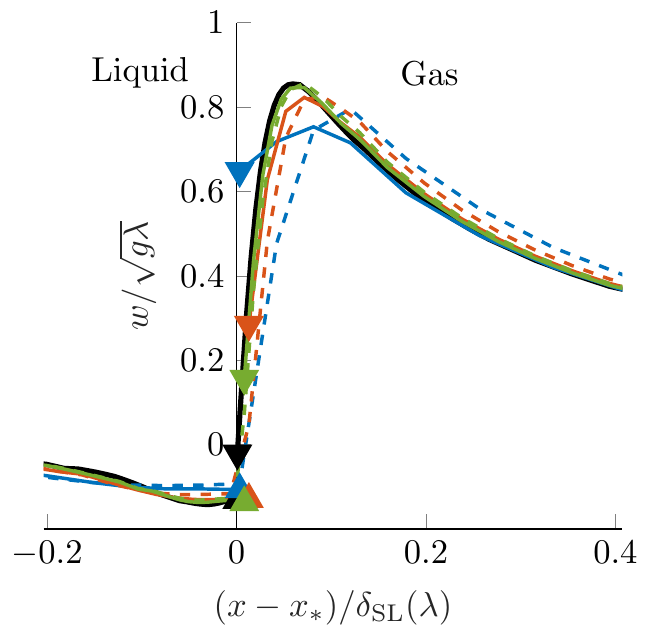}
  }
\end{figure}
We use $\bar\omega\amr{\level}\volfrac\amr{\level}$-AMR with $\basemesh = 2^{-4}\lambda$, resulting in $h = 2^{-(4+\level)}\lambda$ for $\level \in \{4, \ldots, 7\}$.
The solution from the \twofluid model using $\bar\omega\amr{7}\volfrac\amr{7}$-AMR (resulting in $h / \delta_\mtext{SL}(\lambda) \approx 1/200$) is considered as our reference solution, since it essentially overlaps with the corresponding \onefluid solution (in terms of the interface profile as well as velocity profile inside the shear layer).
In~\cref{fig:results:stokes3:convergence} we show the convergence of the interface profile at $t / T_g \approx 0.66$.
We note that not much difference is observed in the interface profiles between the one- and \twofluid formulations.
This makes sense since the shear layer is resolved in this test case, which should indeed result in similar results from the one- and \twofluid models.

In~\cref{fig:results:stokes3:bl} we show the $z$-component of velocity along the line normal to the wave crest (see also the inset in~\cref{fig:results:stokes3:vorticity}, the position is given by $\+x_*/\lambda = [-0.264, 0.093]^T$), at $t/T_g\approx 0.52$, which shows how the velocity field converges under mesh refinement inside the shear layer.
We find that the velocity profiles for both the one- and \twofluid models converge to the same reference velocity profile.
The gas velocity away from the interface seems to converge to the reference solution faster when the \twofluid model is used, which can be expected to be a result of our proposed jump condition which ensures that the shear layer is not artificially thickened when it is not yet resolved.
Moreover we note that the velocity jump resulting from the \twofluid model, as indicated by the triangular markers, tends to zero under mesh refinement, as desired.

In~\cref{fig:results:stokes3:3d} we show an example 3D simulation of the same breaking wave, where we used the \twofluid model with $\volfrac\amr{5}$-AMR.
This illustrates the ability of the proposed \twofluid model to deal with significant topological changes, also in 3D, in the presence of the jump conditions discussed in~\cref{sec:poisson}.

\subsection{Breaking wave impact}\label{sec:results:impact}
\begin{figure}
  \centering
  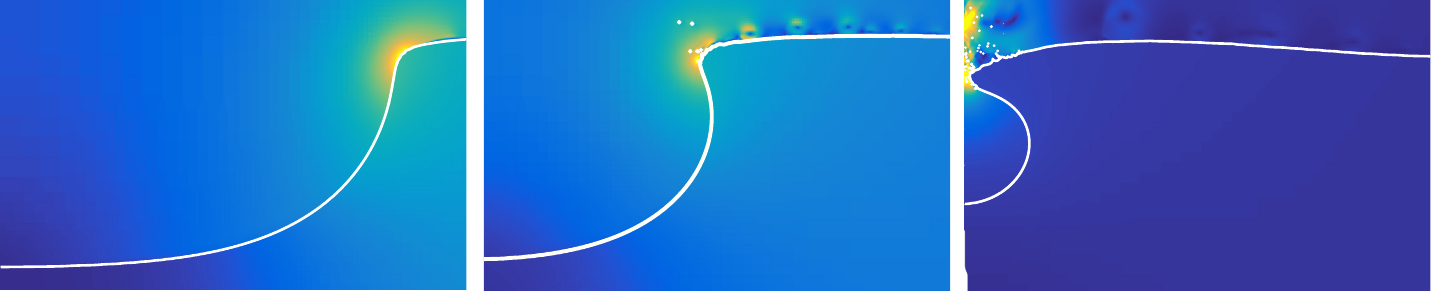
  \caption{Example simulation using the proposed \twofluid model of the large gas-pocket impact from~\citet{Etienne2018} at scale $1 : 5$ (thus resulting in a $\lambda=4$m long wave) and $t /T_g \approx 0.44, 0.51, 0.58$ (from left to right). 
  We show the velocity magnitude, which was clipped to $0 \le \abs{\+u}_2 \le 7, 10, 70$ respectively.}
  \label{fig:results:lgpi:example}
\end{figure}
We finally turn to the example shown in~\cref{fig:results:lgpi:example}.
That is, we consider the simulation of a large gas-pocket impact (LGPI)~\citep{Etienne2018} at scale $s = 5$, resulting in the impact of a wave of length $\lambda = 4$m.
Here we are interested in resolving the most unstable wavelength $\lambda_\mtext{KH}$.
Based on the analysis in~\cref{sec:intro} we do not aim at entirely resolving the shear layer, since this would be too expensive.
We again use $T_g$ as our time scale, which was defined in~\cref{eqn:results:rt:gravity_timescale}.

We will use the~\twofluid model, where we refine the mesh at the interface in a region which moves along with the wave tip.
We refer to this refinement criterion as $\volfrac^{u}\amr{\level}$-AMR, which we will combine with $\volfrac\amr{4}$-AMR, resulting in $\volfrac^u\amr{\level}\volfrac\amr{4}$-AMR (resulting in four levels of refinement at the interface, and $L \ge 4$ levels at the wave tip).
A priori we choose our interface resolution to be sufficiently fine to resolve $\lambda_\mtext{KH}$ for $U_\tau \approx 15$m/s, resulting in $\lambda / \lambda_\mtext{KH} \approx 1307$.
We will use a uniform underlying mesh resulting in $\lambda / \basemesh = 80$, where $\basemesh$ denotes the mesh width of the underlying uniform mesh.
Hence we need $\level = 3+\log_2 \frac{1307}{80} \approx 7$ levels of refinement to have eight gridpoints per unstable wavelength.
Note that $\ceil{40/\sqrt{15}} = 2^4$ (see~\cref{eqn:intro:lengthratio}) and therefore resolving the shear layer with eight gridpoints per shear layer thickness requires at least four additional refinement levels, which would result in $L = 11$.

\begin{figure}
  \captionbox{Approximate solution to the LGPI problem at scale $s = 5$, using $\volfrac^u\amr{7}\volfrac\amr{4}$-AMR.
  The vorticity is shown as well as the interface profile, at $t/T_g \approx 0.45, 0.58$.
  Here we show part of the computational domain $x \in [0, \lambda/3]$.
  The insets show the velocity magnitude, clipped to a maximum of $7$m/s and $70$m/s respectively.\label{fig:results:lgpi:evolution}}
  [\twofigwidth]{
\begingroup%
  \makeatletter%
  \providecommand\color[2][]{%
    \errmessage{(Inkscape) Color is used for the text in Inkscape, but the package 'color.sty' is not loaded}%
    \renewcommand\color[2][]{}%
  }%
  \providecommand\transparent[1]{%
    \errmessage{(Inkscape) Transparency is used (non-zero) for the text in Inkscape, but the package 'transparent.sty' is not loaded}%
    \renewcommand\transparent[1]{}%
  }%
  \providecommand\rotatebox[2]{#2}%
  \newcommand*\fsize{\dimexpr\f@size pt\relax}%
  \newcommand*\lineheight[1]{\fontsize{\fsize}{#1\fsize}\selectfont}%
  \ifx\svgwidth\undefined%
    \setlength{\unitlength}{197.82690203bp}%
    \ifx\svgscale\undefined%
      \relax%
    \else%
      \setlength{\unitlength}{\unitlength * \real{\svgscale}}%
    \fi%
  \else%
    \setlength{\unitlength}{\svgwidth}%
  \fi%
  \global\let\svgwidth\undefined%
  \global\let\svgscale\undefined%
  \makeatother%
  \begin{picture}(1,1.1647459)%
    \lineheight{1}%
    \setlength\tabcolsep{0pt}%
    \put(0,0){\includegraphics[width=\unitlength,page=1]{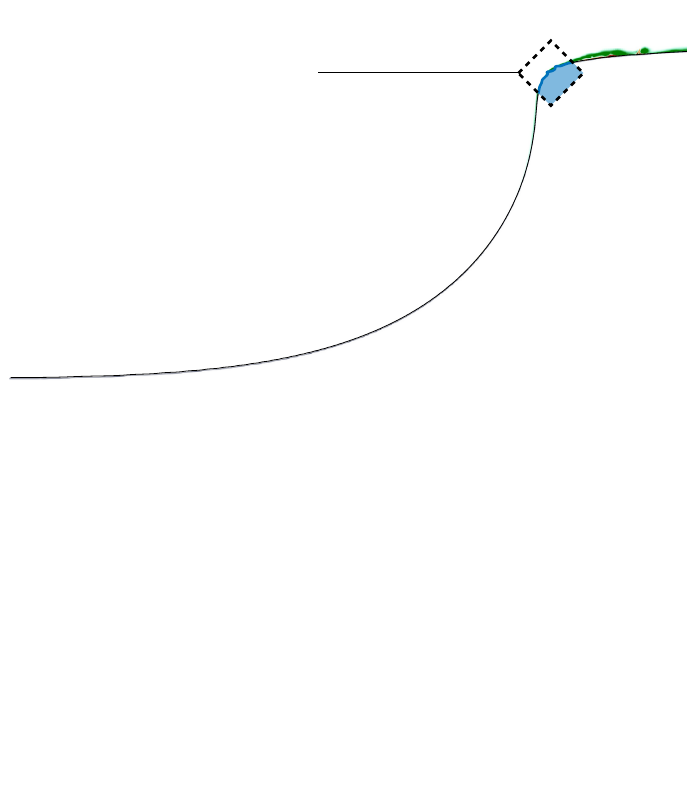}}%
    \put(0.48992335,1.07061313){\color[rgb]{0,0,0.09019608}\makebox(0,0)[lt]{\lineheight{1.25}\smash{\begin{tabular}[t]{l}\Cref{fig:results:lgpi:refinement}\end{tabular}}}}%
    \put(0,0){\includegraphics[width=\unitlength,page=2]{lgpi_global_vort.pdf}}%
    \put(0.24570275,0.8236015){\color[rgb]{0,0,0}\makebox(0,0)[lt]{\lineheight{1.25}\smash{\begin{tabular}[t]{l}$\lambda_\mtext{KH}$\end{tabular}}}}%
    \put(0,0){\includegraphics[width=\unitlength,page=3]{lgpi_global_vort.pdf}}%
  \end{picture}%
\endgroup%

  }\hfill
  \captionbox{Interface profile (top) and velocity jump (bottom) at $t/T_g \approx 0.45$ in rotated and nondimensionalised coordinates (see~\cref{fig:results:lgpi:evolution}).
  The wavelength $\lambda_\mtext{KH}$ was determined using~\cref{eqn:intro:wavelength} with $U_\tau = 6$m/s.
  The mesh was refined according to $\volfrac^u\amr{\level}\volfrac\amr{4}$-AMR for $\level \in \{6, 7\}$.\label{fig:results:lgpi:refinement}}
  [\twofigwidth]{
    \def\tikzWidth{\textwidth*0.33}
    \def\tikzHeighttop{\textwidth*0.19}
    \def\tikzHeightbottom{\textwidth*0.25}
    \inputtikzorpdf{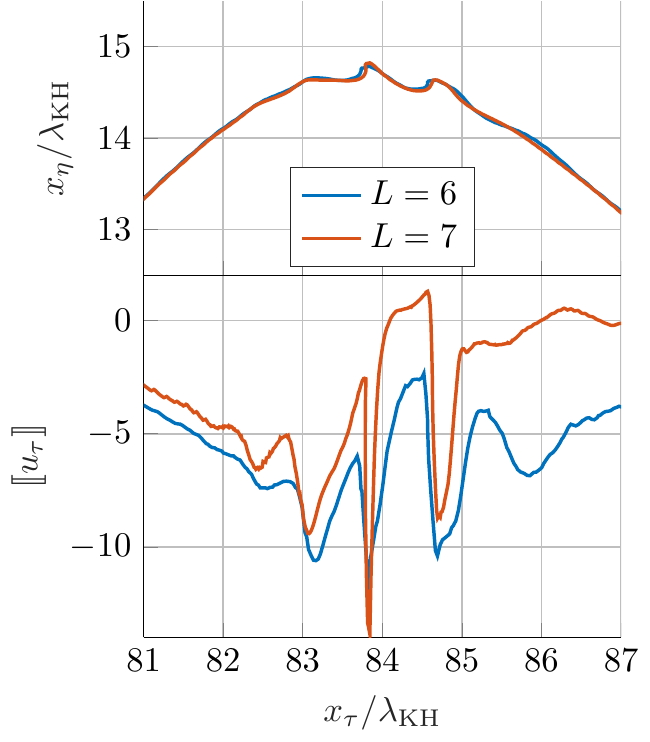}
  }
\end{figure}
In~\cref{fig:results:lgpi:evolution} we show the interface profile at two time instances, together with the vorticity field (clipped to $\omega_y \in [-10^3,10^3]$ and $\omega_y \in [-10^4,10^4]$ respectively).
Furthermore we show a close up of the wave crest with the velocity magnitude.
A KH instability is seen to appear around $t / T_g \approx 0.45$ with a wavelength which agrees well with the indicated theoretical wavelength  according to~\cref{eqn:intro:wavelength}.
This formula depends on the velocity jump (therein referred to as $U_\tau$), for which we have used $U_\tau = 6$m/s based on the velocity jump as shown in~\cref{fig:results:lgpi:refinement}.
In~\cref{fig:results:lgpi:refinement} we furthermore compare two resolutions, $\level \in \{6, 7\}$, resulting in $\lambda_\mtext{KH} / h \approx 24$ and $49$ respectively, which results in a well-resolved instability at $t / T_g \approx 0.45$. 
This is reflected in the fact that the interface profiles essentially overlap.

The instability develops further as shown in the bottom of~\cref{fig:results:lgpi:evolution}. 
Note that the velocities at this point are as large as $70$m/s, resulting in unstable wavelengths of size $\lambda_\mtext{KH} / h \approx 0.36$ for $h = \basemesh/2^7$.
Hence at this point in time the solution is no longer resolved.
Rather than using a fixed level of refinement at the interface, in future work we will explore the possibility of using $\lambda_\mtext{KH} / h$ as a refinement criterion, where $\lambda_\mtext{KH}$ is determined using a local estimate of $U_\tau$, and may also include (possibly stabilising) effects due to gravity.

\subsection{Computational cost}
\begin{figure}
  \centering
  \def\tikzWidth{\textwidth*0.5}
  \def\tikzHeight{\textwidth*0.35}
  \inputtikzorpdf{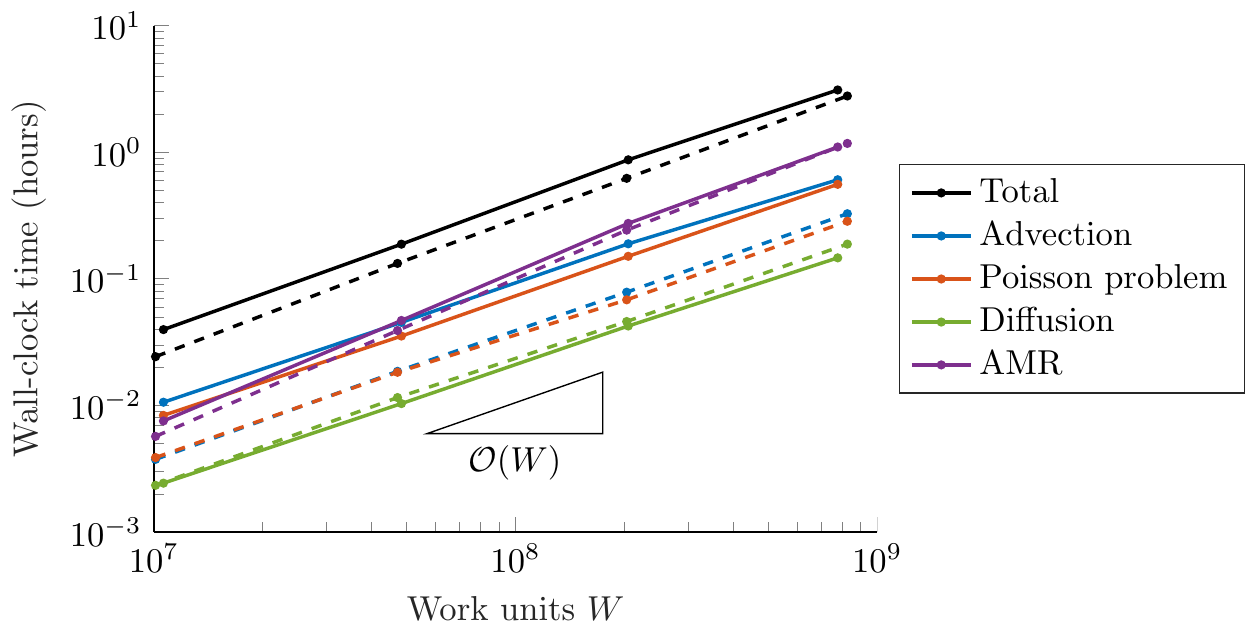}
  \caption{Wall-clock time (using OpenMP with 12 cores from an Intel Xeon E5-2680 v3 2.5 GHz processor) as function of the number of work units $W$ (sum over all time steps of the number of control volumes).
  Here advection refers to the advection of the interface as well as the transport of momentum, and includes the additional projection step required in the \twofluid model.
  Each marker corresponds to a simulation of the Stokes-3 problem from~\cref{sec:results:stokes} with refinement level $\level = 4, \ldots, 7$ for the \onefluid (dashed) and \twofluid (solid) model.\label{fig:results:cost}}
\end{figure}
For high-Reynolds number flow, with thin shear layers, the \twofluid model allows for the use of a coarser mesh, but the efficiency of the resulting method is determined also by its computational cost.
To this end we now compare the wall-clock times resulting from the eight Stokes-3 simulations performed in~\cref{sec:results:stokes} (four different levels of refinement for both the one- and \twofluid model).
In~\cref{fig:results:cost} the wall-clock time is shown as function of the number of `work units' $W$, which we define as the total number of control volumes considered during the simulation
\begin{equation}
  W = \sum_{n=0}^{N_t} |\mathcal{C}|^{(n)},
\end{equation}
where $N_t$ is the number of time steps and $|\mathcal{C}|^{(n)}$ is the number of control volumes at $t=t^{(n)}$.
As expected, the wall-clock time scales linearly with the number of work units. 
The total cost of using the \twofluid model is about $40\%$ larger when compared to the \onefluid model (for a given number of work units), which is acceptable considering the fact that many levels of mesh refinement can be omitted when using the \twofluid model, resulting in much fewer work units.

We note that the current implementation solves for all values of the jump $\xi \in \mathcal{F}^h$, which means that the computational cost may still be reduced by not solving for those jumps which are known to be zero (away from the interface).

  \section{Discussion}\label{sec:conclusion}
We have motivated the use of, and proposed a discretisation of, a sharp \twofluid model for two-phase flow.
The numerical model, which imposes continuity of only the interface normal component of velocity, is of particular interest in the simulation of high density ratio and convection dominated two-phase flow, such as the simulation of breaking wave impacts.
Such flow problems result in thin shear layers, where viscous effects dominate, and which are infeasible to resolve entirely.
In a previous paper~\citep{pubs_Remmerswaal2022a} we discussed the discretisation of the equations governing conservation of mass and momentum, whereas in this paper we considered the implicit coupling of the two fluids via the pressure Poisson equation as well as diffusion.

We have proposed a generalisation of the GFM, referred to as the MDGFM, which allows for the consistent imposition of the \emph{interface normal} component of the pressure gradient, rather than the \emph{face normal} component which is imposed in the GFM.
The MDGFM relies on the definition of a jump interpolant which yields the vector valued jump of our staggered velocity field.
Naive construction of such a jump interpolant results in unboundedness of the resulting interpolation coefficients, which we have resolved by a novel interpolation between two approximations of the face normal component of the jump.
The resulting gradient operator is combined with a cut-cell approximation of the divergence operator, resulting in the convergent approximation of the pressure Poisson problem, for a large range of density ratios.

Furthermore, we have introduced a simple model for diffusion in the presence of a velocity discontinuity.
In particular, the proposed model, which is implicit in time, is defined such that it coincides exactly with the diffusion model for the \onefluid formulation in case the velocity is continuous across the interface.
This approach aims at obtaining the \onefluid formulation when the viscous shear layer is resolved.

Also, a mimetic gravity model was proposed which is shown to be well-balanced, mimics the correct transfer between kinetic and gravitational potential energy, and moreover yields the correct contribution to linear momentum.


We have illustrated the efficacy of the proposed \twofluid formulation by directly comparing it to our \onefluid model.
It has been demonstrated that the well-balancedness of the model is hardly affected by introducing the new jump condition.
Moreover, we have illustrated convergence under mesh refinement of viscous as well as inviscid wave problems, in the presence of surface tension, inertial forces and/or gravity, in both the linear and nonlinear regime.
As can be expected, our \twofluid model performs better than the \onefluid model on inviscid problems, since the latter results in a numerical shear layer, which results in visible oscillations of the velocity, and in some cases results in a lack of convergence.

Topological changes are trivially dealt with, also in 3D, as is demonstrated by the simulation of the Plateau--Rayleigh instability as well as the 3D third-order Stokes wave impact.
The former test problem moreover illustrates the efficiency of the curvature based refinement criterion.

Finally we considered the simulation of breaking waves.
For the third-order Stokes wave we found that, due to the low Reynolds number, the shear layer was easy to resolve, resulting in very similar results when comparing the one- and \twofluid formulations.
We did find that the velocity profile inside the shear layer converges faster for the \twofluid model: a coarser mesh can be used with the \twofluid model.
The simulation of the large gas-pocket impact shows that the \twofluid formulation can effectively be used to model free surface instabilities on breaking waves, without the need for resolving the thin shear layer present at the interface.

In future work it would be interesting to further develop a subgrid model for diffusive stresses at the interface.
The \twofluid model will be used to study the effect of scaling on the development of free surface instabilities in 3D breaking wave impacts.
Therein we want to use an AMR criterion based on resolving the most unstable wavelength.

  \section*{Acknowledgements}
  This work is part of the research programme SLING, which is (partly) financed by the Netherlands Organisation for Scientific Research (NWO).
  We would like to thank the Center for Information Technology of the University of Groningen for their support and for providing access to the Peregrine high performance computing cluster.
  Moreover we thank Dr. Joaquín López (Universidad Politécnica de Cartagena) for kindly providing the VoFTools library.

  \appendix
  \setcounter{figure}{0}
  \section{The ghost fluid method}\label{sec:app:gfm}
We briefly summarise the derivation of the GFM as proposed by~\citet{Liu2000}.
Consider the example shown in~\cref{fig:poisson:gfm:distance_fun} with prescribed value and gradient jump, denoted by $\zeta_f$ and $\xi_f$ respectively (see~\cref{eqn:poisson:gfm:jumps}).
We denote by $\gfmgrad^\pi: \mathcal{C}^h \times \mathcal{F}^h \rightarrow \mathcal{F}^h$ the affine (linear up to a constant) function which approximates the scaled (by the density) pressure gradient
\begin{equation}
  \gfmgrad_f^\pi(p,\xi) \approx \frac{1}{\rho^\pi}\+n_f \cdot (\gradient p^\pi)(\+x_f).
\end{equation}
At the interface position $\stagger{\+x}^I_f$, which we approximate using the sharp interface reconstruction, we impose~\cref{eqn:poisson:gfm:jumps} at $\mathcal{O}(h_f)$ accuracy, resulting in
\begin{subequations}
  \begin{align}
    \roundpar{p_{c_2} - \stagger\phi^g_f h_f \rho^g \gfmgrad^g_f(p,\xi)} - \roundpar{p_{c_1} + \stagger\phi^l_f h_f \rho^l\gfmgrad^l_f(p,\xi)} &= \zeta_f\label{eqn:st:gfm:conditions_1}\\
    \gfmgrad^g_f(p,\xi) - \gfmgrad^l_f(p,\xi) &= \xi_f,\label{eqn:st:gfm:conditions_2}
  \end{align}
\end{subequations}
where~\cref{eqn:st:gfm:conditions_1}, denotes the pressure jump at $y = \stagger{y}^I_f$ in~\cref{fig:poisson:gfm:distance_fun}, and results from linear pressure extrapolation to the interface.

Solving~\cref{eqn:st:gfm:conditions_1,eqn:st:gfm:conditions_2} for $\gfmgrad^\pi_f(p,\xi)$ results in
\begin{equation}
  \gfmgrad^g_f(p,\xi) = \frac{1}{\mean{\stagger\phi_f\rho}} \squarepar{\frac{p_{c_2} - p_{c_1} - \zeta_f}{h_f} + \stagger\phi^l_f\rho^l \xi_f}, \quad \gfmgrad^l_f(p,\xi) = \frac{1}{\mean{\stagger\phi_f\rho}} \squarepar{\frac{p_{c_2} - p_{c_1} - \zeta_f}{h_f} - \stagger\phi^g_f\rho^g \xi_f},
\end{equation}
which can be generalised, by making use of the liquid indicator $\chi^l_{c} \in \{0, 1\}$ as well as the gradient operator as defined in~\cref{eqn:notation:gradient}, to~\cref{eqn:poisson:mdgfm:gfm,eqn:poisson:gfm:rtilde}.

\section{The mimetic gravity model}\label{sec:app:gravity}
\lemmagravitybalance*
\begin{proof}
  We will first show that the following equation holds
  \begin{equation}\label{eqn:app:gravity:reduced_equality}
    (\gradh(\chi^\pi q(\+x^I)))_f = a_f^\pi (\gradh q(\+x^I))_f.
  \end{equation}
  Recall that $\+x^I_c$ equals the interface centroid if $c$ contains part of the interface (and therefore $q(\+x^I_c) = 0$), and coincides with the control volume centroid $\+x_c$ if $c$ does not contain part of the interface (in which case $q(\+x^I_c) \neq 0$).
  Consider a face $f$ with $\mathcal{C}(f) = \{c_1, c_2\}$ and let $\pi$ be either $l$ or $g$, we find that one of the following conditions must hold.
  \begin{figure}
    \subcaptionbox{The case where $c_2$ is completely empty.}
    [\twofigwidth]{
      \centering
\begingroup%
  \makeatletter%
  \providecommand\color[2][]{%
    \errmessage{(Inkscape) Color is used for the text in Inkscape, but the package 'color.sty' is not loaded}%
    \renewcommand\color[2][]{}%
  }%
  \providecommand\transparent[1]{%
    \errmessage{(Inkscape) Transparency is used (non-zero) for the text in Inkscape, but the package 'transparent.sty' is not loaded}%
    \renewcommand\transparent[1]{}%
  }%
  \providecommand\rotatebox[2]{#2}%
  \newcommand*\fsize{\dimexpr\f@size pt\relax}%
  \newcommand*\lineheight[1]{\fontsize{\fsize}{#1\fsize}\selectfont}%
  \ifx\svgwidth\undefined%
    \setlength{\unitlength}{132.61362811bp}%
    \ifx\svgscale\undefined%
      \relax%
    \else%
      \setlength{\unitlength}{\unitlength * \real{\svgscale}}%
    \fi%
  \else%
    \setlength{\unitlength}{\svgwidth}%
  \fi%
  \global\let\svgwidth\undefined%
  \global\let\svgscale\undefined%
  \makeatother%
  \begin{picture}(1,0.5023933)%
    \lineheight{1}%
    \setlength\tabcolsep{0pt}%
    \put(0,0){\includegraphics[width=\unitlength,page=1]{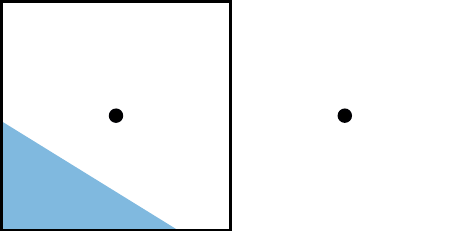}}%
    \put(0.26998437,0.29331206){\color[rgb]{0,0,0}\makebox(0,0)[lt]{\lineheight{1.25}\smash{\begin{tabular}[t]{l}$\+x_{c_1}$\end{tabular}}}}%
    \put(0,0){\includegraphics[width=\unitlength,page=2]{gravity_mgm_proof.pdf}}%
    \put(0.22486617,0.12561421){\color[rgb]{0,0,0}\makebox(0,0)[lt]{\lineheight{1.25}\smash{\begin{tabular}[t]{l}$\+x^I_{c_1}$\end{tabular}}}}%
    \put(0.51022404,0.23432007){\color[rgb]{0,0,0}\makebox(0,0)[lt]{\lineheight{1.25}\smash{\begin{tabular}[t]{l}$f$\end{tabular}}}}%
    \put(0.77696192,0.28927241){\color[rgb]{0,0,0}\makebox(0,0)[lt]{\lineheight{1.25}\smash{\begin{tabular}[t]{l}$\+x_{c_2}$\end{tabular}}}}%
  \end{picture}%
\endgroup%

    }\hfill
    \subcaptionbox{The case where $c_2$ is completely filled.}
    [\twofigwidth]{
      \centering
\begingroup%
  \makeatletter%
  \providecommand\color[2][]{%
    \errmessage{(Inkscape) Color is used for the text in Inkscape, but the package 'color.sty' is not loaded}%
    \renewcommand\color[2][]{}%
  }%
  \providecommand\transparent[1]{%
    \errmessage{(Inkscape) Transparency is used (non-zero) for the text in Inkscape, but the package 'transparent.sty' is not loaded}%
    \renewcommand\transparent[1]{}%
  }%
  \providecommand\rotatebox[2]{#2}%
  \newcommand*\fsize{\dimexpr\f@size pt\relax}%
  \newcommand*\lineheight[1]{\fontsize{\fsize}{#1\fsize}\selectfont}%
  \ifx\svgwidth\undefined%
    \setlength{\unitlength}{132.51707523bp}%
    \ifx\svgscale\undefined%
      \relax%
    \else%
      \setlength{\unitlength}{\unitlength * \real{\svgscale}}%
    \fi%
  \else%
    \setlength{\unitlength}{\svgwidth}%
  \fi%
  \global\let\svgwidth\undefined%
  \global\let\svgscale\undefined%
  \makeatother%
  \begin{picture}(1,0.50275935)%
    \lineheight{1}%
    \setlength\tabcolsep{0pt}%
    \put(0,0){\includegraphics[width=\unitlength,page=1]{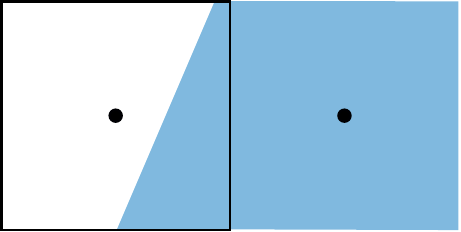}}%
    \put(0.19163226,0.29655743){\color[rgb]{0,0,0}\makebox(0,0)[lt]{\lineheight{1.25}\smash{\begin{tabular}[t]{l}$\+x_{c_1}$\end{tabular}}}}%
    \put(0,0){\includegraphics[width=\unitlength,page=2]{gravity_mgm_proof_filled.pdf}}%
    \put(0.37084562,0.17421691){\color[rgb]{0,0,0}\makebox(0,0)[lt]{\lineheight{1.25}\smash{\begin{tabular}[t]{l}$\+x^I_{c_1}$\end{tabular}}}}%
    \put(0.50986718,0.23449064){\color[rgb]{0,0,0}\makebox(0,0)[lt]{\lineheight{1.25}\smash{\begin{tabular}[t]{l}$f$\end{tabular}}}}%
    \put(0.77679942,0.2894826){\color[rgb]{0,0,0}\makebox(0,0)[lt]{\lineheight{1.25}\smash{\begin{tabular}[t]{l}$\+x_{c_2}$\end{tabular}}}}%
  \end{picture}%
\endgroup%

    }
    \caption{Illustration of the situation where exactly one of the control volumes ($c_1$) contains part of the interface.}\label{fig:app:gravity:proof}
  \end{figure}
  \begin{enumerate}
    \item Both control volumes contain part of the interface: $q(\+x^I_{c_1}) = 0 = q(\+x^I_{c_2})$, for this case we obviously have equality in~\cref{eqn:app:gravity:reduced_equality}.
    \item Neither of the control volumes contain part of the interface: $q(\+x^I_{c_1}) \neq 0 \neq q(\+x^I_{c_2})$, since for non-interface control volumes the `interface centroid' was defined as the control volume centroid instead.
    It follows that $a_f^\pi, \chi_{c_1}^\pi, \chi_{c_2}^\pi$ are either all equal to zero, or all equal to one, and therefore the equality in~\cref{eqn:app:gravity:reduced_equality} holds.
    \item Exactly one of the control volumes contains part of the interface: without loss of generality we assume that ${c_1}$ contains part of the interface, and $c_2$ does not. 
    Hence $q(\+x^I_{c_1}) = 0$ and $q(\+x^I_{c_2}) \neq 0$. 
    Then $a^\pi_f \in \{0, 1\}$, and since $c_2$ is entirely full or empty we find that $\chi^\pi_{c_2} = a^\pi_f$.
    See also~\cref{fig:app:gravity:proof}.
    It follows that
    \begin{align}
      -h_f \gradh(\chi^\pi q(\+x^I))_f 
        &\stackrelwidth{\eqref{eqn:notation:gradient}}{=} \orientation_{{c_1},f} \chi_{c_1}^\pi q(\+x^I_{c_1}) + \orientation_{c_2,f} \chi_{c_2}^\pi q(\+x^I_{c_2})\\
        &= a^\pi_f \orientation_{c_2,f} q(\+x^I_{c_2})\\
        &= a^\pi_f \squarepar{\orientation_{{c_1},f} q(\+x^I_{c_1}) + \orientation_{c_2,f} q(\+x^I_{c_2})}\\
        &\stackrelwidth{\eqref{eqn:notation:gradient}}{=} -h_fa_f^\pi \gradh( q(\+x^I))_f,\\
    \end{align}
    and therefore also in this case the equality given by~\cref{eqn:app:gravity:reduced_equality} holds.
  \end{enumerate}
  
  When multiplying~\cref{eqn:app:gravity:reduced_equality} by $\rho^\pi \abs{\+g}_2$ (which is constant and may therefore be moved into the gradient operator) and summing over the phases, we find that
  \begin{equation}\label{eqn:app:gravity:intermediate1}
    \gradh(\mean{\chi \rho} \abs{\+g}_2 q(\+x^I)) = \mean{a \rho} \gradh(\abs{\+g}_2 q(\+x^I)),
  \end{equation}
  where by definition of the level set function in~\cref{eqn:gravity:levelset} it follows that
  \begin{equation}\label{eqn:app:gravity:intermediate2}
    \gradh(\abs{\+g}_2 q(\+x^I)) = \gradh(\abs{\+g}_2 \+N \cdot \+x^I) - \gradh(\abs{\+g}_2S) = - \gradh(\+g \cdot \+x^I),
  \end{equation}
  where we made use of $\abs{\+g}_2\+N = -\+g$ and the fact that the gradient operator applied to a constant scalar vanishes: $\gradh 1 = 0$.
  Combining~\cref{eqn:app:gravity:intermediate1,eqn:app:gravity:intermediate2,eqn:gravity:model} then yields~\caref{eqn:app:gravity:steady_state}, where the steady state pressure was given by~\caref{eqn:app:gravity:mgm_sol}.
\end{proof}

The gravity force is nonconservative, but we can still check whether the discrete contribution to linear momentum mimics the analytical counterpart.
Analytically we find that the contribution to linear momentum due to gravity is given by
\begin{equation}\label{eqn:app:gravity:lin_mom_cont}
  \integral{\Omega}{\+F}{V} = \integral{\Omega}{\gradient\roundpar{\rho \+g \cdot \+x}}{V} = -\jump{\rho} \integral{I}{\+\eta (\+g \cdot \+x)}{S},
\end{equation}
which follows from the application of Gauss's divergence theorem on each of the phase domains $\Omega^{\pi}$.
Hence the contribution to linear momentum is some integral over the phase interface $I$.

Numerically we can approximate the interface area vector in the following way
\begin{equation}\label{eqn:st:apert:gauss}
  \+0 = \integral{c^l}{\gradient 1}{V} = \integral{\partial c^l}{\+\eta}{V},
\end{equation}
which follows from Gauss's divergence theorem.
Splitting the boundary integral on the right-hand side of~\cref{eqn:st:apert:gauss} into an integral over the interface $I_c$ and an integral over all but the interface $\partial c^l \setminus I_c$, results in the former being able to be expressed in terms of the latter
\begin{equation}\label{eqn:st:apert:apertnormal}
  |I_c|\apertnormal_c = -|c|\divh(a^l\+n)_c,
\end{equation}
which shows that the interface area vector $|I_c|\apertnormal_c$ can be expressed in terms of the same fluid apertures that are present in the gravity model.
Using~\cref{eqn:st:apert:apertnormal} allows us to express the change in linear momentum due to the MGM as
\begin{equation}\label{eqn:app:gravity:lin_mom_disc}
  \sum_\setextrusion{F} \+n F \stackrel{\eqref{eqn:gravity:model}}{=} \sum_\setextrusion{F} \+n \mean{a \rho} \gradh(\+g \cdot\+x^I) = \sum_\mathcal{C} \jump{\rho} \divh(\+n a^l) \+g \cdot\+x^I \stackrel{\eqref{eqn:st:apert:apertnormal}}{=} -\jump{\rho} \sum_{c \in \mathcal{C}} |I_c| \apertnormal_{c} (\+g \cdot\+x^I_c),
\end{equation}
where in the second equality we again made use of the adjointness relation between the divergence and gradient operators as given by~\cref{eqn:notation:adjointness}, and moreover made use of 
\begin{equation}
  \mean{a \rho} = a^l\rho^l + a^g\rho^g = a^l (\rho^l - \rho^g) + \rho^g = -\jump{\rho} a^l + \rho^g.
\end{equation}
Comparison of~\cref{eqn:app:gravity:lin_mom_cont,eqn:app:gravity:lin_mom_disc} shows that the contribution to linear momentum is correctly modelled by the MGM.

  \bibliography{library,pubs}
\end{document}